\documentclass[twosides,10pt]{amsart}
\usepackage{graphics}
\usepackage[english]{babel}
\usepackage{amssymb,amsmath,amsfonts}
\usepackage{datetime}
\usepackage{color}
\usepackage{ulem}
\usepackage{lipsum}
\usepackage{amsfonts}
\usepackage{cite}
\RequirePackage{amscd}
\RequirePackage{epic}
\RequirePackage[all]{xy}
\RequirePackage{url}
\RequirePackage[shortlabels]{enumitem}
\usepackage{empheq}
\usepackage{epsfig}
\usepackage[english]{babel}

\usepackage[utf8]{inputenc}
\usepackage{cite}
\RequirePackage{amscd}
\RequirePackage{epic}
\RequirePackage{eepic}
\RequirePackage[all]{xy}
\RequirePackage{url}
\RequirePackage[shortlabels]{enumitem}
\usepackage{empheq}
\usepackage{nomencl}
\usepackage{epsfig}
\usepackage{graphicx}

\newcommand{\comment}[1]{}
   
    \newcommand{\set}[1]{{\left\{#1\right\}}}
\newcommand{\pa}[1]{{\left(#1\right)}}
\newcommand{\sq}[1]{{\left[#1\right]}}
\newcommand{\gen}[1]{{\left\langle #1\right\rangle}}
\newcommand{\abs}[1]{{\left|#1\right|}}
\newcommand{\norm}[1]{{\left\|#1\right\|}}
\newcommand{\normal}{\mathrm}
\newcommand{\G}{\mathcal{G}}
\newcommand{\A}{\mathcal{A}}
\newcommand{\U}{\mathcal{U}}
\newcommand{\V}{\mathcal{V}}

\newcommand{\T}{\mathbb{T}}
\newcommand{\Z}{\mathbb{Z}}
\newcommand{\R}{\mathbb{R}}
\newcommand{\C}{\mathbb{C}}
\newcommand{\teta}{\theta}
\newcommand{\eps}{\varepsilon}
\renewcommand{\Re}{\operatorname{Re}}
\renewcommand{\Im}{\operatorname{Im}}
\newcommand{\sgn}{\operatorname{sgn}}

\newcommand{\Mat}{\operatorname{Mat}}
\newcommand{\GL}{\operatorname{GL}}

\newcommand{\pmat}[1]{{\begin{pmatrix}#1\end{pmatrix}}}
\newcommand{\amat}[2]{{\begin{array}{#1}#2\end{array}}}
\newcommand{\eqsys}[1]{{\left\{\begin{array}{ll}#1\end{array}\right.}}

\newcommand{\Ham}{\operatorname{Ham}}
\newcommand{\h}{\operatorname{H}}
\newcommand{\average}[1]{\int_{\T^n}#1}
\newcommand{\media}[1]{\int_\T #1}
\newcommand{\Uh}{\mathcal{U}^{\operatorname{Ham}}}
\newcommand{\Vh}{\mathcal{V}^{\operatorname{Ham}}}

\newcommand{\tensor}{\otimes}% prodotto tensore
\newcommand{\inderiv}[1]{#1'^{-1}}    
\newcommand{\pull}[1]{#1^\ast} %pull back
\newcommand{\push}[1]{#1_\ast}%push forward
\newcommand{\base}[1]{{\frac{\partial}{\partial{#1}}}} %base campo vettoriale

\newcommand{\ph}[1]{\frac{\partial H}{\partial #1}}
\newcommand{\co}[1]{\textit{#1}}%corsivo
\newcommand{\gr}[1]{\textbf{#1}}%grassetto

\newcommand{\upp}[1]{\uppercase{#1}}
\newcommand{\id}{\operatorname{id}}

\usepackage{amsthm}
\RequirePackage{url}
\newtheorem{prop}{Proposition}[section]
    \newtheorem{thm}{Theorem}
    \newtheorem*{thm*}{Theorem}
    \newtheorem*{cor*}{Corollary}

    \newtheorem{cor}{Corollary}
    \newtheorem{lemma}{Lemma}
    \theoremstyle{remark}
     
\newtheorem{rmk}{Remark}[section]
\theoremstyle{definition}
\newtheorem{defn}{Definition}

\numberwithin{equation}{section}
\numberwithin{thm}{section}
\numberwithin{defn}{section}
\numberwithin{prop}{section}
\numberwithin{cor}{section}
\numberwithin{rmk}{section}
\newcounter{tmp}
\begin{document}
\author{Jessica Elisa Massetti}\address{Universit\'e  Paris-Dauphine, CEREMADE, Place du Mar\'echal de Lattre de Tassigny, 75775 Paris Cedex 16, France\\ 
\& Astronomie et Systèmes Dynamiques, IMCCE (UMR 8028) - Observatoire de Paris, 77 Av. Denfert Rochereau 75014 Paris, France  \\
              Tel.: +33 140512024\\
              Fax: +33 140512055
              }
 
 \title[Normal forms \& Diophantine tori]{Normal forms for perturbations of systems possessing a Diophantine invariant torus
}
\email{jessica.massetti@obspm.fr}
\date{May, 17th 2016}

%pdflatex -interaction=nonstopmode %.tex
%\author{Jessica Elisa Massetti}
%\authorunning{Jessica Elisa Massetti}
%\institute{ \at Astronomy and Dynamical Systems, IMCCE (UMR 8028) - Observatoire de Paris, 77 Av. Denfert Rochereau 75014 Paris, France  \\
 %             Tel.: +33 140512024\\
  %            Fax: +33 140512055\\ 
   %           \&\\
    %          Universit\'e  Paris-Dauphine, Place du Mar\'echal de Lattre de Tassigny, 75016 Paris, France\\
     %         \email{jessica.massetti@obspm.fr}                   
      %    }
%%%%%%%%%%%%%%%%%%%%%%%%%%%%%%%%%%%%%%%%% for svjour %%%%%%%%%%%%%%%%%%%%%%%%%%%%%%%%%
%\titlerunning{Moser's and others normal form }
\comment{\author{Jessica Elisa Massetti\inst{1}\fnmsep\inst{2}}
\institute{Astronomy and Dynamical Systems, IMCCE (UMR 8028) - Observatoire de Paris, 77 Av. Denfert Rochereau 75014 Paris, France  \\
              Tel.: +33 140512024\\
              Fax: +33 140512055\\ 
              \and Universit\'e  Paris-Dauphine, CEREMADE, Place du Mar\'echal de Lattre de Tassigny, 75775 Paris Cedex 16, France\\
              \email: {jessica.massetti@obspm.fr}    }}
%%%%%%%%%%%%%%%%%%%%%%%%%%%%%%%%%%%%%%%%%%%%%%%%%%%%%%%%%%%%%%%%%%%%%%%%%%%%%%%%%%%%%%%%%%%

\begin{abstract}
In $1967$ Moser proved the existence of a normal form for real analytic perturbations of vector fields possessing a reducible Diophantine invariant quasi-periodic torus. In this paper we present a proof of existence of this normal form based on an abstract inverse function theorem in analytic class. The given geometrization of the proof can be opportunely adapted accordingly to the specificity of systems under study. In this more conceptual frame, it becomes natural to show the existence of new remarkable normal forms, and provide several translated-torus theorems or twisted-torus theorems for systems issued from dissipative generalizations of Hamiltonian Mechanics, thus providing generalizations of celebrated theorems of Herman and R\"ussmann, with applications to Celestial Mechanics.

\comment{Some specific problems of Celestial Mechanics fit well in these normal forms for which a KAM-type result can be given by direct application a technique of elimination of parameters. \\

As a  byproduct, an application, by a technique of elimination of parameters set up by R\"ussmann, Herman and other authors in the 80's, we prove a KAM-type result for the dissipative spin-orbit problem of Celestial Mechanics, a problem recently presented by Celletti-Chierchia \cite{Celletti-Chierchia:2009} and Stefanelli-Locatelli \cite{Locatelli-Stefanelli:2012}.  }

\end{abstract}
\maketitle
\setcounter{tocdepth}{1} 
\tableofcontents

\section{Introduction}
\label{intro} 

\subsection{Moser's normal form}
The starting point of this article is Moser's 1967 theorem \cite{Moser:1967} which, although has been used by various authors, has remained relatively unnoticed for several years. We will present an alternative proof of this reult, relying on a more geometrical and conceptual construction that will serve as inspiration to prove new normal forms theorems that carry the seeds of further results around the persistence of Diophantine tori. \\ Although the difficulties to overcome in this proof are the same as in the original one (proving the fast convergence of a Newton-like scheme), it relies on a relatively general inverse function theorem - theorem \ref{teorema inversa} - (unlike in Moser's approach), following an alternative strategy with respect to the one proposed by Zehnder in \cite{Zehnder:1975,Zehnder:1976}. Recently Wagener in \cite{Wagener:2010} generalized the theorem to vector fields of different kind of regularity, focusing on possible applications in the context of bifurcation theory. We focus here on the analytic category.\\  We introduce here Moser's result so to define the frame in which we will state our theorems.\smallskip \\
Let $\V$ be the space of germs of real analytic vector fields along $\T^n\times\set{0}$ in $\T^n\times\R^m$. Let us \co{fix} $\alpha\in\R^n$ and $A\in\Mat_m(\R)$ a diagonalizable matrix of eigenvalues $a_1,\ldots,a_m\in\C^m$. The focus of our interest is on the affine subspace of $\V$ consisting of vector fields of the form 
\begin{equation}
u(\teta,r)= (\alpha + O(r), A\cdot r + O(r^2)),
\label{u}
\end{equation}
where $O(r^k)$ stands for terms of order $\geq k$ which may depend on $\teta$ as well. We will denote this subset with $\U(\alpha,A)$.\\
Vector fields in $\U(\alpha,A)$ possess a reducible invariant quasi-periodic torus $\text{T}^n_0 := \T^n\times\set{0}$ of Floquet exponents $a_1,\ldots,a_m$. \\ We will refer to $\alpha_1,\ldots,\alpha_n,a_1,\ldots,a_m$  as the \co{characteristic numbers} or characteristic \co{frequencies}.

Let $\Lambda$ be the subspace of $\V$ of constant vector fields of the form \[\lambda(\teta,r) = (\beta, b + B\cdot r),\] where $\beta\in\R^n, b\in\R^m, B\in\Mat_m(\R): A\cdot b = 0,\, [A,B]=0.$\\ In the following we will refer to $\lambda$ as \co{(external) parameters} or \co{counter terms}.\footnote{Conditions $[A,B]= 0$ and $A\cdot b = 0$ guarantee the uniqueness of such counter terms.} 

Eventually, let $\G$ be the space of germs of real analytic isomorphisms of $\T^n\times\R^m$ of the form \[g(\teta,r)= (\varphi(\teta), R_0(\teta) + R_1(\teta)\cdot r),\] $\varphi$ being a diffeomorphism of the torus fixing the origin and $R_0, R_1$ being respectively an $\R^m$-valued and $\Mat_m(\R)$-valued functions defined on $\T^n$.

We assume that among the linear combinations \[i\,k\cdot\alpha + l\cdot a\qquad (k,l)\in\Z^n\times\Z^m,\,\abs{l}\leq 2,\quad\abs{l}=\abs{l_1} +\ldots +\abs{l_m}\] where $a= (a_1,\ldots,a_m)$, there are only finitely many which vanish. Moreover, to avoid resonances and small divisors, we impose the following Diophantine condition on $\alpha\in\R^n$ and the eigenvalues $(\bar a, 0):= (a_1,\ldots,a_{\mu}, 0,\ldots,0)\in(\C^{\ast})^{\mu}\times \C^{m-\mu}$, for some real positive $\gamma,\tau$, 
\begin{equation}
\abs{\imath k\cdot\alpha + l\cdot \bar a}\geq \frac{\gamma}{\pa{1 + \abs{k}}^\tau}\qquad\text{for all } (k,l)\in\Z^n\times\Z^{\mu}\setminus \set{(0,0)},\, \abs{l}\leq 2.
\label{DC}
\end{equation}
It is a known fact that if $\tau$ is large enough and $\gamma$ small enough, the measure of the set of "good frequencies" tends to the full measure as $\gamma$ tends to $0$. See \cite{Poschel:1989,Poschel:2001} and references therein. Also, remark that only purely imaginary parts of 
Floquet exponents may interfere and create small divisors, due to the factor $\imath$ in front of $k\cdot\alpha$. We will indicate with $\mathcal{D}_{\gamma,\tau}$ the set of characteristic numbers satisfying such a condition.

\begin{thm}[Moser 1967]If $v\in\V$ is close enough to $u^0\in\U(\alpha,A)$ there exists a unique triplet $(g,u,\lambda)\in\G\times\U(\alpha,A)\times\Lambda$, in the neighborhood of $(\id,u^0,0),$ such that $v = \push{g}u + \lambda$.
\label{Moser theorem}
\end{thm}
The notation $\push{g}u$ indicates the push-forward of $u$ by $g$: $\push{g}u= (g'\cdot u)\circ g^{-1}.$
\begin{figure}
\begin{picture}(0,0)%
\includegraphics{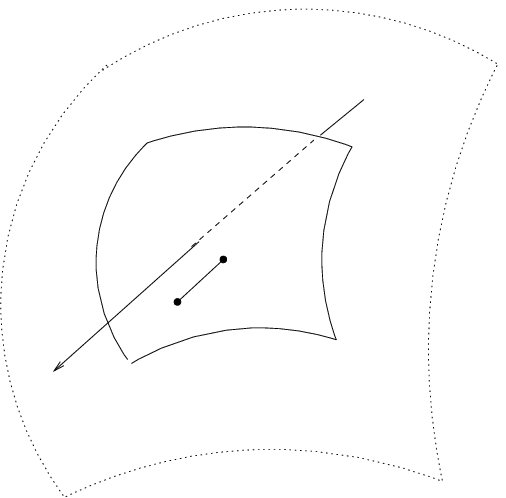}% 
\end{picture}%
\setlength{\unitlength}{1657sp}%
\begingroup\makeatletter\ifx\SetFigFont\undefined%
\gdef\SetFigFont#1#2#3#4#5{%
  \reset@font\fontsize{#1}{#2pt}%
  \fontfamily{#3}\fontseries{#4}\fontshape{#5}%
  \selectfont}%
\fi\endgroup%
\begin{picture}(5697,5596)(4167,-5919)
\put(4501,-4786){\makebox(0,0)[lb]{\smash{{\SetFigFont{5}{6.0}{\familydefault}{\mddefault}{\updefault}{\color[rgb]{0,0,0}$\Lambda\equiv \mathbb{R}^N$}%
}}}}
\put(5806,-2266){\makebox(0,0)[lb]{\smash{{\SetFigFont{5}{6.0}{\familydefault}{\mddefault}{\updefault}{\color[rgb]{0,0,0}$\mathcal{G}_{\ast}\mathcal{U}(\alpha,A)$}%
}}}}
\put(8956,-871){\makebox(0,0)[lb]{\smash{{\SetFigFont{5}{6.0}{\familydefault}{\mddefault}{\updefault}{\color[rgb]{0,0,0}$\mathcal{V}$}%
}}}}
\put(6841,-3166){\makebox(0,0)[lb]{\smash{{\SetFigFont{5}{6.0}{\familydefault}{\mddefault}{\updefault}{\color[rgb]{0,0,0}$g_{\ast}u$}%
}}}}
\put(6346,-3796){\makebox(0,0)[lb]{\smash{{\SetFigFont{5}{6.0}{\familydefault}{\mddefault}{\updefault}{\color[rgb]{0,0,0}$g_{\ast}u + \lambda$}%
}}}}
   
\end{picture}
\label{fig:1}  
\caption{Geometrical interpretation of Moser's theorem.} 
\end{figure}
\\
The introduction of the parameter $\lambda\in\Lambda$ is a powerful trick that switches the frequency obstruction (obstruction to the conjugacy to the initial dynamics) from one side of the conjugacy to the other. Although the presence of the counter-term $\lambda = (\beta, b + B\cdot r)$ breaks the dynamical conjugacy down, it is a \co{finite dimensional} obstruction: geometrically, the $\G$-orbits of all $u's$ in $\U(\alpha,A)$ form in $\V$ a submanifold of finite co-dimension $ N\leq n+ m + m^2$, transverse to $\Lambda$. This co-dimension will depend on the dimension of $\beta\in\R^n$ and the one of the kernels of $A$ and $[A,\cdot]$.\\
Zhender's approach and ours differ for the following reason, although both rely on the fact that the convergence of the Newton scheme is somewhat independent of the internal structure of variables.\\ Inverting the operator \[\phi: (g,u,\lambda)\mapsto \push{g}u + \lambda = v,\] as we will in section \ref{the normal form of Moser}, is equivalent to solving implicitly the pulled-back equation ($\pull{g} = \push{g^{-1}}$) \[\Phi(g,u,\lambda;v) = g^{\ast}(v - \lambda) - u=0,\] with respect to $u, g$ and $\lambda$, as Zehnder did. \\The problem is that whereas $\phi$ is a local diffeomorphism (in the sense of scales of Banach spaces), the linearization of $\Phi$,  $$ \frac{\partial\Phi}{\partial (g,u,\lambda)}(g,u,\lambda;v)\cdot (\delta g, \delta u,\delta\lambda) = \sq{\pull{g}( \lambda - v), g'^{-1}\cdot\delta g} + \pull{g}\delta\lambda + \delta u$$ is not invertible if for instance $\pull{g}(\lambda - v)$ is a resonant vector field. It is invertible in a whole neighborhood of $\Phi = 0$ only up to a second order term (see Zehnder \cite[\S $5$]{Zehnder:1975}), which prevents us from using a Newton scheme in a straightforward manner.
In section \ref{the normal form of Moser} we give the functional setting in which we prove the theorem of Moser.

\subsection{Persistence of tori: elimination of parameters}

 The fact that the submanifold $\G_{\ast}\U(\alpha,A)$ has finite co-dimension leaves the possibility that in some cases obstructions represented by counter terms can be even totally eliminated: if the system depends on a sufficient number of free parameters - either internal or external parameters - and $\lambda$ smoothly depends on them we can try to tune the parameters so that $\lambda = 0$.\\ When $\lambda = 0$ we have $g_{\ast} u = v$: the image $g(\normal{T}^n_0)$ is invariant for $v$ and $u$ determines the first order dynamics along this torus.\\ When $\push{g}u + \lambda = v$, we will loosely say that
 \begin{itemize}[leftmargin=*]
 \item $\text{T}^n_0$ persists \emph{up to twist}, if $b = 0$ and $B=0$  
 \item $\text{T}^n_0$ persists \emph{up to translation}, if $\beta=0$ and $B=0$
 \item $\text{T}^n_0$ persists \emph{up to twist-translation}, if $B=0$.
  \end{itemize}
  The \co{infinite} dimensional conjugacy problem is reduced to a \co{finite} dimensional one.  
   In some cases the crucial point is to allow frequencies $(\alpha_1,\ldots,\alpha_n,a_1,\ldots,a_m)$ to vary, using the fact that $\lambda$ is Whitney-smooth with respect to them. Herman understood the power of this reduction in the $80'$s (see \cite{Sevryuk:1999}) and other authors (R\"ussmann, Sevryuk, Chenciner, Broer-Huitema-Takens, Féjoz...) adopted this technique of "elimination of parameters" to prove invariant tori theorems in multiple contexts, at various level of generality, contributing to clarify this procedure. See \cite{Broer-Huitema-Takens:1990, Chenciner:1985, Chenciner:1985a, Sevryuk:1999, Sevryuk:2006} at instance. 
   \comment{ In the case of interest to us (proving the persistence of an invariant attractive torus in the spin-orbit problem, see section below) we will have at our disposal an external vector of free parameters, and we will not need to vary the characteristic frequencies $(\alpha,a)$ to eliminate the obstruction $\lambda$.}
\subsection{Main results}   

The proposed geometrization of Moser's result arises different questions about the equivariance of the correction with respect to the groupoid $\G$ and its canonical sub-groupoids. In section \ref{The Hamiltonian case: Herman's theorem} and \ref{Hamiltonian-dissipative section} we study some of these equivariance properties in some particular cases issued from Hamiltonian dynamics and its dissipatives versions issued from Celestial Mechanics. As a by-product, several twisted-torus and translated-torus theorems are given (see section \ref{general Russmann}).

\comment{ The motivation of this work comes from Celestial Mechanics. One of the main interests of Celestial Mechanics is to understand the dynamics of the revolution of a body (or many bodies) around a fixed massive point or the rotation's dynamics of such a body around its spin-axis. Since the internal structure of planets and satellites can influence such motions, depending on its elastic or non elastic response to the gravitational attraction, it is of major importance to understand the behavior of a planet or satellite subjected to non conservative forces (tide's effects in particular). The first attempt in this direction, is to consider systems in which the dissipative effect translates in the presence of a linear friction (whose precise characteristics is very difficult to determine) in the equations of motions, while the remaining interactions are still Hamiltonian. }
\subsubsection{Hamiltonian-dissipative systems}

In the preceeding line of thought we start by recalling the classic Hamiltonian counter part of Moser's theorem (see section \ref{The Hamiltonian case: Herman's theorem}). \\ On $\T^n\times\R^n$ , if $\Uh(\alpha,0)\subset \U(\alpha,A)$ is the space (of germs) of Hamiltonian vector fields of the form \eqref{u} (hence $\alpha$ is Diophantine and $A=0$), contained in the space $\Vh\subset \V$ of Hamiltonian vector fields, and if $\G^{\Ham}\subset \G$ is the space of germs of exact-symplectic isomorphisms of the form $$g(\teta,r)= (\varphi(\teta), \,^t\varphi'(\teta)^{-1}\cdot(r + S'(\teta))),$$ where $\varphi$ is an isomorphism of $\T^n$ fixing the origin and $S$ a function on $\T^n$ fixing the origin, the space of counter terms is reduced to the set of $\lambda = (\beta,0)$: we have Herman's "twisted conjugacy" theorem, see\cite{Fejoz:2004, Fejoz:2010, Fejoz:2015}.

\begin{thm*}[Herman] If $v$ is sufficiently close to $u^0\in\Uh(\alpha,0)$, the torus $\normal{T}^n_0$ persists up to twist. \comment{More specifically, there exists a unique triplet $(g,u,\beta)\in\G^{\Ham}\times \Uh(\alpha,0)\times\Lambda(\beta)$, in the neighborhood of $(\id,u^0,0)$, such that $\push{g}u + \beta= v$.}
\end{thm*}
 
 \begin{itemize}[label=$\bullet$,leftmargin=*]

\item In section \ref{Hamiltonian-dissipative section} we prove a first dissipative-generalization of this classic result by considering the affine spaces\footnote{We noted $\partial_r = (\partial_{r_1},\ldots,\partial_{r_n})$ and omitted the tensor product sign $r\tensor\partial_r$} 
$$\Uh(\alpha,-\eta):=\Uh(\alpha,0)\oplus (-\eta  r\,\partial_r)\subset\Vh\oplus (-\eta  r\,\partial_r),$$ 
where $\eta\in\R^{\ast}$, extending the normal direction with the constant linear term $\eta r$  (when $\eta>0$ we speak of "radial dissipation"), but keeping the same space of exact-symplectic isomorphisms $\G^{\Ham}$ and Hamiltonian corrections $\lambda = (\beta,0)$.

 \begingroup
\setcounter{tmp}{\value{thm}}% store current value of theorem counter
\setcounter{thm}{0} %assign desired value to theorem counter
\renewcommand\thethm{\Alph{thm}}% locally redefine the representation of the theorem counter

\begin{thm}
 If $v$ is sufficiently close to $u^0\in\Uh(\alpha,-\eta)$, for any $\eta\in [-\eta_0,\eta_0]$, $\eta_0\in\R^+$, the torus $\normal{T}^n_0$ persists up to twist.\comment{\\ More specifically there exists a unique triplet $(g,u,\beta)\in\G^{\Ham}\times \Uh(\alpha,-\eta)\times\Lambda(\beta)$,  in the neighborhood of $(\id,u^0,0)$, such that $\push{g}u + \beta= v$.}
\label{theorem A}
\end{thm}
\endgroup

Remark that, a part from the fact that the number of counter terms breaking the dynamical conjugacy is \emph{the same} as in the purely Hamiltonian context (a twisting term $\beta\,\partial_{\teta}, \, \beta\in\R^n$ in the angle's direction), we control both the tangent and the normal dynamics of the torus, which survive perturbations (up to twist) uniformly with respect to dissipation (as opposed to the classic normally hyperbolic frame). See remark \ref{remark uniform} in the proof of Proposition \ref{proposition infinitesimal ham}.  \smallskip\\
If we make abstraction of the geometry, this study can be included in the general Moser's theorem in the case $A$ has simple non $0$ eigenvalues, where the corrections space is immediately given by the set of $\lambda = (\beta, B\cdot r),$ with $B$ a diagonal matrix. \\ A diagram of inclusions summarizing these results is given at the end of section \ref{Hamiltonian-dissipative section}.
 \item If vector fields in $\Uh(\alpha,-\eta)$ satisfy a \emph{torsion hypotesis} (coming from Hamiltonians with non degenerate quadratic term), it is possible to widen the space of perturbations to $$ \Vh\oplus \pa{-\eta r + \zeta}\partial_r,$$ where $\zeta\in\R^n$, keeping the same space $\Uh(\alpha,-\eta)$, provided that the space of transformations is extended to the space $\G^{\omega}$ of symplectic isomorphisms of the form $$g(\teta,r)= (\varphi(\teta), \,^t\varphi'(\teta)^{-1}\cdot(r + S'(\teta) + \xi)),\quad \xi\in\R^n.$$
 The space of counter terms becomes the set of translations \co{in action} $\lambda = (0, b)$. 
  \begingroup
\setcounter{tmp}{\value{thm}}% store current value of theorem counter
\setcounter{thm}{1} %assign desired value to theorem counter
\renewcommand\thethm{\Alph{thm}}% locally redefine the representation of the theorem counter

\begin{thm}[vector fields à la R\"ussmann]
If $v$ is sufficiently close to $u^0\in\Uh(\alpha,-\eta)$, for any $\eta\in [-\eta_0,\eta_0]$, $\eta_0\in\R^+$,  the torus $\normal{T}^n_0$ persists up to translation.\comment{\\More specifically there exists a unique $(g,u,\beta)\in\G^{\omega}\times \Uh(\alpha,-\eta)\times\Lambda(b)$,  in the neighborhood of $(\id,u^0,0)$, such that $\push{g}u + b = v$.}
\label{theorem B}
\end{thm}
\endgroup
Here again, the bound on admissible perturbations $v$ is proved to be uniform with respect to $\eta$ and the first order dynamics on the translated torus $g(\normal{T}^n_0)$ is again characterized by the same frequencies $(\alpha, \eta)$. Thus, this result can be seen as a multidimensional remarkable generalization for vector fields of R\"ussmann's translated curve theorem \cite{Russmann:1970}, where the normal dynamics of the translated torus is determined too - in addition to the $\alpha$-quasiperiodic tangent one -. 
\end{itemize}
\subsubsection{General-dissipative systems}

At the expense of losing the control of the normal dynamics, we can achieve in proving the persistence - up to translation or up to twist - of $\alpha$-quasi-periodic Diophantine tori in a more general dissipative frame than the Hamiltionian-dissipative one previously considered, by application of the classic implicit function theorem (in finite dimension). The following results will be proved in section \ref{general Russmann}, where a more functional statement will be given (Theorem \ref{proposition B=0} and \ref{translated torus}).

On $\T^n\times\R^m$, let $u\in\U(\alpha,A)$, defined in expression \eqref{u}, be such that $A$ has simple, real, non $0$  eigenvalues $a_1,\ldots,a_m$.  This hypothesis of course implies that the only frequencies that can cause small divisors are the tangential ones $\alpha_1,\ldots,\alpha_n$, so that we only need to require the standard Diophantine hypothesis on $\alpha$.

  \begingroup
\setcounter{tmp}{\value{thm}}% store current value of theorem counter
\setcounter{thm}{2} %assign desired value to theorem counter
\renewcommand\thethm{\Alph{thm}}% locally redefine the representation of the theorem counter

\begin{thm}[Twisted torus]
Let $\alpha$ be Diophantine. If $v$ is sufficiently close to $u^0\in\U(\alpha,A)$, the torus $\normal{T}^n_0$ persists up to twist.
\label{theorem C}
\end{thm}
\endgroup

\comment{\begin{cor*}[Twisted torus]
Let $\alpha$ be Diophantine and $A$ have simple, real, non $0$ eigenvalues. If $v$ is sufficiently close to $u^0\in\U(\alpha,A)$, the torus $\normal{T}^n_0$ persists up to twist.
\label{cor C}
\end{cor*}}

Eventually, let $m\geq n$ and let vector fields in $\U(\alpha,A)$ have a twist in the following sense: the matrix term $u_1: \T^n\to \Mat_{n\times m}(\R)$ in \[u(\teta,r)= (\alpha + u_1(\teta)\cdot r + O(r^2), A\cdot r + O(r^2)),\] is such that $\int_{\T^n}u_1(\teta)\, d\teta$ has maximal rank $n$.

  \begingroup
\setcounter{tmp}{\value{thm}}% store current value of theorem counter
\setcounter{thm}{3} %assign desired value to theorem counter
\renewcommand\thethm{\Alph{thm}}% locally redefine the representation of the theorem counter

\begin{thm}[Translated torus]
Let $u^0\in\U(\alpha,A)$ have a twist and $\alpha$ be Diophantine. If $v$ is sufficiently close to $u^0$, the torus $\normal{T}^n_0$ persists up to translation.
\label{theorem D}
\end{thm}
\endgroup

\subsection{An application to Celestial Mechanics} The motivation of the previous geometric results on normal forms for dissipative systems comes from Celestial Mechanics. The normal forms we proved provide ready-to-use theorems that, in some cases fit very well concrete problems issued from Celestial Mechanics. Besides, if on the one hand these theorems clarify in a very neat way the "lack of parameters" problem, on the one other the procedure of elimination of parameters highlights relations between physical parameters and the existence of invariant tori in the system.

\comment{One of the main interests of Celestial Mechanics is to understand the dynamics of the revolution of a body (or many bodies) around a fixed massive point or the rotation's dynamics of such a body around its spin-axis. Since the internal structure of planets and satellites can influence such motions, depending on its elastic or non elastic response to the gravitational attraction, it is of major importance to understand the behavior of a planet or satellite subjected to non conservative forces (tide's effects in particular). The first attempt in this direction, has been to consider systems in which the dissipative effect translates in the presence of a friction term (whose precise characteristics is very difficult to determine) in the equations of motions, while the remaining interactions are still Hamiltonian. }

To give a major exemple, we conclude the paper with an application of Theorem \ref{theorem B} to the problem of persistence of quasi-periodic attractors in the spin-orbit system; this astronomical problem wants to study the dynamics of the rotation about its spin axis of a non-rigid and non-elastic body whose center of mass revolves along a given elliptic Keplerian orbit around a fixed massive point (see section \ref{The spin-orbit problem} for the precise formulations of the model). A study of this problem using a PDE approach was given in \cite{Celletti-Chierchia:2009}, while a generalization in higher dimension was presented in \cite{Locatelli-Stefanelli:2012}, but using Lie series techniques instead. 

For the $2n$-dimensional model, on $\T^n\times\R^n$ we consider vector field of the form $$ \hat{v} = v - \eta(r - \Omega)\partial_r $$ where $v\in\Vh$ is a perturbation of $u^0\in\Uh(\alpha,0)$ with non-degenerate torsion, $\eta\in\R^+$ a dissipation constant and $\Omega\in\R^n$ a vector of external free parameters. By simple application of the translated torus theorem \ref{theorem B} and the implicit function theorem in finite dimension, the persistence result is phrased as follows.

\begin{thm*}[spin-orbit in $n$ d.o.f.] If $v$ is sufficiently close to $u^0$, there exists a unique frequency adjustment $\Omega\in\R^n$ close to $0$, a unique $u\in\Uh(\alpha,-\eta)$ and a unique $g\in\G^{\omega}$ such that $\hat{v}$ verifies $\push{g} u = \hat{v}$. Hence $\hat{v}$ possesses an invariant $\alpha$-quasi-periodic and $\eta$-normally attractive torus. 
\end{thm*}

This result is eventually applied to the astronomical spin-orbit problem; in this case $n=2$ and the vector field $v$ corresponds to the Hamiltonian (issued from a non-autonomous model)
$$ H(\teta,r) = \alpha r_1 + r_2 + \frac{1}{2} r_1^2 + \eps f(\teta,r),$$ where the degeneracy of torsion is nothing but an artificial problem. The frequency adjustment is in this case $\Omega= (\nu-\alpha,0)$, $\nu\in\R$.\\ 
The point of view of normal forms and elimination of parameters allow to formulate the persistence result as follows.

\begin{thm*}[Surfaces of invariant tori] Let $\eps_0$ be the maximal value that the perturbation can attain. In the space $(\eps,\eta,\nu)$, to every $\alpha$ Diophantine corresponds a surface $\nu=\nu(\eta,\eps)$ ($\eps\in [0,\eps_0]$) analytic in $\eps$, smooth in $\eta$, for those values of parameters of which $\hat{v}$ admits an invariant $\alpha$-quasi-periodic torus. This torus is $\eta$-normally attractive (resp. repulsive) if $\eta>0$ (resp. $\eta<0$)
\end{thm*}

All the due reductions being made (see corollary \ref{Russmann spin-orbit}), the proof is a particular case of the theorem \ref{theorem B} "à la R\"ussmann" and the elimination of the translation parameter; as a byproduct it starts a portrait of the parameters' space of this problem in terms of zones where different kind of dynamics occurs. See theorem \ref{teorema superfici} and corollary \ref{corollario curve}.

\comment{ In this line of thought in section \ref{The Hamiltonian case: Herman's theorem} we start by presenting Moser's theorem in the purely {Hamiltonian context} (proved independently by Herman and presented in Féjoz's works \cite{Fejoz:2004} and \cite{Fejoz:2010} as the "twisted conjugacy" theorem).\\ Vector fields $v^{\h}\in\Vh\subset\V$ involved correspond to real analytic Hamiltonians $H$ defined in a neighborhood of $\text{T}^n_0:=\T^n\times\set{0}$, the corresponding subspace of Hamiltonian vector fields $u^{\h}\in\Uh(\alpha,0)$ for which $\text{T}^n_0$ is invariant come from a Hamiltonian $K(\teta,r)= \alpha\cdot r + O(r^2)$ instead:
\[u^{\h}=\eqsys{\dot\teta = \alpha + O(r)\\
\dot r = O(r^2).}\] }

\comment{\subsection{Dissipative systems: normal forms "à la Herman" and "à la R\"ussmann"}

From this context, slightly modifying the class of these vector fields by adding the aforementioned dissipative linear term\footnote{We noted $\partial_r = (\partial_{r_1},\ldots,\partial_{r_n})$ and omitted the tensor product sign $r\tensor\partial_r$} in the normal direction $u^{\h}\oplus (-\eta  r\,\partial_r)$, $\eta\in\R^{+}$, it is possible to prove a first \emph{generalization of Herman's theorem} to dissipative systems (see section \ref{dissipative Herman} and theorem \ref{theorem Herman} baptized "Herman dissipative") in which  the number of needed external parameters breaking the dynamical conjugacy is the same as in the purely Hamiltonian context (a translation term $\beta\,\partial_{\teta}, \, \beta\in\R^n$ in the angle's direction). For this to be true it is fundamental that the dissipation acts the same in any direction: the constant matrix $A$ appearing in $r$-directions is a homothety $-\eta\id$. See lemma \ref{lemma Hamiltonian1}.

\comment {In section \ref{The Hamiltonian case: Herman's theorem}, we present the normal form in the classical Hamiltonian context (proved independently by Herman, and presented also by Féjoz in his works \cite{Fejoz:2004} and \cite{Fejoz:2010}) from which, modifying the class of vector fields involved by adding a linear dissipative term in the normal direction, it is possible to give and prove a first generalization in the direction of dissipative systems (see section \ref{dissipative Herman}). The more general dissipative case is presented as well in \eqref{general dissipative}.  A diagram portraying normal forms in the dissipative context, from the more general to the Hamiltonian one, concludes this part.}

In a second step (see section \ref{Russmann section vector fields}) we add a \emph{twist hypothesis} on the Hamiltonian  assuming that the average of the coefficient of the quadratic term in $$K(\teta,r)= \alpha\cdot r + \frac{1}{2}Q(\teta)\cdot r^2 + O(r^3)$$ is a non degenerate quadratic form: $\det \int_{\T_n} Q(\teta)\,\frac{d\teta}{(2\pi)^n}\neq 0$.\\ In this context it is natural to take advantage of this non degeneracy condition and perform transformations by symplectic diffeomorphisms. Because of the presence of the constant term $-\eta r\,\partial_r$ we obtain a translated torus result via a normal form theorem (see section \ref{Russmann section vector fields} and the \emph{translated torus theorem} therein) that can be considered as an analog for vector fields in this class of the celebrated R\"ussmann's translated curve theorem for diffeomorphisms of the annulus (see \cite{Russmann:1970}). As a matter of fact we prove that the operator 
\begin{equation}
\phi : (g,u,b)\mapsto \push{g}u + b\,\partial_r = v,\quad b\in\R^n
\label{translated torus intro}
\end{equation}
is a local diffeomorphism and the image of the torus $g(\normal{T}^n_0)$ by the flow of $v$ is translated by $b$. 

The more {general dissipative} case in which no Hamiltonian hypothesis is made and the dissipative term is provided by a more general matrix $A\cdot r$ with negative real eigenvalues, is also given as a straightforward corollary to Moser's theorem.\\ We eventually summarize these results in two diagrams that give a portrait of these dissipative systems in terms of normal forms (see section \ref{general dissipative} and the end of section \ref{Russmann section vector fields}).

}

\comment{

\subsubsection{The spin-orbit problem} 
In section \ref{section A KAM result}, we deduce the central results of Celletti-Chierchia \cite[Theorem $1$]{Celletti-Chierchia:2009} and Stefanelli-Locatelli \cite[Theorem $3.1$]{Locatelli-Stefanelli:2012} (who generalize the work of Celletti-Chierchia to any dimension) from the translated torus theorem of section \ref{Russmann section vector fields} and the  elimination of the translation parameter $b$ (see section \ref{elimination of parameters} and theorems \ref{teorema implicita} and \ref{A curve of invariant tori}).\\ The result can be phrased as following.\\
A satellite (or a planet) is said to be in $n:k$ spin-orbit resonance when it accomplishes $n$ complete rotations about its spin axis, while revolving exactly $k$ times around its planet (or star). There are various examples of such a motion in Astronomy, among which the Moon $(1:1)$ or Mercury $(3:2)$.\\ We want to study the dynamics of the rotation about its spin axis (represented by the angular variable $\teta$) of a triaxial body whose center of mass revolves along a given elliptic Keplerian orbit around a fixed massive point. The rotation axis is supposed to be perpendicular to the orbital plane. The internal structure of the body responds in a non-elastic way to gravitational forces. In the case of a triaxial ellipsoid with different equatorial axis, the calculation of the potential gives out a supplementary term proportional to the difference of the two smallest axes of inertia which perturbs the rotation. Under these hypothesis, the dynamics is described by the following equation of motion in $\R$ (see \cite{Goldreich-Peale:1966} and \cite{Laskar-Correia:2010} for instance):
\begin{equation}
\ddot\teta + \eta(\dot\teta - \nu) + \eps\partial_{\teta}f(\teta,t) = 0,
\label{spin-orbit intro}
\end{equation}
where $(\teta,t)$ are $2\pi$-periodic variables.\\
In the equation $\eta\dot\teta$, with $\eta\in\R^{+}$, is a friction term due to the non-elastic response of the internal structure,  $\eps f(\teta,t)$ is the potential function where $\eps = \frac{B-A}{C}$ ($A,B,C$ being the principal axes of inertia), while $\nu\in\R$ is an external free parameter representing the proper frequency of the attractor of the dynamics when $\eps = 0$. 
\begin{figure*}[h!]
\begin{picture}(0,0)%
\includegraphics{Fig2.eps}%
\end{picture}%
\setlength{\unitlength}{1657sp}%
\begingroup\makeatletter\ifx\SetFigFont\undefined%
\gdef\SetFigFont#1#2#3#4#5{%
  \reset@font\fontsize{#1}{#2pt}%
  \fontfamily{#3}\fontseries{#4}\fontshape{#5}%
  \selectfont}%
\fi\endgroup%
\begin{picture}(5199,4241)(2689,-6099)
\put(6211,-4156){\makebox(0,0)[lb]{\smash{{\SetFigFont{6}{7.2}{\familydefault}{\mddefault}{\updefault}{\color[rgb]{1,0,0}$\theta$}%
}}}}
\put(4141,-2356){\makebox(0,0)[lb]{\smash{{\SetFigFont{5}{6.0}{\familydefault}{\mddefault}{\updefault}{\color[rgb]{0,0,0}$\eta>0$ dissipation}%
}}}}
\put(4141,-2581){\makebox(0,0)[lb]{\smash{{\SetFigFont{5}{6.0}{\familydefault}{\mddefault}{\updefault}{\color[rgb]{0,0,0}$\eps$ perturbation}%
}}}}
\put(4141,-2806){\makebox(0,0)[lb]{\smash{{\SetFigFont{5}{6.0}{\familydefault}{\mddefault}{\updefault}{\color[rgb]{0,0,0}$\nu\in\R$ frequency}%
}}}}
\put(4141,-3031){\makebox(0,0)[lb]{\smash{{\SetFigFont{5}{6.0}{\familydefault}{\mddefault}{\updefault}{\color[rgb]{0,0,0}$f$ potential}%
}}}}
\end{picture}%
\caption{The spin-orbit problem}
\label{fig:2}  
    % Give a unique label
\end{figure*}\\ 
The vector field corresponding to equation \eqref{spin-orbit intro}, in the coordinates $(\teta, r = \dot\teta - \alpha)$, $\alpha$ being a fixed Diophantine frequency conveniently introduced, reads 
\begin{equation*}
\eqsys{\dot\teta = \alpha + r\\
\dot r = -\eta r + \eta(\nu - \alpha) -\eps \partial_{\teta}f(\teta,t).}
\end{equation*}
 When $\eps = 0$, for every Diophantine $\alpha$ the torus $r=0$ is invariant and attractive, provided $\nu=\alpha$. \smallskip\\
The persistence of the attractive torus under perturbation is shown in two steps. By the translated torus theorem \ref{theorem Russmann} adapted to this particular context (corollary \ref{Russmann spin-orbit}), we prove that, provided the perturbation is small enough, a normal form like \eqref{translated torus intro} exists for \emph{any} values of $\eta\in [-\eta_0,\eta_0]$. In a second step we show that the translation $b$ can be eliminated by implicitly solving $b(\alpha,\nu,\eta,\eps) = 0$ for a unique choice of $\nu$, on which $b$ smoothly depends. Since the maximal bound of the perturbation $\eps$ turns out to be uniform with respect to $\eta$, the smooth dependence on parameters allows to define, for every Diophantine $\alpha$, a surface $\nu=\nu(\eta,\eps)$ in the space $(\eta,\nu,\eps)$ on which the counter term $b(\alpha,\nu,\eta,\eps)$ vanishes, guaranteeing the existence of reducible $\alpha$-quasi-periodic attractive (resp. repulsive, when $\eta<0$) invariant tori (see theorem \ref{teorema superfici} and corollary \ref{corollario curve}). Every plane $\eps= const.$ ($\eps$ being an admissible perturbation), thus carries a Cantor set of curves $C_\alpha$, along which the counter term $b(\alpha,\nu,\eta,\eps)=0$.\\An important dichotomy between dissipative systems carrying Hamiltonian as opposed to non-Hamiltonian perturbation will be pointed out as conclusion of these results (see section \ref{dichotomy}). \\ This interpretation opens the way to a more global study in the parameter' space $(\eta,\nu,\eps)$ that will aim at delimiting regions in which the dynamics, or its important features, is understood. This will be carried out in a future work.

}

\section{The normal form of Moser}
\label{the normal form of Moser}
Theorem \ref{Moser theorem} will be deduced by the abstract inverse function theorem \ref{teorema inversa} and the regularity results contained in appendix \ref{section teorema inversa}.
\subsection{Complex extensions}
\label{complex extensions}
Let us extend the tori \[\T^n=\R^n/{2\pi\Z^n}\qquad \text{and}\qquad \normal{T}^n_0=\T^n\times\set{0}\subset\T^n\times\R^m,\] as \[\T^n_{\C}= \C^n/{2\pi\Z^n}\qquad \text{and} \qquad \text{T}^n_\C = \T^n_{\C}\times\C^m\] respectively, and consider the corresponding $s$-neighborhoods defined using $\ell^\infty$-balls (in the real normal bundle of the torus): \[\T^n_s=\set{\teta\in\T^n_{\C}:\, \max_{1\leq j \leq n}\abs{\Im\teta_j}\leq s}\quad\text{and}\quad \text{T}^n_s=\set{\pa{\teta,r}\in\text{T}^n_{\C}:\, \abs{(\Im\teta,r)}\leq s},\]
where $\abs{(\Im\teta,r)}:= \max_{1\leq j\leq n}\max(\abs{\Im\teta_j},\abs{r_j})$. \\

Let now $f: \normal{T}^n_s\to \C$ be holomorphic, and consider its Fourier expansion $f(\teta,r)=\sum_{k\in\Z^n}\,f_k(r)\,e^{i\,k\cdot\teta}$, noting $k\cdot\teta = k_1\teta_1 +\ldots k_n\teta_n$. In this context we introduce the so called "weighted norm": \[\abs{f}_s := \sum_{k\in\Z^n}\abs{f_k}\, e^{\abs{k}s},\quad \abs{k}=\abs{k_1}+\ldots +\abs{k_n},\] $\abs{f_k}=\sup_{\abs{r}<s} \abs{f_k(r)}$. Whenever $f : \normal{T}^n_s\to\C^{n}$, $\abs{f}_s = \max_{1\leq j\leq n}(\abs{f_j}_s)$, $f_j$ being the $j$-th component of $f(\teta,r)$.\\ It is a trivial fact that the classical sup-norm is bounded from above by the weighted norm: \[\sup_{z\in{\normal{T}^n_s}}\abs{f(z)}\leq\abs{f}_s\] and that $\abs{f}_s<+\infty$ whenever $f$ is analytic on its domain, which necessarily contains some $\normal{T}^n_{s'}$ with $s'>s$. In addition, the following useful inequalities hold if $f,g$ are analytic on $\normal{T}^n_{s'}$ \[\abs{f}_s\leq\abs{f}_{s'}\,\text{ for }\, 0<s<s',\] and \[\abs{fg}_{s'}\leq \abs{f}_{s'}\abs{g}_{s'}.\] For more details about the weighted norm, see for example \cite{Meyer:1975}.\\
In general for complex extensions $U_s$ and $V_{s'}$ of $\T^n\times\R^n$, we will denote $\mathcal{A}(U_s,V_{s'})$ the set of holomorphic functions from $U_s$ to $V_{s'}$ and $\mathcal{A}(U_s)$, endowed with the $s$-weighted norm, the Banach space $\mathcal{A}(U_s,\C)$.

Eventually, let $E$ and $F$ be two Banach spaces,

\begin{itemize}[leftmargin=*]
\item We indicate contractions with a dot "$\,\cdot\,$", with the convention that if $l_1,\ldots, l_{k+p}\in E^\ast$ and $x_1,\ldots, x_p\in E$ 
\begin{equation*}
(l_1\tensor\ldots\tensor l_{k+p})\cdot (x_1\tensor\ldots\tensor x_p) = l_1\tensor\ldots \tensor l_k \gen{l_{k+1},x_1}\ldots \gen{l_{k+p},x_p}.
\end{equation*}
In particular, if $l\in E^\ast$, we simply note $l^n= l\tensor\ldots\tensor l$.\\
\item If $f$ is a differentiable map between two open sets of $E$ and $F$, $f'(x)$ is considered as a linear map belonging to $F\tensor E^{\ast}$,  $f'(x): \zeta\mapsto f'(x)\cdot\zeta$; the corresponding norm will be the standard operator norm \[\abs{f'(x)} = \sup_{\zeta\in E, \abs{\zeta}_E=1}\abs{f'(x)\cdot\zeta}_F.\]
\end{itemize}


\subsection{Space of conjugacies}We define $\mathcal{G}^\sigma_s$ as the subspace of $\mathcal{A}(\text{T}^n_s,\text{T}^n_\C)$ consisting of maps of the form \[g(\teta,r)=\pa{\varphi(\teta), R_0(\teta) + R_1(\teta)\cdot r},\]where
\begin{itemize}[leftmargin=*]
\item  the function $\varphi$ belongs to $\mathcal{A}(\T^n_s,\T^n_\C)$ and is such that $\varphi(0) = 0$ and \[\abs{\varphi - \id}_s<\sigma,\] where $\varphi - \id$ is considered as going from $\T^n_s$ to $\C^n$,\\
\item  $R_0 \in \mathcal{A}(\T^n_s,\C^m)$ and  $R_1\in\mathcal{A}(\T^n_s,\Mat_m(\C))$ satisfy 
\[\abs{R_0(\teta) + R_1(\teta)\cdot r - r}_s < \sigma.\]
\end{itemize}
%Indicating with $g_j(\teta,r)$ the $j$-th component of $g(\teta,r)$, we put \[\abs{g}_s = \max_{1\leq j\leq 2n} (\abs{g_j(\teta,r)}_s).\]
\begin{figure}[h!]
\begin{picture}(0,0)%
\includegraphics{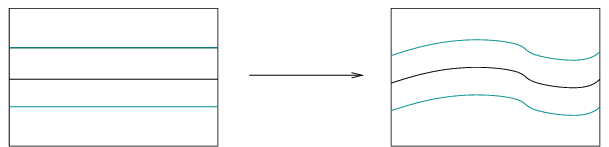}%
\end{picture}%
\setlength{\unitlength}{1657sp}%
\begingroup\makeatletter\ifx\SetFigFont\undefined%
\gdef\SetFigFont#1#2#3#4#5{%
  \reset@font\fontsize{#1}{#2pt}%
  \fontfamily{#3}\fontseries{#4}\fontshape{#5}%
  \selectfont}%
\fi\endgroup%
\begin{picture}(7905,1881)(1561,-5473)
\put(5581,-4426){\makebox(0,0)[lb]{\smash{{\SetFigFont{6}{7.2}{\familydefault}{\mddefault}{\updefault}{\color[rgb]{0,0,0}$g$}%
}}}}
\put(1846,-3751){\makebox(0,0)[lb]{\smash{{\SetFigFont{6}{7.2}{\familydefault}{\mddefault}{\updefault}{\color[rgb]{0,0,0}$\text{T}^n_{s+\sigma}$}%
}}}}
\put(1576,-4246){\makebox(0,0)[lb]{\smash{{\SetFigFont{6}{7.2}{\familydefault}{\mddefault}{\updefault}{\color[rgb]{0,.56,.56}$\text{T}^n_s$}%
}}}}
\put(1711,-4606){\makebox(0,0)[lb]{\smash{{\SetFigFont{6}{7.2}{\familydefault}{\mddefault}{\updefault}{\color[rgb]{0,0,0}$\text{T}^n_{0}$}%
}}}}
\put(9451,-4201){\makebox(0,0)[lb]{\smash{{\SetFigFont{6}{7.2}{\familydefault}{\mddefault}{\updefault}{\color[rgb]{0,.56,.56}$g(\text{T}^n_s)$}%
}}}}
\put(9406,-4651){\makebox(0,0)[lb]{\smash{{\SetFigFont{5}{6.0}{\familydefault}{\mddefault}{\updefault}{\color[rgb]{0,0,0}$g(\text{T}^n_0)$}%
}}}}
\end{picture}%

\caption{Deformed complex domain}
\end{figure}
The "Lie Algebra" $T_{\id}\G^\sigma_s$ of $\mathcal{G}^\sigma_s$, consists of maps \[\dot g (\teta,r)= \pa{\dot\varphi(\teta), \dot R_0(\teta) + \dot R_1(\teta)\cdot r}.\]
Here $\dot{g}$ lies in $\mathcal{A}(\text{T}^n_s,\C^{n+m})$; more specifically $\dot{\varphi}\in\mathcal{A}(\T^n_s,\C^n)$, $\dot{R}_0\in\mathcal{A}(\T^n_s,\C^m)$ and $\dot{R}_1\in\mathcal{A}(\T^n_s,\Mat_m(\C))$. We endow this space too with the norm \[\abs{\dot{g}}_s = \max_{1\leq j\leq n+m} (\abs{\dot{g}_j(\teta,r)}_s).\] 
\subsection{Spaces of vector fields} 
\label{Spaces of vector fields}
We define 
\begin{itemize}[leftmargin=*]
\item $\mathcal{V}_s= \mathcal{A}(\text{T}^n_s,\C^{n+m})$, endowed with the norm \[\abs{v}_s:= \max_{1\leq j\leq n+m}(\abs{v_j(\teta,r)}_s),\] and $\mathcal{V}=\bigcup_s\,\mathcal{V}_s$.
\item For $\alpha\in\R^n$ and $A\in\Mat_n{\R}$, $\mathcal{U}_s(\alpha,A)$ is the subspace of $\V_s$ consisting of vector fields in the form \[u(\teta,r)=\pa{\alpha + O(r) , A\cdot r + O(r^2)}.\]
\end{itemize}

Finally, for a given isomorphism $g\in\mathcal{G}^\sigma_{s}$, we define as \[\abs{v}_{g,s} := \abs{\pull{g} v}_s\] a "deformed" norm depending on $g$,  the notation $\pull{g}$ standing for the pull-back of $v$: this in order not to shrink artificially the domains of analyticity. The problem, in a smooth context, may be solved without changing the domain, by using plateau functions. 

\subsection{The normal form operator $\phi$}
According to theorem \ref{theorem well def} and corollary \ref{cor well def}, the operators 
\begin{equation}
\phi: \G^{\sigma/n}_{s+\sigma}\times\U_{s+\sigma}(\alpha,A)\times\Lambda\to \V_s,\, (g,u,\lambda)\mapsto \push{g}u + \lambda,
\label{normal operator}
\end{equation}
$\push{g}u = (g'\cdot u)\circ g^{-1}$, are now defined. It would be more appropriate to write $\phi_{s,\sigma}$ but, since these operators commute with source and target spaces, we will refer to them simply as $\phi$.\\ We will always assume that $0<s<s+\sigma<1$ and $\sigma<s$.\\ In the following we do not intend to be optimal. 

\subsection{Cohomological equations}

Here we present three derivation operators and the three associated cohomological equations.\\ 
Let $\alpha\in\R^n$ and $a = (a_1,\ldots,a_m)\in\C^m$, the vector of eigenvalues of a matrix $A\in\Mat_m(\R)$, satisfy the following conditions, which all follow from \eqref{DC} in the case $i\,k\cdot\alpha + l\cdot a\neq 0$.
\begin{align}
\label{Dio 1}
&\abs{k\cdot\alpha}\geq \frac{\gamma}{\abs{k}^\tau},\qquad \forall k\in\Z^n\setminus\set{0}\\
\label{Dio 2}
&\abs{\imath k\cdot\alpha + a_j}\geq \frac{\gamma}{\pa{1 + \abs{k}}^\tau},\qquad \forall k\in\Z^n, j=1,\dots,m,\\
\label{Dio 3}
&\abs{\imath k\cdot\alpha + l\cdot a}\geq \frac{\gamma}{\pa{1 + \abs{k}}^\tau},\qquad \forall (k,l)\in\Z^n\times\Z^m\setminus\set{0},\quad\abs{l}=2.
\end{align}

Let us consider a constant vector field $\alpha = (\alpha_1,\ldots,\alpha_n)$  on $\T^n_s$, identified with a vector $\alpha\in\R^n$ and the Lie derivative operator associated to it  
\begin{equation*}
L_{\alpha} : \mathcal{A}(\T^n_{s+\sigma}) \to \mathcal {A}(\T^n_s),\quad f \mapsto L_\alpha f = f'\cdot \alpha := \sum_{j=1}^n \alpha_j \frac{\partial f}{\partial\teta_j},
%\label{cohomological1}
\end{equation*}
$f$ being an analytic function on $\T^n_{s+\sigma}$ with values in $\C$.
\begin{lemma}[Straightening dynamics on the torus]
Let $\alpha\in\R^n$ satisfy condition \eqref{Dio 1} and let $0<s<s+\sigma$. For every $g\in\mathcal{A}(\T^n_{s+\sigma},\C)$ having zero average on the torus, there exists a unique preimage $f\in\mathcal{A}(\T^n_s,\C)$ of zero average such that \[L_\alpha f = g;\] moreover, the following estimate holds\[\abs{f}_{s}=\abs{L_{\alpha}^{-1}g}_s\leq \frac{C_1}{\gamma}\frac{1}{ \sigma^{n+\tau}}\abs{g}_{s+\sigma},\] $C_1$ being a constant depending only on the dimension $n$ and the exponent $\tau$. 
\label{tangent lemma}
\end{lemma}
\begin{proof}
Let \begin{equation*}
g(\teta)= \sum_{k\in\Z^n\setminus\set{0}} g_k e^{i\, k\cdot \teta},
\end{equation*}
be the Fourier expansion of $g$. Coefficients $g_k$ decay exponentially:
\begin{equation*}\abs{g_k}=\abs{\int_{\T^n}g(\teta)e^{- i\,k\cdot\teta}\,\frac{d\teta}{2\pi}}\leq\abs{g}_{s+\sigma}e^{-\abs{k}(s+\sigma)},
\end{equation*}
obtaining the inequality by deforming the path of integration to $\Im\teta_j = - \sgn (k_j)(s+\sigma)$. Expanding the term $L_\alpha f$ too, we see that a formal solution of $L_\alpha f = g$ is given by 
\begin{equation}
f = \sum_{k\in\Z^n\setminus\set{0}} \frac{g_k}{i\,k\cdot \alpha} e^{i\, k\cdot \teta}.
\label{inversa tangente}
\end{equation}
Taking into account Diophantine condition \eqref{Dio 1} we have
\begin{align*}
\abs{f}_s &\leq \frac{\abs{g}_{s+\sigma}}{\gamma}\sum_k \abs{k}^\tau e^{-\abs{k}\sigma}
         \leq \frac{2^n\abs{g}_{s+\sigma}}{\gamma}\sum_{\ell \geq 1}\binom{\ell + n + 1}{\ell} e^{-\ell\sigma}\ell^\tau \\
         &\leq \frac{4^n\abs{g}_{s+\sigma}}{\gamma(n-1)!}\sum_{\ell \geq 1}(n + \ell - 1)^{n - 1 + \tau} e^{-\ell\sigma}\\
         &\leq \frac{4^n\abs{g}_{s+\sigma}}{\gamma(n-1)!}\int_1^\infty (\ell + n -1)^{n + \tau - 1} e^{-(\ell - 1)\sigma}\,d\ell.
\end{align*}
The integral is equal to 
\begin{align*}
&\sigma^{-\tau - n} e^{n\sigma}\int_{n\sigma}^\infty \ell^{\tau +n -1}e^{-\ell}\,d\ell\\
& < \sigma^{-\tau - n} e^{n\sigma} \int_0^\infty \ell^{\tau +n -1}e^{-\ell}\,d\ell = \sigma^{-\tau - n} e^{n\sigma} \Gamma (\tau + n).
\end{align*}
Hence $f\in\A(\T^n_s)$ and satisfies the claimed estimate.
\end{proof}

 Let 
\begin{equation*}
L_\alpha + A : \mathcal{A}(\T^n_{s+\sigma},\C^m) \to \mathcal {A}(\T^n_s,\C^m),\quad f\mapsto L_\alpha f + A\cdot f = f'\cdot \alpha + A\cdot f.
\label{cohomological2}
\end{equation*}
\begin{lemma}[Relocating the torus]
Let $\alpha\in\R^n$ and $A\in\Mat_m(\R)$ be a diagonalizable matrix satisfying the Diophantine condition \eqref{Dio 2}. For every $g\in\mathcal{A}(\T^n_{s+\sigma},\C^m)$, there exists a unique preimage $f\in\mathcal{A}(\T^n_{s},\C^m)$  by $L_\alpha + A$. Moreover the following estimate holds
\[\abs{f}_{s}=\abs{\pa{L_\alpha + A}^{-1} g }_s\leq \frac{C_2}{\gamma}\frac{1}{ \sigma^{n+\tau}}\abs{g}_{s+\sigma}, \] $C_2$ being a constant depending only on the dimension $n$ and the exponent $\tau$.
\label{normal lemma}
\end{lemma} 
\begin{proof}
Let us start for simplicity with the scalar case $g\in\mathcal{A}(\T^n_{s+\sigma})$ and $A=a\neq 0\in\R$. Expanding both sides of $L_\alpha f + a\cdot f = g $ we see that the Fourier coefficients of the formal preimage $f$ is given by \[f_k= \frac{g_k}{i k\cdot\alpha + a},\] hence 
\begin{equation}
f = (L_\alpha + a)^{-1}g = \sum_{k\in{\Z^n}}\frac{g_k}{i k\cdot\alpha + a}\,e^{i k\cdot\teta}.
\label{inversa normale}
\end{equation}
Taking now into account the Diophantine condition and doing the same sort of calculations as in Lemma \ref{tangent lemma}, we get the wanted estimate.

The case where $A$ is a diagonal matrix can be recovered from the scalar one just by noticing that to $g(\teta)=\pa{g^1(\teta),\ldots,g^m(\teta)}$ correspond a preimage $f(\teta)= (f^1(\teta),\ldots,f^m(\teta))$ whose components read like in the scalar case.

When $A$ is diagonalizable, let $P\in\GL_n(\C)$ such that $PAP^{-1}$ is diagonal. Considering $f'\cdot\alpha + A\cdot f = g$, and left multiplying both sides by $P$, we get \[\tilde{f}'\cdot\alpha + PAP^{-1}\tilde{f}=\tilde{g},\] where we have set $\tilde{g}=Pg$ and $\tilde{f}=Pf$. This equation has a unique solution with the wanted estimates. We just need to put $f= P^{-1}\tilde{f}$. 

\end{proof}

Finally, consider an analytic function $F$ on $\T^n_{s+\sigma}$ with values in $\Mat_m(\C)$. Define the operator 
\begin{equation*}
\amat{llcl}{
L_\alpha + [A, \cdot]:\, & \A(\T^n_{s+\sigma},\Mat_m(\C)) &\to &\A(\T^n_s,\Mat_m(\C))\\
& F &\mapsto &L_\alpha F + [A,F]},
\label{cohomological3}
\end{equation*}
where here the notation $L_\alpha F$ (or $F'\cdot\alpha$) means that we are applying the Lie derivative operator to each component of the matrix $F$, and $[A,F]$ is the usual commutator.

\begin{lemma}[Straighten the first order dynamics] Let $\alpha\in\R^n$ and $A\in\Mat_m(\R)$ be a diagonalizable matrix satisfying  the Diophantine conditions \eqref{Dio 1} and \eqref{Dio 3} respectively. For every $G\in\A(\T^n_{s+\sigma},\Mat_m(\C))$, such that $\int_{\T^n} G^i_i \, \frac{d\teta}{(2\pi)^n}= 0$, there exists a unique $F\in\A(\T^n_{s},\Mat_m(\C))$, having zero average diagonal elements $\int_{\T^n} F^i_i \, \frac{d\teta}{(2\pi)^n}= 0$, such that the matrix equation \[L_\alpha F + [A,F] = G\] is satisfied; moreover the following estimate holds
\[\abs{F}_s \leq \frac{C_3}{\gamma}\frac{1}{\sigma^{n+\tau}}\abs{G}_{s+\sigma},\]
$C_3$ being a constant depending only on the dimension $n$ and the exponent $\tau$. 
\label{matrix lemma}
\end{lemma}
%(!!!) (Rediger mieux, je suis passée en norme ponderée) Before proceeding with the proof, we need to be a little more precise about $F\in\A(\T^n_{s+\sigma},\Mat_n{\C})$.  As a linear operator, we want the matrix $F(\teta)$ to act on the complex polydisc \[\D^n_s = \set{r\in \C^n: \max_{1\leq j\leq n} \abs{r_j}<s};\] we define then its norm in the usual way \[\abs{F}_{s}=\sup_{\abs{r}\leq 1} \abs{F(\teta) \cdot r}_s,\] the right hand side being the usual weighted norm for vectors that we introduced in the previous subsection. We are now ready to proceed with the proof.
\begin{proof}
Let us start with the diagonal case. Let $A=\operatorname{diag}(a_1,\ldots,a_m)\in\R^m$ be diagonal and $F\in\Mat_m(\C)$ be given, the commutator $[A,F]$ reads
\begin{equation}
\pmat{0 & (a_1 - a_2) F^1_2 & (a_1 - a_3) F^1_3 & \ldots  & (a_1 - a_m) F^1_m \\
        (a_2 - a_1) F^2_1 & 0 & (a_2 - a_3) F^2_3 & \ldots & (a_2 - a_m) F^2_m\\
        \vdots & \vdots & \vdots &\vdots & \vdots\\
        (a_m - a_1) F^m_1 & (a_m - a_2) F^m_2 & \ldots & \ldots & 0 },
        \label{commutator}
        \end{equation}
               where we called $F^i_j$ the element corresponding to the $i$-th line and $j$-th column of the matrix $F(\teta)$. Using components notation, the matrix reads
        \[\pa{[A,F]^i_j} = \pa{(a_i - a_j)F^i_j},\]
           and shows all zeros along the diagonal. 
Adding it now up with the matrix $L_\alpha F$, which reads 
\begin{equation}
\pmat{L_\alpha F^1_1 &\ldots & L_\alpha F^1_m\\
\vdots & L_\alpha F^i_j & \vdots\\
L_\alpha F^n_1 &\ldots & L_\alpha F^m_m },
\label{Lie matrix}
\end{equation}
we see that to solve the equation $L_\alpha F + [A,F] = G$, $G$ being given, we need to solve $n$ equations of the type of Lemma \ref{tangent lemma} and $m^2 - m$ equations 
of the type of Lemma \ref{normal lemma}. Expanding every element in Fourier series, we see that the formal solution is given by a matrix $F$ whose diagonal elements are of the form \[F^j_j= \sum_{k\in{\Z^n\setminus \set{0}}}\frac{G^j_{j,k}}{i k\cdot\alpha}\,e^{i k\cdot\teta},\] while the non diagonal are of the form \[F^i_j= \sum_{k\in{\Z^n}}\frac{G^i_{j,k}}{i k\cdot\alpha + (a_i - a_j)}\,e^{i k\cdot\teta}.\] By conditions \eqref{Dio 1}-\eqref{Dio 3}, via the same kind of calculations we did in the previous lemmata, we get the wanted estimate.

Eventually, to recover the general case, we consider the transition matrix $P\in\GL_m(\C)$ and the equation \[L_\alpha (PFP^{-1}) + P\sq{A,F}P^{-1} = PGP^{-1},\] and observe that we can see $P\sq{A,F}P^{-1}$ as \[P\sq{A,F}P^{-1}= PAP^{-1}PFP^{-1} - PFP^{-1}PAP^{-1}=\sq{PAP^{-1},PFP^{-1}}.\] Letting $\tilde{F}=PFP^{-1}$ and $\tilde{G}=PGP^{-1}$, $\tilde{F}$ satisfies the wanted estimates, and $G=P^{-1}\tilde{G}P$.
\end{proof}
We address the reader looking for optimal estimates to the paper of R\"ussmann \cite{Russmann:1975}.\\ Remark that in case of a real eigenvalue condition \eqref{Dio 2} is redundant. Condition \eqref{Dio 1} suffices, choosing $\gamma < \min_j (\Re a_j)$. 
\subsection{Inversion of the operator $\phi$: estimates on $\phi'^{-1}$ and $\phi''$}
\label{subsection The infinitesimal equation}
The following theorem represents the main result of this first part, from which Moser's theorem \ref{Moser theorem} follows. \\ Let us fix $u^0\in\U_s(\alpha,A)$ and note $\V^{\sigma}_{s+\sigma}=\set{v\in\V : \abs{v - u^0}_s<\sigma}$ the ball of radius $\sigma$ centered at $u^0$.
\begin{thm}
\label{abstract Moser} The operator $\phi$ is a local diffeomorphism in the sense that for every $s<s+\sigma<1$ there exist $\eps>0$ and a unique $C^\infty$-map $\psi$ $$\psi : \V^\eps_{s+\sigma}\to \G_{s}\times\U_s(\alpha,A)\times\Lambda$$ such that $\phi\circ\psi = \id.$ Moreover $\psi$ is Whitney-smooth with respect to $(\alpha,A)$. 
\end{thm}
This result will follow from the inverse function theorem \ref{teorema inversa} and regularity prepositions \ref{smoothness}-\ref{lipschitz}.\\ In order to solve locally $\phi(x)=y$, we use the remarkable idea of Kolmogorov and find the solution by composing infinitely many times the operator \[x = (g,u,\lambda)\mapsto x + \phi'^{-1}(x)(y - \phi(x)),\] on extensions $\normal{T}^n_{s+\sigma}$ of shrinking width.\\At each step of the induction, it is necessary that $\inderiv{\phi}(x)$ exist at an unknown $x$ (not only at $x_0$) in a whole neighborhood of $x_0$ and that $\inderiv{\phi}$ and $\phi''$ satisfy a suitable estimate, in order to control the convergence of the iterates.\\ Thus let us start to check the existence of a right inverse for 
\[\phi'(g,u,\lambda): T_g\mathcal{G}^{\sigma/n}_{s+\sigma}\times\overrightarrow{\mathcal{U}}_{s+\sigma}\times\Lambda\to\mathcal{V}_{g,s},\] if $g$ is close to $\id$. We indicated with $\overrightarrow{\U}$ the vector space directing $\U(\alpha,A)$.
\begin{prop}
\label{proposition infinitesimal}
There exists $\eps_0$ such that if $(g,u,\lambda)$ are in $\G_{s+\sigma}^{\eps_0}\times\U_{s+\sigma}(\alpha,A)\times\Lambda$ then for every $\delta v$ in $\mathcal{V}_{g,s+\sigma}$ there exists a unique triplet $(\delta g,\delta u, \delta \lambda)\in T_g\mathcal{G}_{s}\times\overrightarrow{\mathcal{U}_{s}}\times\Lambda$ such that 
\begin{equation}
\phi'(g,u,\lambda)\cdot (\delta g,\delta u, \delta \lambda) = \delta v;
\end{equation}
\label{linearized equation}
moreover, we have the following estimate
\begin{equation}
\max\,(\abs{\delta g}_s,\abs{\delta u}_s,\abs{\delta\lambda})\leq \frac{C'}{\sigma^{\tau'}}\abs{\delta v}_{g,s+\sigma},
\label{linear bounds}
\end{equation}
where $\tau'>1$ and $C'$ is a constant that depends only on $\abs{\pa{g - \id, u - (\alpha, A\cdot r)}}_{s+\sigma}$.
\end{prop}

\begin{proof}
Let a vector field $\delta v$ in $\V_{g,s+\sigma}$ be given, we want to solve the linearized equation \[\phi'(g,u,\lambda)\cdot (\delta g,\delta u, \delta \lambda) = \delta v,\] where $\delta v$ is the data, and the unknowns are $\delta u \in O(r)\times O(r^2)$, $\delta g$ (geometrically a vector field along $g$) and $\delta\lambda\in\Lambda$. Calculating explicitly the left hand side of the equation, we get 
\begin{equation}
\sq{\push{g}u,\delta g \circ g^{-1}} + \push{g}\delta u + \delta \lambda = \delta v.
\label{explicit linearized}
\end{equation}

Both sides are supposed to belong to $\V_{g,s+\sigma}$; in order to solve the equation we pull it back, obtaining an equation between germs along the standard torus $\text{T}^n_0$ (as opposed to the $g$-dependent torus $g(\text{T}^n_0)$). By naturality of the Lie bracket with respect to the pull-back operator, we thus obtain the equivalent system in $\V_{s+\sigma}$
\begin{equation*}
\sq{u,\pull{g}\delta g \circ g^{-1}} + \delta u + \pull{g} \delta \lambda = \pull{g}\delta v.
\end{equation*}
To lighten the notation we baptize the new terms as \[\dot\lambda:=\pull{g} \delta \lambda,\qquad \dot v:=\pull{g}\delta v,\qquad \dot{g}= \pull{g}\delta g\circ g^{-1}= g'^{-1}\cdot\delta g\] and read
\begin{equation}
\sq{u,\dot{g}} + \delta u + \dot {\lambda} = \dot{v}.
\label{pullback linearized}
\end{equation}
 The unknowns are now $\dot{g}$ (geometrically a germ of vector fields along $\text{T}^n_0$), $\delta u$ and $\dot{\lambda}$; the new infinitesimal vector field of counter terms $\dot{\lambda}$ is no more constant in general, on the other hand, we can take advantage of $u$ in its "straight" form.\\ Let us expand the vector fields along $\T^n_{s+\sigma}\times\set{0}$; we obtain 
\[\eqsys{
u(\teta,r)= \pa{\alpha + u_1(\teta)\cdot r + O(r^2), A\cdot r + U_2(\teta)\cdot r^2 + O(r^3)}\\
\dot {g}(\teta,r) = \pa{\dot\varphi(\teta) , \dot {R}_0(\teta) + \dot {R}_1(\teta)\cdot r} \\
\dot\lambda(\teta,r) = \pa{\dot \lambda_0(\teta) , \dot \Lambda_0(\teta) + \dot \Lambda_1(\teta)\cdot r }\\
\dot v(\teta,r) = \pa{\dot v_0(\teta) + O(r), \dot V_0(\teta) + \dot V_1(\teta)\cdot r + O(r^2)}.}
 \]

We are interested in normalizing the dynamics tangentially at the order zero with respect to $r$, while up to the first order in the normal direction; we then consider the "mixed jet" : 
\[j^{0,1}\dot{v} = \pa{\dot v_0(\teta), \dot V_0(\teta) + \dot V_1(\teta)\cdot r}.\]
Using the expression
\begin{align*}
\sq{u,\dot g} &= \pa{\dot{\varphi}'\cdot\alpha - u_1\cdot \dot{R}_0 + O(r^2)}\base{\teta} +\\ &\pa{\dot{R}'_0\cdot\alpha - A\cdot \dot{R}_0 + ([A,\dot{R}_1] + \dot{R}_1'\cdot\alpha + \dot{R}_0'\cdot u_1 - 2 U_2\cdot \dot{R}_0)\cdot r  + O(r^2)}\base{r}
\end{align*}
and identifying terms of the same order in \eqref{pullback linearized}, yelds
\begin{align}
\label{Teta}
\dot\varphi' \cdot \alpha - u_1\cdot \dot{R_0} &= \dot {v}_0 - \dot{\lambda}_0,\\
\label{R_0}
\dot{R}_0'\cdot \alpha - A\cdot \dot{R}_0 &= \dot{V}_0 - \dot{\Lambda}_0,\\
\label{R_1}
[A,\dot{R}_1] + \dot{R}_1'\cdot\alpha + \dot{R}_0'\cdot u_1 - 2 U_2\cdot \dot{R}_0 &= \dot{V}_1 - \dot{\Lambda}_1,
\end{align}
where the first equation concerns the tangent direction and \eqref{R_0}-\eqref{R_1} the normal direction. This is a triangular system that, starting from \eqref{R_0}, we are able to solve; actually these equations are of the same type as the ones we already solved in Lemmata \ref{tangent lemma}-\ref{normal lemma}-\ref{matrix lemma} (in the sense of their projection on the image of the operator $[u,\dot{g}]$). \\We remark that since $\delta u = (O(r),O(r^2))$, $j^{0,1}\delta u=0$ and $\delta u$ has no contribution to the previous equations. Once we have solved them, we will determine $\delta u$ identifying the reminders.
\begin{rmk}
Every equation contains two unknowns: the components of $\dot g$ and $\dot \lambda$, and the given $\dot v$. We start to solve equations modulo $\dot\lambda$, eventually $\delta\lambda$ will be uniquely chosen to kill the component of the right hand side belonging to the kernel of $[u,\dot{g}]$ (i.e. the constant part of the given terms in \eqref{Teta}-\eqref{R_0}-\eqref{R_1} belonging to the kernel of $A$ and $[A, \cdot]$ respectively), and solve the cohomological equations.
\end{rmk}
Let us proceed with solving the system. We are going to repeatedly apply lemmata \ref{tangent lemma}-\ref{normal lemma}-\ref{matrix lemma} and Cauchy's inequality. Furthermore, we do not keep track of constants - just know that they depend only on $n, \tau>0$ (from the Diophantine condition), $\abs{g - \id}_{s+\sigma}$ and $\abs{(u - (\alpha, A\cdot r))}_{s+\sigma}$ - and hence refer to them as $C$.\\

First, consider \eqref{R_0}. Defining $\bar b = \int_{\T^n} \dot V_0 - \dot\Lambda_0\,\frac{d\teta}{(2\pi)^n}$, we have
\begin{equation*}
\dot{R}_0 = (L_\alpha + A)^{-1} (\dot{V_0} - \dot{\Lambda}_0 - \bar b),  
\end{equation*}
and
\[\abs{\dot{R}_0}_{s} \leq \frac{C}{\gamma}\,\frac{1}{\sigma^{n+\tau}}\abs{\dot{V}_0 - \dot\Lambda_0 }_{s+\sigma}. \]

Secondly, consider equation \eqref{Teta}. Calling the average
\[\bar\beta = \int_{\T^n} \dot{v_0} + u_1\cdot \dot{R}_0 - \dot{\lambda}_0\,\frac{d\teta}{(2\pi)^n}, \]
the solution reads
\[\dot{\varphi}= L_\alpha^{-1} (\dot{v_0} + u_1\cdot \dot{R}_0 - \dot{\lambda}_0 - \bar\beta),\]
with
\[\abs{\dot\varphi}_{s-\sigma}\leq \frac{C}{\gamma}\,\frac{1}{\sigma^{n+\tau}}\abs{\dot{v}_0 + u_1\cdot{\dot{R}}_0- \dot\lambda_0}_{s}.\]
Thirdly, the $\Mat_m(\C)$-valued solution of \eqref{R_1} reads
\[\dot{R}_1 = (L_\alpha + [A,\cdot])^{-1} (\dot{\tilde{V}}_1 + \dot{\Lambda}_1 - \bar{B}),\]
having defined $\dot{\tilde{V}}_1= \dot{V}_1 -\dot{R}_0'\cdot u_1 +2 U_2\cdot \dot{R}_0$, and $\bar{B}$ being the average \[\bar{B}=\int_{\T^n}\dot{ V}_1 -\dot{R}_0'\cdot u_1 +2 U_2\cdot \dot{R}_0 - \dot{B} \,\frac{d\teta}{(2\pi)^n}.\] 
It now remains to handle the choice of $\delta\lambda$ that makes equations average free.
Consider the vector field $\bar\lambda(\teta,r) = (\bar\beta,\bar b + \bar B\cdot r)$, which consequently lays in $\Lambda$, and the map $$ F_g: \Lambda\to \Lambda,\quad \delta\lambda\mapsto - \bar\lambda. $$
When $g=\id$, $F'_{\id} = -\id$. Provided that $g$ stays sufficiently close to the identity, say $\eps_0$-close to the identity in the $\abs{\,\cdot\,}_{s_0}$-norm ($s_0<s<s+\sigma$), $F'$ will be bounded away from $0$.  Note in particular that $-\bar\lambda$ is affine in $\delta\lambda$, the system to solve being triangular of the form $\average{a(g,\dot{v}) + A(g)\cdot\delta\lambda}=0$, with diagonal close to $1$ if the smalleness condition above is assumed, thus there exists a unique $\delta\lambda$ such that $F_g(\delta\lambda)=0$, and $$\abs{\delta\lambda} \leq \frac{C}{\gamma\sigma^{\tilde\tau}}\abs{\dot v}_{s+\sigma},$$
for some $\tilde\tau>1.$
We finally have  
\[\abs{\dot{g}}_{s-2\sigma} \leq \frac{C}{\gamma}\,\frac{1}{\sigma^{\tau '''}}\abs{\delta v}_{g,s+\sigma}. \]
Remembering the definition of $\dot g$ we have $\delta g = g'\cdot \dot g$, hence similar kind of estimates hold for $\delta g$:

\[\abs{\delta g}_{s-2\sigma} \leq \sigma^{-1}(1 + \abs{g - \id}_{s+\sigma}) \frac{C}{\gamma}\,\frac{1}{\sigma^{\tau'''}}\abs{\delta v}_{g,s+\sigma}. \] 
Finally, we see that $\delta u$ is actually well defined in $\overrightarrow{\U}_{s-2\sigma}$ and have \[\abs{\delta u}_{s-2\sigma}\leq \frac{C}{\gamma}\,\frac{1}{\sigma^{\tau'}}\abs{\delta v}_{g,s+\sigma}.\]
Up to defining $\sigma' = \sigma/3$ and $s'= s+\sigma$, the proposition is proved for all indexes $s'$ and $\sigma'$ with $s'<s' + \sigma'$.
\end{proof}
\begin{lemma}[Bounding $\phi ''$]
 The bilinear map
\[\phi ''(x) : (T_g\mathcal{G}_{s+\sigma}^{\sigma/n}\times\overrightarrow{\mathcal{U}}_{s+\sigma}\times\Lambda)^{\otimes 2}\to\mathcal{V}_{s}, \] where $x = (g,u,\lambda)$,
%\[\sq{\sq{\push{g}u,\delta g\circ g^{-1}},\delta g\circ g^{-1}} + 2 \sq{\push{g} \delta u,\delta g\circ g^{-1}} - \sq{\push{g}u, (\delta g'\cdot g'^{-1}\cdot\delta g)\circ g^{-1}}. \] 
satisfies the following estimate \[\abs{\phi''(x)\cdot \delta x^{\otimes 2}}_{g,s}\leq \frac{C''}{\sigma^{\tau ''}} \abs{\delta x}^2_{s+\sigma},\]
$C''$ being a constant depending on $\abs{x}_s$.
\label{lemma derivata seconda}
\end{lemma}

\begin{proof}
For simplicity call $x=(g,u,\lambda)$ and $\delta x = (\delta g,\delta u,\delta\lambda)$. Recall the expression of $\phi'(x)\cdot\delta x= \sq{\push{g}u,\delta g \circ g^{-1}} + \push{g}\delta u + \delta \lambda $. Differentiating again with respect to $x$ yelds
\[\sq{\sq{\push{g}u,\delta g\circ g^{-1}} + \push{g}\delta u, \delta g\circ g^{-1}}  - \sq{\push{g}u, \delta g'\circ g^{-1}\cdot\delta g^{-1}} + \sq{\push{g}\delta u,\delta g\circ g^{-1}}. \]
Since $\delta g^{-1} = - (g'^{-1}\cdot \delta g)\circ g^{-1}$, 
\[\pull{g}\phi''(x)\cdot \delta x^{\otimes 2} = 2\sq{\delta u, \dot g} + \sq{\sq{u,\dot g}, \dot g} + \sq{u,\pull{g} (\delta g'\cdot \inderiv{g}\cdot \delta g)\circ g^{-1}}, \] where the last term simplifies in \[\sq{u,\inderiv{g}\cdot (\delta g'\cdot \inderiv{g}\cdot \delta g)}.\] the wanted bound follows from repeatedly applying Cauchy's inequality, triangular inequality and Lemma \ref{lemma lie brackets}.
% in the appendix and the fact that the product of the modules of two components of a vector is less than the square norm of the vector itself, we get 
%\[\abs{\phi''(x)\cdot\delta x^{\tensor 2}}_{g,s} = \abs{\pull{g} \phi''(x)\cdot\delta x^{\tensor 2}}_s \leq \max\set{A,B,C}\,\frac{1}{\sigma^{2}} \abs{\delta x}^2_{s+\sigma},  \] where $A = \frac{4}{\sigma^2}(1+\frac{1}{e})C^2_A$, $B = \frac{8}{\sigma^2}(1+\frac{1}{e})^2C^2_B$, $C = \frac{2}{\sigma}(1+\frac{1}{e})C^2_C\abs{u}_{s+\sigma}$, $C_i$ being constants depending on $\abs{g}_{s+\sigma}$.
\end{proof}

\subsection{Proof of Moser's theorem} Proposition \ref{proposition infinitesimal} and Lemma \ref{lemma derivata seconda} guarantee to apply Theorem \ref{teorema inversa}, which  provides the existence of $(g,u,\lambda)$ such that $g_\star u + \lambda = v$. Uniqueness and smooth differentiation follow from propositions \ref{lipschitz}, \ref{smoothness} and \ref{Whitney}, once $\abs{v - u^0}_{s+\sigma}$ satisfies the required bound. The only brick it remains to add is the log-convexity of the weighted norm: let $x\in E_s$, to prove that $s\mapsto \log{\abs{x}_s}$ is convex one can easily show that \[\abs{x}_s\leq \abs{x}^{1-\mu}_{s_1}\abs{x}^{\mu}_{s_0},\quad \mu\in[0,1],\ \forall s=(1-\mu)s_1 + s_0\mu\] by H\"older inequality with conjugates $(1-\mu)$ and $\mu$, with the counting measure on $\Z^n$, observing that $\abs{x}_s$ coincides with the $\ell^{1}$-norm of the sequence $(\abs{x_k}e^{\abs{k}s})$. Theorem \ref{abstract Moser} follows, hence theorem \ref{Moser theorem}. 
\section{Hamiltonian systems. Herman's twisted conjugacy theorem}

\label{The Hamiltonian case: Herman's theorem}

The Hamiltonian analogue of Moser's theorem was presented by Michael Herman in a colloquium held in Lyon in 1990. It is also an extension of the normal form theorem of Arnold for vector fields on $\T^n$ (see \cite{Arnold:1961}). \\ In what follows we rely on the formalism developed by Féjoz in his remarkable papers \cite{Fejoz:2012, Fejoz:2004, Fejoz:2015}. This frame will be also used in section \ref{Hamiltonian-dissipative section}, for generalizing Herman's result. Vector fields will be defined on $\T^n\times\R^n$. As always the standard identification $\R^{n\ast}\equiv \R^n$ will be used.
\subsection{Spaces of vector fields}
Let $\mathcal{H}$ be the space of germs of real analytic Hamiltonians defined on some neighborhood of $\normal{T}^n_0 = \T^n\times\set{0}\subset\T^n\times\R^n$, and $\mathcal{V}^{\text{Ham}}$ the corresponding set of germs along $\normal{T}^n_0$ of real analytic Hamiltonian vector fields.\\ In this and the following sections we will only need to consider the standard Diophantine condition \eqref{Dio 1}, for some $\gamma,\tau>0$.\\ Fixing $\alpha\in \mathcal{D}_{\gamma,\tau}	\subset \R^n$, consider the following affine subspace of $\mathcal{H}$, $$\mathcal{K}^{\alpha}=\set{K\in\mathcal{H}: K(\teta,r)= c + \alpha\cdot r + O(r^2)}.$$ 
$\mathcal{K}^{\alpha}$ is the set of Hamiltonians $K$ for which $\normal{T}^n_0$ is invariant by the flow $u^K$ and $\alpha$-quasi-periodic:
\begin{equation}
u^{\text{K}}=\eqsys{\dot\teta = \frac{\partial K}{\partial r}(\teta,r)= \alpha + O(r)\\
\dot{r} = -\frac{\partial K}{\partial \teta}(\teta,r)= O(r^2).} 
\end{equation}
We define
 $$\mathcal{U}^{\text{Ham}}(\alpha,0)= \set{u^{\text{K}}\in\mathcal{V}^{\text{Ham}} : K\in \mathcal{K}^\alpha}$$ 
and introduce the set of counter terms $$\Lambda^{\operatorname{Ham}}=\set{\lambda\in\V^{\operatorname{Ham}}: \lambda(\teta,r)= (\beta,0)} \equiv \R^n.$$ We define the complex extension of width $s$ of $\T^n\times\R^n$  as in section \ref{complex extensions}, and note $\mathcal{H}_s = \mathcal{A}(\T^n_s)$ the space of Hamiltonians defined on this extension.  $\mathcal{K}^\alpha_s$ is the affine subspace consisting of those $K\in\mathcal{H}_s$ of the form $K(\teta,r)= c + \alpha\cdot r + O(r^2)$.

\subsection{Spaces of conjugacies}

Let $\mathcal{D}^\sigma_s$ be the space real holomorphic invertible maps  $\varphi = \id + v : \T^n_s\to\T^n_\C$, fixing the origin with \[\abs{v}_s = \max_{1\leq j\leq n}(\abs{v_j}_s)<\sigma.\]
 We consider the contragredient action of $\mathcal{D}_s$ on $\normal{T}^n_s$, with values in $\normal{T}^n_\C$: \[\varphi(\teta,r) = (\varphi(\teta), \,^t\varphi'^{-1}(\teta)\cdot r).\] This is intended to linearize the dynamics on the tori.\\ Let $\mathcal{B}^\sigma_s$ be the space of exact $1$-forms $\rho(\teta)= dS(\teta)$ on $\T^n_s$ ($S$ being a map $\T^n_s\to \C$, vanishing at the origin) such that $$\abs{\rho}_s = \max_{1\leq j\leq n} (\abs{\rho_j}_s)<\sigma;$$ we hence consider the space $\G^{{\Ham},\sigma}_{s}= \mathcal{D}^\sigma_s\times \mathcal{B}^\sigma_s$ of those Hamiltonian transformations $g=(\varphi,\rho)$ acting this way $$g(\teta,r)= (\varphi(\teta),\,^t\varphi'^{-1}(\teta)\cdot (r + \rho(\theta))),$$ that is identified, locally in the neighborhood of the identity, to an open set of the affine space passing through the identity and directed by $\set{(\varphi - \id), S}$. 
The form $\rho= dS$ being exact, it doesn't change the cohomology class of the torus.\footnote{In this work we indicated derivations sometimes by "$\,'\,$", "$\operatorname{d}$" or "$\operatorname{D}$" to avoid heavy notations.}

The tangent space at the identity of $\G^{\Ham}_s$, $T_{\id}\G^{\Ham}_s= \chi_s\times \mathcal{B}_s$ is endowed with the norm $$\abs{\dot g}_s = \max (\abs{\dot\varphi}_s,\abs{\dot\rho}_s).$$

\comment{\subsection{Spaces of conjugacies} 
Let $\mathcal{D}_s$ be the space of maps $$\varphi = \id + v \in\A(\T^n_s,\T^n_\C),$$ fixing the origin. \\ We consider the contragredient action of $\mathcal{D}_s$ on $\normal{T}^n_s$, with values in $\normal{T}^n_\C$: \[\varphi(\teta,r) = (\varphi(\teta), \,^t\varphi'^{-1}(\teta)\cdot r).\] This is intended to linearize the dynamics on the tori.\\
Let $\mathcal{B}_s$ be the space of exact complex valued $1$-forms $\rho$ on $\T^n_s$.\\
We define $\G^{\Ham}_s=\mathcal{D}_s\times\mathcal{B}_s$ and identify it with the space of exact symplectomorphisms \footnote{For "exact symplectomorphism" we mean a symplectic $g$ such that $\pull{g}\lambda - \lambda$ is exact, $\lambda(\teta,r)=\sum_{j=0}^n r_jd\teta_j$ being the fundamental $1$-form of Liouville on $\T^n\times\R^n$} $$\G^{\operatorname{Ham}}_s=\set{g\in\G_s: g(\teta,r)=(\varphi(\teta),\,^t\varphi'^{-1}(\teta)\cdot(r + \rho(\teta))}.$$ The form $\rho= dS$ being exact ($S:\T^n_s\to \C$), it doesn't change the cohomology class of the torus.\footnote{In this work we indicated derivations sometimes by "$\,'\,$", "$\operatorname{d}$" or "$\operatorname{D}$" to avoid heavy notations.}
}

\begin{thm}[Herman] 
\label{Herman Hamiltonian}
Let $\alpha\in \mathcal{D}_{\gamma,\tau}$ and $K^0 \in \mathcal{K}_{s+\sigma}^\alpha$. If $H\in\mathcal{H}_{s+\sigma}$ is close enough to $K^0$, there exists a unique $(K, g, \beta)\in \mathcal{K}_s^\alpha \times \G^{\operatorname{\Ham}}_s \times \Lambda^{\operatorname{\Ham}}$ close to $(K^0,\id,0)$ such that $$ H = {K}\circ g + \beta\cdot r. $$
\end{thm}
Here too, the presence of $\beta\cdot r$ breaks the dynamical conjugacy between $H$ and $K$: the orbits of $K's\in\mathcal{K}^\alpha$ under the action of diffeomorphisms of $\G^{\Ham}$, form a subspace of codimension $n$. \\
For a proof of this result, known also as "twisted conjugacy theorem", see \cite{Fejoz:2010}, and \cite{Fejoz:2004} for an analogue in the context of Hamiltonians with both tangent and normal frequencies.\\
Phrased in terms of vector fields, the theorem becomes
\begin{thm}[Herman] If $v^{\h}$ is close enough to $u^{K^0}\in\U^{\Ham}_s(\alpha,0)$, there exists a unique $(g,u^K,\beta)\in\G^{\Ham}_s\times\U_s^{\Ham}(\alpha,0)\times\Lambda^{\Ham},$ close to $(\id,u^{K_0},0)$ such that $$\push{g}u^{K} + \beta\,{\partial_\theta} = v^{\h}.$$
\end{thm}

\section{Hamiltonian-dissipative systems. Generalization of Herman's theorem and translated tori à la R\"ussmann}
\label{Hamiltonian-dissipative section}

\subsection{A generalization of Herman's theorem}
\label{dissipative Herman}
Here we generalize to a particular class of dissipative vector-fields the theorem of Herman.
\subsubsection{Spaces of vector fields}
Let $\mathcal{H}_s = \A(\normal{T}^n_s)$ and $\V^{\Ham}_s$ the  space of Hamiltonian vector fields corresponding to Hamiltonians $H's\in\mathcal{H}_s$. Let now $\eta\in\R$ be a fixed constant and let extend $\V^{\Ham}_s$ as $\pa{\V^{\Ham}\oplus(-\eta r\partial_r)}_s$. The corresponding affine subspace becomes $$\U^{\Ham}_s(\alpha,-\eta) = \set{u\in\pa{\V^{\Ham}\oplus(-\eta r\partial_r)}_s: u(\teta,r)= (\alpha + O(r), -\eta r + O(r^2))}.\footnote{We recall that the notation $r\partial_r$ is a shortcut for $\sum_j^n r_j\partial_{r_j}$.}$$ When $\eta>0$ (resp. $\eta<0$) the invariant quasi-periodic torus $\text{T}^n_0$ of $u$ is $\eta$-normally attractive (resp. repulsive).  \\ The class $\V^{\Ham}\oplus(-\eta r\partial_r)$ is mathematically peculiar: it is invariant under the Hamiltonian transformations in $\G^{\Ham}$. Physically, the described system undergoes a constant linear friction (resp. amplification) which is the same in every direction.

According to theorem \ref{theorem well def} and corollary \ref{cor well def}, the operators 
\begin{equation}
\phi : \G^{{\Ham},\sigma^2/2n}_{s+\sigma}\times\Uh_{s+\sigma}({\alpha,-\eta})\times\Lambda^{\Ham}\to \Vh_s,\quad (g,u,\beta)\mapsto \push{g}u + \beta\partial_\teta,
\label{operator Herman}
\end{equation}
 commuting with inclusions, are well defined.
\begin{thm}["Dissipative Herman"]
\label{theorem Herman}
If $v\in\pa{\Vh\oplus(-\eta r\partial_{r})}_{s+\sigma}$ is sufficiently close to $u^0\in\Uh_{s+\sigma}(\alpha,-\eta)$, for any $\eta\in [-\eta_0,\eta_0], \eta_0\in\R^+$,  there exists a unique $(g,u,\beta)\in \G^{\operatorname{Ham}}_s\times\Uh_s(\alpha,-\eta)\times\Lambda^{\Ham}$, close to $(\id,u^0,0),$ such that \[\push{g}u + \beta\partial_{\teta} = v.\]
\end{thm}

The key point relies on the following two technical lemmata.

\begin{lemma}
\label{lemma Hamiltonian1}
If $g\in\G^{\Ham}$ and $v\in\Vh\oplus(-\eta r\partial_r)$, the vector field $\push{g}v$ is given by 
\begin{equation}
\push{g}v = \eqsys{\dot{\Theta} = \frac{\partial \hat{H}}{\partial R}\\
\dot{R} = -\frac{\partial \hat{H}}{\partial\Theta} - \eta R,}
\label{tesi lemma Hamiltonian1}
\end{equation}
where $$\hat{H}(\Theta,R) = H\circ g^{-1}(\Theta,R) - \eta \pa{S\circ\varphi^{-1}(\Theta)}.$$
\end{lemma}
The fact that $\eta\in\R$ is fundamental to maintain the Hamiltonian structure, which would be broken even if $\eta$ was a diagonal matrix. Geometrically, the action of $g$ on $H$ is "twisted" by the dissipation.
\begin{proof}
$g(\teta,r)= (\Theta,R)$, that is, 
\begin{equation*}
\eqsys{\Theta = \varphi(\teta)\\
R =\, ^t\varphi'^{-1}(\teta)\cdot(r + dS(\teta)).}
\end{equation*}
We have 
\begin{itemize}[leftmargin=*]
\item in the tangent direction 
\[\dot\Theta = \varphi'(\teta)\cdot \dot\teta = \frac{\partial (H\circ g^{-1})}{\partial R}.\]
\item The derivation of $\dot R$ requires a little more attention: 
\begin{align*}
\dot{R} &= \underbrace{(\,^t \varphi'^{-1}(\teta))' \cdot r \cdot \dot\teta}_A + \underbrace{\,^t\varphi'^{-1}(\teta)\cdot \dot r}_B + \underbrace{\,^t \inderiv{\varphi}(\teta)\cdot D^2 S(\teta)\cdot  \dot \teta}_C \\  
& + \underbrace{(\,^t \varphi'^{-1}(\teta))'\cdot dS(\teta)\cdot\dot{\teta}}_D
\end{align*}
where, expanding and composing with $g^{-1}$
\begin{align*}
A &= \pa{-\,^t\varphi'^{-1}\cdot\,^t\varphi''\cdot\,^t\varphi'^{-1}}\circ\varphi^{-1}(\Theta)\cdot(\,^t\varphi'\circ\varphi^{-1}(\Theta)\cdot R - dS\circ\varphi^{-1}(\Theta))\cdot \ph{r}\\
B & = -\,^t\inderiv{\varphi}\circ\varphi^{-1}(\Theta)\cdot \ph{\teta} - \eta R + \eta\,^t\inderiv{\varphi}\circ\varphi^{-1}(\Theta)\cdot dS\circ\varphi^{-1}(\Theta) \\
C &= \,^t\inderiv{\varphi}(\teta)\cdot D^2 S(\teta)\cdot \ph{r} \\
 &= \,^t\inderiv{\varphi}\circ\varphi^{-1}(\Theta)\cdot D^2S\circ\varphi^{-1}(\Theta)\cdot\ph{r}\\
D &= -\pa{\,^t\varphi'^{-1}\cdot\,^t\varphi''\cdot\,^t\varphi'^{-1}}\circ\varphi^{-1}(\Theta)\cdot dS\circ\varphi^{-1}(\Theta)\cdot \ph{r}\\
\end{align*}
\end{itemize}
Remark that if $$H\circ g^{-1}(\Theta,R) = H\pa{\varphi^{-1}(\Theta),\,^t\varphi'\circ\varphi^{-1}(\Theta)\cdot R - dS\circ\varphi^{-1}(\Theta)},$$
we have
\begin{align*}
\ph{\Theta} &= \ph{\teta}\cdot \inderiv{\varphi}\circ\varphi^{-1}(\Theta)\\
            & + \ph{r}\cdot \sq{\,^t\varphi''\circ\varphi^{-1}(\Theta)\cdot\inderiv{\varphi}\circ\varphi^{-1}(\Theta)\cdot R - D^2S\circ\varphi^{-1}(\Theta)\cdot\inderiv{\varphi}\circ\varphi^{-1}(\Theta)}.
\end{align*}
Summing terms we get  $$\dot{R} = -\frac{\partial H\circ g^{-1}}{\partial \Theta} - \eta R + \eta \pa{\,^t\varphi'^{-1}\circ\varphi^{-1}(\Theta)\cdot dS\circ\varphi^{-1}(\Theta) }.$$\medskip 
Introducing the modified Hamiltonian  $\hat{H}$ as in the statement, the transformed system has the claimed form \eqref{tesi lemma Hamiltonian1}.  
\end{proof}
The same is true for the pull-back of such a $v$:
\begin{lemma}
\label{lemma Hamiltonian2}
If $g\in\G^{\operatorname{Ham}}$ and $v\in\V^{\Ham}\oplus(-\eta r\partial_r)$, the vector field $\pull{g}v=\push{g^{-1}}v$ is given by 
\begin{equation}
\pull{g}v= \eqsys{\dot{\teta} = \frac{\partial \hat{H}}{\partial r}\\
\dot{r} = -\frac{\partial \hat{H}}{\partial\teta} - \eta r ,}
\end{equation}
$\hat{H}$ being $\hat{H}(\teta,r)= H\circ g\, (\teta,r) + \eta S(\teta)$.
\end{lemma}

\subsubsection{The linearized problem}
Theorem \ref{dissipative Herman} will follow - again - from the inverse function theorem \ref{teorema inversa}, once we check the existence of a right (and left) inverse for $\phi'$ and bounds on it and $\phi''$. \\ Except from a minor difference, the system that solves the linearized problem is the same as the one in the purely hamiltonian context. 
\begin{prop}
\label{proposition infinitesimal ham}
There exists $\eps_0$ such that if $(g,u,\beta)$ is in $\G^{\Ham,\eps_0}_{s+\sigma}\times\U^{\Ham}_{s+\sigma}(\alpha,-\eta)\times\Lambda^{\Ham}$, then for every $\delta v$ in $\pa{\mathcal{V}^{\Ham}\oplus(-\eta r\partial_r)}_{g,s+\sigma}$ there exists a unique triplet $(\delta g,\delta u, \delta \beta)\in T_g\mathcal{G}^{\Ham}_{s}\times\overrightarrow{\mathcal{U}_s}(\alpha,-\eta)\times\Lambda^{\Ham}$ such that 
\begin{equation}
\phi'(g,u,\beta)\cdot (\delta g,\delta u, \delta \beta) = \delta v;
\label{linearized equation ham}
\end{equation}
moreover, we have the following estimate
\begin{equation}
\max\,(\abs{\delta g}_s,\abs{\delta u}_s,\abs{\delta\beta})\leq \frac{C}{\sigma^{\tau'}}\abs{\delta v}_{g,s+\sigma},
\label{linear bounds ham}
\end{equation}
where $\tau'>0$ and $C$ is a constant that depends only on $\abs{g - \id}_{s+\sigma}$ and $\abs{u - (\alpha, -\eta r)}_{s+\sigma}$.
\end{prop}

\begin{proof}
The proof is recovered from the one of proposition \ref{proposition infinitesimal}, additionally imposing that the transformation is Hamiltonian and the vector fields belong to this particular class "Hamiltonian + dissipation". The  interesting fact relies on the homological equation intended to "relocate" the torus. \\ Calculating $\phi'(x)\cdot\delta x$ and pulling back equation \eqref{linearized equation ham} we get \[\sq{u,\dot g} + \delta u = \dot v - \dot\lambda,\] here $\dot g$ has the form $\dot g = (\dot\varphi, -r\cdot\dot\varphi' + \dot\rho),$ where $\dot\varphi\in \chi_{s}$ and $\dot\rho = d\dot S\in\mathcal{B}_{s}$. The system to solve translates in
\begin{align*}
\dot{\varphi}'\cdot\alpha - u_1\cdot d\dot{S} &= \dot v_0^{H} -\dot\lambda_0 ,\\
d\dot S'\cdot \alpha +\eta d\dot S &= \dot{V}_0^{H} - \dot\Lambda_0,\\
 - ^t D\dot\varphi' \cdot \alpha + \,^t D(u_1\cdot d\dot S) &= \dot{V}_1^{H} - \dot\Lambda_1,
\end{align*}
where $\dot\lambda_0 = \varphi'^{-1}\cdot\delta\beta$, $\dot\Lambda_0 = -\partial_\teta (^t\varphi'^{-1}\cdot\rho(\teta))\cdot\delta\beta$ and $\dot\Lambda_1 =-\, ^t\dot\lambda_0'$. \\
Thanks to lemma \ref{lemma Hamiltonian2}, the right hand sides consist of Hamiltonian terms, normal directions are of $0$-average and, according to the symmetry of a Hamiltonian system, just the first two equations are needed to solve the whole systems, as the third one (corresponding to the coefficient of the linear term of the $\dot r$-component) is the transpose of the $\teta$-derivative of the first, with opposite sign.\\Coherently, the term $\dot\Lambda_0$ has 0-average and the $d\dot S$-equation can readily be solved.
\begin{rmk}\label{remark uniform} The fact that $d\dot S$ has zero average implies that $$d\dot S (\teta)= 0 + \sum_{k\neq 0} \frac{\dot V_{k}}{i\,k\cdot\alpha + \eta}\,e^{i\,k\cdot\teta}.$$ Hence, when passing to norms on the extended phase space, we can bound the divisors uniformly with respect to $\eta$, since $\abs{i\,k\cdot\alpha + \eta}>\abs{i\,k\cdot\alpha}$; we just need the standard Diophantine condition \eqref{Dio 1}. This will imply that the limit distance $\abs{v - u^0}_{s + \sigma}<\eps$ entailed in theorem \ref{teorema inversa}, will be defined for any $\eta$ varying in some interval containing $0$ ($\eps$ would depend on $\eta$ though $\gamma$ of the Diophantine condition \eqref{DC}, which appears in $C'$ in the bound of $\phi'^{-1}$). This is fundamental for the results in the last section.
\end{rmk}
Solutions and inequalities follow readily from lemmata \ref{tangent lemma}-\ref{normal lemma} and Cauchy's inequality.
\end{proof}
\begin{rmk} {The system above is the one that solves, when $\eta=0$, the infinitesimal problem of the "twisted conjugacy" theorem presented in \cite[\S 1.1]{Fejoz:2010}. Hence, up to the slight difference in the equation determining $d\dot S$, the proof of theorem \ref{theorem Herman} follows the same steps and difficulties as in \cite{Fejoz:2010} (application of theorem \ref{teorema inversa} in the frame of remark \ref{remark polynomial}).}
\end{rmk}
\subsubsection{A first portrait}
\label{general dissipative}
If the eigenvalues $a_i$ of $A$ are all distinct and different from $0$, it is immediate to see that the external parameters are of the form $\lambda = (\beta, B\cdot r)$, with $B$ a diagonal matrix as well (remember lemma \ref{matrix lemma}).
\begin{cor} Let $A\in\Mat_m(\R)$ be diagonalizable with simple, non $0$ eigenvalues. If $v$ is sufficiently close to $u^0\in\U(\alpha,A)$, there exists a unique $(g,u,\lambda)\in \G\times\U(\alpha,A)\times\Lambda(\beta, B\cdot r)$, close to $(\id, u^0,0)$, such that \[\push{g}u + \lambda = v,\] $\lambda$ being of the form $\lambda = (\beta, \operatorname{diag}B\cdot r)$, $B$ being diagonal.
\end{cor}
Here a diagram that summarizes our results, from the most general to the purely Hamiltonian one. We emphasizes the parameters in the notation of $\Lambda$. \medskip
\begin{equation*}
\xymatrix{\text{Moser:} & \G\times\U(\alpha,A)\times\Lambda(\beta, b + B\cdot r)\ar[r]_-{\simeq\,\text{loc.}} & \V\\
\text{General dissip. ($\operatorname{diag}A$):} & \G\times\U(\alpha,A)\times\Lambda(\beta, \operatorname{diag} B\cdot r)\ar[r]_-{\simeq\,\text{loc.}} & \V\\
\text{Herman dissip.:} & \G^{\Ham}\times\U^{\Ham}(\alpha,\eta )\times\Lambda(\beta,0)\ar@{^{(}->}[u]\ar[r]_-{\simeq\,\text{loc.}} & \Vh\oplus (\eta r\,\base{r})\ar @{^{(}->}[u]\\
\text{Herman ($\eta=0$):} & \G^{\Ham}\times\U^{\Ham}(\alpha,0)\times\Lambda(\beta,0)\ar[r]_-{\simeq\,\text{loc.}} & \Vh}
\end{equation*}

\subsection{Normal form "à la R\"ussmann"}
\label{Russmann section vector fields}
In the context of the diffeomorphisms of the cylinder $\T\times\R$,  R\"ussmann proved a result that admits among the most important applications in the study of dynamical systems: the "theorem of the translated curve" (for the statement see \cite{Russmann:1970}, or \cite{Yoccoz:Bourbaki} for instance).\\ We give here an extension to vector fields of this theorem.\\
If hamiltonians considered above are non degenerate (see below), we can define a "hybrid normal form" that both relies on the peculiar structure of the vector fields and this torsion property; this makes unnecessary the introduction of \co{all} the counter terms a priori needed if we would have attacked the problem in the pure spirit of Moser.
\subsubsection{Twisted vector fields}
\label{objects spin}
The starting context is the one of section \ref{dissipative Herman} and notations are the same. \\ We are interested in those $K\in\mathcal{K}^\alpha$ of the form 
\begin{equation}
K(\teta,r)= c + \alpha\cdot r + \frac{1}{2}Q(\teta)\cdot r^2 + O(r^3),
\end{equation} 
$Q$ being a non degenerate quadratic form on $\T^n_s$ : $\det\frac{1}{(2\pi)^n}\int Q(\teta)\,d\teta \neq 0$.\\
There exist $s_0$ and $\eps_0$ such that $\forall s>s_0$, $K^0\in\mathcal{H}_s$ and for all $H\in\mathcal{H}_s$ such that $\abs{H - K^0}_{s_0}<\eps_0$ one has \[\abs{\det\int_{\T^n}\frac{\partial^2 H}{\partial r^2}(\teta,0)\,\frac{d\teta}{(2\pi)^n}}\geq \frac{1}{2}\abs{\det\int_{\T^n}\frac{\partial^2 K^0}{\partial r^2}(\teta,0)\,\frac{d\teta}{(2\pi)^n}}\neq 0.\]
We assume that $s\geq s_0$ and define $$\mathcal{K}^\alpha_s = \set{K\in \mathcal{K}^\alpha_s: \abs{K - K^0}_{s_0}\leq \eps_0}.$$
We hence consider the corresponding set of vector fields 
\begin{equation}
\Uh_s(\alpha,0) = \set{u^{K} (\teta,r)= (\alpha + \frac{1}{2}Q(\teta)\cdot r + O(r^2), O(r^2))}, 
\end{equation}
affine subset of $\Vh_s=\set{v^{\h} \text{ vector fields along $\normal{T}^n_s$}}$.\\
Now, fix $\eta\in\R$ and considered the extended spaces 
\begin{equation}
\Uh_s(\alpha,-\eta) \text{ and }\, \pa{\Vh\oplus (-\eta r + \eta\R^n){\partial_r}}_s.
\end{equation}

\begin{rmk} We enlarged the target space with the translations in actions $$ \zeta \mapsto v^{\h}\oplus (-\eta r + \eta\zeta)\partial_r$$ in order to handle symplectic transformations and guarantee the well definition of the normal form operator (see below). Note that the constant $\eta$ multiplying $\zeta$ is unessential; it just lighten notations in calculations and make results below ready-to-use for the application presented in section \ref{section A KAM result}. 
\end{rmk}

Like in the previous section, $\mathcal{D}^\sigma_s$ is the space holomorphic invertible maps $\varphi= \id + v : \T^n_s\to\T^n_\C$, fixing the origin with $\abs{v}_s<\sigma$, while $\mathcal{Z}^\sigma_s$ the space of closed $1$-forms $\rho(\teta)= dS(\teta) + \xi$ on $\T^n_s$ (which we see as maps $\T^n_s\to \C^n$) such that $$\abs{\rho}_s := \max (\abs{\xi}, \abs{dS}_s)<\sigma,$$ we consider the set $\G^{\omega,\sigma}_{s}= \mathcal{D}^\sigma_s\times \mathcal{Z}^\sigma_s$ of those symplectic transformations $g=(\varphi,\rho)$ of the form $$g(\teta,r)= (\varphi(\teta),\,^t\varphi'^{-1}(\teta)\cdot (r + dS(\theta) + \xi)).$$ \\
The tangent space at the identity $T_{\id}\G^{\omega}=\chi_s\times\mathcal{Z}_s$ is endowed with the norm \[\abs{\dot g}_s = \max(\abs{\dot\varphi}_s,\abs{\dot\rho}_s).\] 
Concerning the space of constant counter terms we define the space of translations in action as $$\Lambda(0,b) = \set{ \lambda = (0, b), b\in\R^n}.$$
According to the following lemmata and corollary \ref{cor well def}, the normal form operators (commuting with inclusions) 
\begin{equation}
\begin{aligned}
\phi: \G^{\omega,\sigma^2/2n}_{s+\sigma}\times\Uh_{s+\sigma}(\alpha,-\eta)\times\Lambda(0,b) &\to \pa{\Vh\oplus(-\eta r + \eta\R^n)\partial_r}_s, \\ (g,u,\lambda) &\mapsto \push{g}u + b
\end{aligned}
\label{normal Russmann}
\end{equation} 
are well defined.
\begin{lemma} If $g\in\G^\omega$ and $v\in\V^{\Ham}\oplus\pa{(-\eta r + \eta\zeta)\partial_r)}$, the push forward $\push{g}v$ is given by \[\push{g}v=\eqsys{\dot\Theta = \frac{\partial \hat{H}}{\partial R}\\ \dot R = -\frac{\partial\hat{H}}{\partial\Theta} - \eta(R - \hat\zeta),\quad \hat{\zeta}= \zeta + \xi}\] where $\hat{H}(\Theta, R)= H\circ g^{-1} - \eta (S\circ\varphi^{-1}(\Theta) + \hat{\zeta}\cdot(\varphi^{-1}(\Theta) - \Theta))$. 
\label{lemma symplectic0}
\end{lemma}
The proof is the same as for lemma \ref{lemma Hamiltonian1}, taking care of the additional term $\eta \,^t \varphi'^{-1}\circ\varphi^{-1}\cdot(\xi + \zeta)$ coming from the non exactness of $\rho(\teta)$ and the translation $\zeta$.

\comment{
Concerning the pull-back intervening in the equation of $\phi'^{-1}$, we have the following
\begin{lemma}
\label{lemma symplectic}
If $g\in\G^\omega$ and $v\in\V^{\Ham}\oplus{(-\eta r )\partial_r}$, the vector field $\pull{g}v$ is given by 
\begin{equation}
\eqsys{\dot{\teta} = \frac{\partial \hat{H}}{\partial r}\\
\dot{r} = -\frac{\partial \hat{H}}{\partial\teta} - \eta (r +\xi),}
\end{equation}
with $\hat{H}(\teta,r)= H\circ g (\teta,r) + \eta S(\teta)$.
\end{lemma}
Let now consider the parametrization $$\R^n\ni\zeta \mapsto v^H\oplus (-\eta r + \eta\zeta)\partial_r.$$ 

}

\begin{lemma}
\label{lemma Hamiltonian3}
The pull back of $v=v^H\oplus (-\eta r + \eta\zeta)\partial_r$ by a symplectic transformation $g\in\G^\omega$ reads 
\begin{equation}
\pull{g}v = \eqsys{\dot{\teta} = \frac{\partial \hat{H}}{\partial r}\\
\dot{r} = -\frac{\partial \hat{H}}{\partial\teta} - \eta (r - \hat\zeta),\quad \hat\zeta = \zeta - \xi, }
\end{equation}
where $\hat{H}(\teta,r) = H\circ g (\teta,r) + \eta (S(\teta) - \zeta\cdot (\varphi(\teta) - \teta))$.
\end{lemma}
The proof of these results are immediate from the definition of $g$ and follow the one of lemma \ref{lemma Hamiltonian1}.

\begin{thm}[Translated torus]
\label{theorem Russmann}
If $v=v^{\h}\oplus ((-\eta r + \eta\zeta)\partial_r)$ is sufficiently close to $u^0\in\U^{\Ham}(\alpha,-\eta)$, for any $\eta\in [-\eta_0,\eta_0]$, $\eta_0\in\R^{+}$, there exists a unique $(g,u,b)\in \G^{\omega}\times\Uh(\alpha,-\eta)\times\Lambda(0,b)$, close to $(\id,u^0,0)$, such that \[\push{g}u + b\,\partial_{r} = v.\]
 \end{thm}
From the normal form, the image $g(\normal{T}^n_0)$ is not invariant by $v$, but translated in the action direction during each infinitesimal time interval. \\The proof can still be recovered from the inverse function theorem \ref{teorema inversa} (in the frame of remark \ref{remark polynomial}) and propositions \ref{lipschitz}-\ref{smoothness}.
\begin{proof}
The main part consists in checking the invertibility of $\phi'$. Let \[\phi : \G^{\omega,\sigma^2/2n}_{s+\sigma}\times \Uh_{s+\sigma}(\alpha,-\eta)\times\Lambda(0,b)\to ({\Vh\oplus (-\eta r + \eta\R^n)})_s,\] \[(g,u,b)\mapsto \push{g}u + b = v\] and the corresponding $$\phi'(g,u,b): (\delta g,\delta u,\delta b)\mapsto [\push{g}u, \delta g\circ g^{-1}] + \push{g}\delta u + \delta b$$ defined on the tangent space be given. As in proposition \ref{proposition infinitesimal}, we pull it back and expand vector fields along $\normal{T}^n_0$. \\ In this context $$\dot g = g'^{-1}\cdot\delta g = (\dot\varphi, -\,^t\dot\varphi'\cdot r + d\dot S + \dot\xi),$$ with $\dot S\in\mathcal{A}(\T^n_s)$, $\dot\varphi\in\A(\T^n_s,\C^n)$, $\dot\xi\in\R^n.$ 
\begin{align}
\label{varphi}
\dot{\varphi}'\cdot\alpha - Q(\theta)\cdot(d\dot{S} + \dot\xi) &= \dot v_0^{H} ,\\
\label{Ds}
d\dot S'\cdot \alpha +\eta (d\dot S + \dot\xi) &= \dot{V}_0^{H} + \eta\widehat{\delta\zeta} - \dot{b},\\
\label{terza}
 - \,^t D\dot\varphi' \cdot \alpha + \,^t D(Q(\teta)\cdot(d\dot S + \dot\xi)) &= \dot{V}_1^{H},
\end{align} 
where $\dot b$ is of the form $^t\varphi'\cdot \delta b = (\id + ^t v')\cdot\delta b$ (remember that $\varphi = \id + v$). As always we wrote "H" to emphasize the Hamiltonian nature of terms.\\
We are now going to repeatedly apply lemmata \ref{tangent lemma}, \ref{normal lemma} and Cauchy's estimates. As before we do not keep track of constants.
\begin{itemize}[leftmargin=*]
\item Note that, averaging the second equation on the torus, we can determine $$\delta b = \eta (\widehat{\delta\zeta} - \dot\xi),$$ hence solve the average free $$d\dot S'\cdot \alpha + \eta d\dot S = \dot V^{\h}_0 - \,^t v'\cdot\delta b.$$ Denoting $\dot V_0 = \dot V_0^{\h} - \eta ^t v'\cdot\widehat{\delta\zeta}$, the solution can be written as 
\begin{equation}
d\dot S (\teta) = \sum_k \frac{\dot V_{0,k}}{i\,k\cdot\alpha + \eta}e^{i\,k\teta} + \eta M(\teta)\cdot\dot\xi,
\label{dS rus}
\end{equation}
 where $M(\teta)$ is the matrix  whose $(ij)$ component reads $(\sum_k \frac{^t v'^i_{j,k}}{i\,k\cdot \alpha + \eta}\,e^{i\,k\cdot\teta})$. In particular by $\abs{i\,k\cdot\alpha + \eta}\geq \abs{\eta}$, we have $\eta\abs{M}_s\leq {n\abs{v}_{s+\sigma}}/{\sigma}$, which will remain small in all the iterates, not modifying the torsion term (see below).\\ The Fourier coefficients smoothly depend on $\eta$ and remark \ref{remark uniform} holds.
\item Call $S_0$ the first part of \eqref{dS rus}, averaging on the torus equation \eqref{varphi}, and thanks to the torsion hypotheses, we determine 
\begin{equation}
\dot\xi = -\pa{\frac{1}{(2\pi)^n}\media{Q\cdot (\eta M + \id)\,d\teta}}^{-1}\cdot \pa{\frac{1}{(2\pi)^n}\media{\dot v_0 + Q\cdot S_0\,d\teta}},
\end{equation}
and have $$\abs{\dot\xi} \leq \frac{C}{\sigma^{\tau+n}} \abs{\delta v}_{g,s+\sigma},$$ hence 
\begin{equation}
\abs{d\dot S}_{s}\leq \frac{C}{\gamma\sigma^{\tau + n}}\abs{\delta v}_{g,s+\sigma}\quad \text{and}\quad \abs{\delta b}\leq \frac{C}{\gamma\sigma^{\tau+n}}\abs{\delta v}_{g,s+\sigma}.
\end{equation} 
\item There remains to solve equation \eqref{varphi}; since 
\begin{equation}
\dot\varphi = L_\alpha^{-1} (\dot v_0 + Q\cdot (d\dot S + \dot\xi))
\end{equation}
we have 
\begin{equation}
\abs{\dot\varphi}_{s-\sigma} \leq \frac{C}{\gamma^2\sigma^{2\tau + 2n}}\abs{\delta v}_{g,s+\sigma}.
\end{equation}
As $\delta g = g'\cdot \dot g$, we have the same sort of estimates for the wanted $\delta g$:
\begin{equation}
\abs{\delta g}_{s-\sigma}\leq \frac{1}{\sigma}(\abs{g - \id}_{s+\sigma} + 1)\frac{C}{\gamma^2 \sigma^{2\tau + 2n }}\abs{\delta v}_{g,s+\sigma}.
\end{equation}
\item Again, $\sq{u,\dot g} + \delta u = \dot v - \dot b$ determines $\delta u$ explicitly, and we have
\begin{equation*}
\abs{\delta u}_{s-\sigma}\leq \frac{C}{\gamma^2\sigma^{2\tau + 2n + 1}}\abs{\delta v}_{g, s+\sigma}.
\end{equation*}
Up to defining $\sigma' = \sigma/2$ and $s' = s + \sigma$ we have proved the following lemma for all $s',\sigma'$ such that $s'< s' + \sigma'$.
\end{itemize}
\begin{lemma}
\label{lemma Russmann}
If $(g,u,b)$ are in $\G_{s+\sigma}^{\omega,\sigma^2/2n}\times\Uh_{s+\sigma}(\alpha,-\eta)\times\Lambda(0,b)$ then for every $\delta v$ in $({\Vh\oplus (-\eta r + \eta\R^n)})_{g,s+\sigma}$, there exists a unique triplet $(\delta g,\delta u, \delta \lambda)\in T_g\mathcal{G}_{s}^{\omega}\times\overrightarrow{\Uh_s}(\alpha,-\eta)\times\Lambda(0,b)$ such that 
\begin{equation}
\phi'(g,u,\lambda)\cdot (\delta g,\delta u, \delta \lambda) = \delta v;
\end{equation}
moreover, we have the following estimate
\begin{equation*}
\max\,(\abs{\delta g}_s,\abs{\delta u}_s,\abs{\delta b})\leq \frac{C'}{\sigma^{\tau'}}\abs{\delta v}_{g,s+\sigma},
\end{equation*}
$C'$ being a constant depending on $\abs{g - \id}_{s+\sigma}$ and $\abs{u - (\alpha, -\eta r)}_{s+\sigma}$.
\end{lemma}  
Concerning the bound on $\phi''$, the analogue of lemma \ref{lemma derivata seconda} follows readily.\medskip \\It just remains to apply theorem \ref{teorema inversa}, and complete the proof for the chosen $v$ in $\pa{\Vh\oplus (-\eta r + \eta\R^n)\partial_r}_{s+\sigma}\in \V=\bigcup_{s>0}\V_s$. 
\end{proof}
We conclude the section with a second diagram.\\ 
\hfill{
\begin{equation*}
\xymatrix{\text{Moser:} & \G\times\U(\alpha,A)\times\Lambda(\beta, b + B\cdot r)\ar[r]_-{\simeq\,\text{loc.}} & \V\\
\text{"à la R\"ussmann":} & \G^\omega\times\U^{\Ham}(\alpha,-\eta )\times\Lambda(0,b)\ar[r]_-{\simeq\,\text{loc.}} & \Vh\oplus (-\eta r + \eta\R^n)\base{r})\ar @{^{(}->}[u] }
\end{equation*} 

\section{Extension of Herman's and R\"ussmann's theorems to simple normally hyperbolic tori}
\label{general Russmann}
The peculiarity of the normal forms proved in the previous section, is that the translated (or twisted) $\alpha$-quasi-periodic torus $g(\normal{T}^n_0)$ of the perturbed $v$ keeps its $\eta$-normally attractive dynamics (resp. repulsive, if $\eta<0$), the reason of such a result relying on the Hamiltonian nature of perturbations.

On $\T^n\times\R^m$, let $u\in\U(\alpha,A)$. We will say that $\text{T}^n_0$ is \emph{simple normally hyperbolic} if $A$ has simple, non $0$,  real eigenvalues. \\ Note that the space of matrices $A\in\Mat_m(\R)$ with simple non $0$ real eigenvalues is open in $\Mat_m(\R)$, thus it provides a consistent interesting set of frequencies to work on.\\ We show here that for general perturbations, at the expense of conjugating $v- \lambda$ to a vector field $u$ with different (opportunely chosen) normally hyperbolic dynamics,  we can show that a translated or twisted reducible $\alpha$-quasi-periodic Diophantine torus exists. After all, the classic translated curve theorem of R\"ussmann does not provide any information on the normal dynamics of the curve, but just the tangent one; for this reason it seems significant to us to give the following more general results (Theorems \ref{theorem C} and \ref{theorem D} stated in the introduction).\\
Notations are the same as in section \ref{the normal form of Moser}.

Let $\Delta^s_m(\R)\subset\Mat_m(\R)$ be the space of matrices with simple, non $0$, real eigenvalues and let
\comment{ Let first
\begin{equation*}
\U_s = \bigcup_{A\in\Delta_m(\R)} \U_s(\alpha,A) = \set{u(\teta,r)= (\alpha + O(r), A\cdot r + O(r^2)),\, A\in\Delta_m(\R)}, 
\end{equation*} 
and }
\begin{equation*}
\widehat{\U}_s = \bigcup_{A\in\Delta^s_m(\R)} \U_s(\alpha,A) = \set{u(\teta,r)= (\alpha + O(r), A\cdot r + O(r^2)),\, A\in\Delta^s_m(\R)}.
\end{equation*}

\comment{Eventually, denote with $\Lambda_{n+m^2}$ the set of $\lambda=(\beta, b + B\cdot r)\in\Lambda$ with $B$ having $0$-diagonal elements.
\begin{prop}
For every $u^0\in \U_{s+\sigma}(\alpha,A^0)$ with $\alpha$ Diophantine, there is a germ of $C^{\infty}$-maps
\begin{equation*}
\psi: \V_{s+\sigma}\to \G_{s}\times \U_{s}\times\Lambda_{n+m^2},\quad v\mapsto (g,u,\lambda),
\end{equation*} at $u^0\mapsto (\id, u^0, 0)$, such that $v = \push{g}u + \lambda$.
\label{general}
\end{prop}   
}

\comment{\begin{rmk} Since we assumed that $A$ has simple eigenvalues, the (eventual) translation vector $b$ of $\lambda = (\beta, b + B\cdot r)$, is of dimension $1$.
\end{rmk}}

\begin{thm}[Twisted torus]
For every $u^0\in {\U}_{s+\sigma}(\alpha,A^0)$ with $\alpha$ Diophantine and $A^0\in\Delta^s_m(\R)$, there is a germ of $C^{\infty}$-maps
\begin{equation*}
\psi: \V_{s+\sigma}\to \G_{s}\times \widehat{\U}_{s}\times\Lambda(\beta),\quad v\mapsto (g,u,\beta),
\end{equation*} at $u^0\mapsto (\id, u^0, 0)$, such that $v = \push{g}u + \beta\,\partial_\teta$.
\label{proposition B=0}
\end{thm}

\begin{proof}
Without loss of generality we suppose that $A^0$ is already in its diagonal form of real, simple, non $0$ eigenvalues.\\
We denote $\phi_A$ the normal form operator $\phi$ (recall definition in \eqref{normal operator}), since in Moser's theorem $A$ was fixed while now we want it to vary. Identifying with $\R^m$ the space of diagonal $m$-dimensional matrices, we define the map 
\begin{equation*}
\hat\psi : \R^m\times \V_{s+\sigma} \mapsto \G\times\widehat{\U}_s\times\Lambda,\quad \hat{\psi}_A(v) := \phi'^{-1}_A (v) = (g,u,\lambda), 
\end{equation*}
 locally in the neighborhood of $(A^0,u^0)$, such that $\push{g}u +\lambda = v.$\\ Let us now write $u^0$ as 
\begin{equation*}
u^0 = \pa{\alpha + O(r), (A^0 - A)\cdot r + A\cdot r + O(r^2)}.
\end{equation*}
Since $$ \phi_A \pa{\id, u^0 + (0,(A - A^0)\cdot r), (0, (A^0 - A)\cdot r)} \equiv u^0,$$ locally for all $A$ close to $A^0$ we have $$\hat{\psi}(A, u^0) = (\id, u, B\cdot r),\quad B (A, u^0) = (A^0 - A) = \delta A.$$ In particular $$\frac{\partial B}{\partial A} = - \id, $$ hence $A\mapsto B(A)$ is a local diffeomprhism\footnote{Recall that $B$ is diagonal since eigenvalues are simple.}; thus by the implicit function theorem locally for all $v$ there exists a unique $\bar A$ such that $B (\bar A, v)=0$. It remains to define $\psi(v) = \hat{\psi}(\bar A, v)$.
\end{proof}

\comment{Corollary \ref{proposition B=0} immediatly follows from condition $[A^0,B]=0$, $A^0\in\Delta^s_m(\R)$ and the classic implicit function theorem.}

\begin{rmk} The fact that $A^0$ has real eigenvalues makes the correction $A = A^0 + \delta A$ of $A^0$ (provided by the implicit function theorem) well defined. If we had considered possibly complex eigenvalues, submitted to Diophantine condition \eqref{DC}, the procedure would have been more delicate, using the Whitney dependence of $\phi$ in $A$. In this line of thought see \cite{Fejoz:2004} and the "hypothetical conjugacy" theorem therein.
\end{rmk}

\begin{rmk}  If we let the possiblility of having a $0$ eigenvalue, the torus would be twisted-translated, due to the presence of $b\in\ker A \equiv \R$, providing a generalization in higher dimension, for vector fields, of Herman's translated torus theorem for perturbations of smooth embeddings of $F : \T^n\times [-r_0,r_0]\to \T^n\times \R$ verifying $F(\teta,0) = (\teta + \alpha,0)$, see \cite{Yoccoz:Bourbaki}.
\end{rmk}

On $\T^n\times\R^m$, with $m\geq n$, suppose that vector fields in $\U(\alpha,A)$ have a twist, in the sense that the coefficient $u_1: \T^n\to \Mat_{n\times m}(\R)$ in $$u(\teta,r)= (\alpha + u_1(\teta)\cdot r + O(r^2), A\cdot r + O(r^2))$$ is such that $\int_{\T^n} u_1(\teta)\frac{d\teta}{(2\pi)^n}$ has maximal rank $n$.

\begin{thm}[Translated torus]
\comment{For every $u^0\in \U_{s+\sigma}(\alpha,A^0)$ with $\alpha$ Diophantine and $A^0\in\Delta^s_m(\R)$, there is a germ of $C^{\infty}$-maps
\begin{equation*}
\psi: \V_{s+\sigma}\to \G_{s}\times \widehat{\U}_{s}\times\R^n,\quad v\mapsto (g,u,b),
\end{equation*} at $u^0\mapsto (\id, u^0, 0)$, such that $v = \push{g}u + b\,\partial_r$.}
Let $\alpha$ be Diophantine and let $u^0\in\U_{s+\sigma}(\alpha,A^0)$, with $A^0\in\Delta^s_m(\R)$, have a twist. Every $v$ sufficiently close to $u^0$ possesses a translated simple normally hyperbolic torus on which the dynamics is $\alpha$-quasi-periodic.
\label{translated torus}
\end{thm}   

Like in Theorem \ref{theorem Russmann}, if on the one hand we take advantage of the twist hypothesis in order to avoid the twist-term $\beta$ à la Herman, on the other one the linear term $A\cdot r$ necessairily gives out a constant (translation) term, which one need to keep track of by the introduction of a translation parameter (remember the form of equation \eqref{R_0} or \eqref{Ds}), that adds up to counter terms already needed. And this can be directly seen from the normal form at the first order; the proof is an immediate consequence of the torsion hypothesis on $u^0$.
\begin{proof}
Let $\hat{\varphi}$ be the function defined on $\T^n$ taking values in $\Mat_{n\times m}(\R)$ that solves the (matrix of) homological equation\footnote{Each component reads as an equation of the scalar case in Lemma \ref{normal lemma}} \[L_{\alpha}\hat\varphi(\teta) + \hat\varphi(\teta)\cdot A + u_1(\teta) = \int_{\T^n} u_1(\teta)\frac{d\teta}{(2\pi)^n},\] and let $F: (\teta,r)\mapsto (\teta + \hat{\varphi}(\teta)\cdot r,r)$. The diffeomorphism $F$ restricts to the identity at $\normal{T}^n_0$. At the expense of substituting $u^0$ and $v$ with $\push{F}u^0$ and $\push{F}v$ respectively, we can assume that 
$$u(\teta,r)= (\alpha + u_1\cdot r + O(r^2), A\cdot r + O(r^2)),\quad u_1 = \int_{\T^n} u_1(\teta)\frac{d\teta}{(2\pi)^n}.$$ The germs so obtained are close to one another. \\ Consider now the family of trivial perturbations obtained by translating $u^0$ in actions $u^0_c (\teta, r) = u^0(\teta, c + r),\quad c\in\R^m.$ Taking its Taylor expansion and the approximation obtained by cutting it from terms $O(c)$ that possibly depend on angles, we immediately read its normal form where, in particular, $b^0(c)= A\cdot c + O(c^2)$ and $\beta^0(c) = u_1\cdot c + O(c^2)$. The map $c\mapsto \beta^0(c)$ is indeed a submersion since $\partial_c\beta^0(0) = u_1$ and the twist hypotesis on $u^0.$ The analogous map for $v_c$, being its small $C^1$ perturbation, is thus submersive too. The theorem follows from the same argument of Theorem \ref{proposition B=0} (applied to eliminate $B$), together with keeping track of the supplementary translation term $b(c)$, and the fact that $c\mapsto \beta(c)$ is sumbersive, guaranteeing the existence of $c$ such that $\beta(c)=0.$ Hence for this $c$, $v_c = \push{g}u + b$. When $n=m$ this $c\in\R^n$ is unique. 
\end{proof}
\section{An application to Celestial Mechanics} 
\label{section A KAM result}
\comment{To deduce the existence of invariant tori from the normal forms presented up to now, we must eliminate the obstruction to the dynamical conjugacy represented by the counter-terms. To do so, the idea is to exploit the following reasoning. Notations are the ones given in sections \ref{Spaces of vector fields} and \ref{subsection The infinitesimal equation}.  \\
Suppose that the vector field $v\in\V$ smoothly depends on some external parameter $\Omega\in B^N(0)$ (the unit ball in $\R^N$) and that, $u^0\in\U(\alpha,A)$ being given, $v$ is sufficiently close to it. Suppose also that estimates proven in propositions \ref{proposition infinitesimal} and lemma \ref{lemma derivata seconda} are uniform with respect to $\Omega$ over some closed subset of $\R^n$. The parametrized version of Moser's theorem follows readily. Calling $\phi_\Omega$ and $\psi_{\Omega}$ the corresponding parametrizations of the normal form operators,  let $$\psi_\Omega : v \mapsto (g, u, \lambda) $$ be the triplet given by the theorem; if  $\Omega\mapsto\lambda(\Omega)$ is submersive, there exists $\bar\Omega$ such that $\lambda(\bar\Omega)=0$. In particular, if $N$ equals the dimension of $\Lambda$, this point is locally unique. The corresponding $g$ hence conjugates $v$ and $u$.\medskip \\
The normal form, thus reduces the issue of proving the existence of an invariant torus to the applicability of the standard inverse function theorem in finite dimension.}

The normal forms constructed in section \ref{Hamiltonian-dissipative section} fit well in the dissipative spin-orbit problem. We deduce here the central results of \cite[Theorem $3.1$]{Locatelli-Stefanelli:2012} and \cite[Theorem $1$]{Celletti-Chierchia:2009}, by easy application of the translated torus theorem \ref{theorem Russmann} and the elimination of the translation parameter.

\subsection{Spin-orbit in $n$ d.o.f.}

\subsubsection{Normal form \& elimination of $b$ }
\label{elimination of parameters}
We consider a vector field on $\T^n\times\R^n$ of the form \[\hat{v} = v^{\h}\oplus (-\eta (r -\Omega)\partial_r)\] where $v^{\h}$ is a Hamiltonian vector-field whose Hamiltonian $H$ is close to the Hamiltonian in Kolmogorov normal form with non degenerate quadratic part introduced in section \ref{objects spin}: \[K^0(\teta,r)= \alpha\cdot r + \frac{1}{2}Q(\teta)\cdot r^2 + O(r^3).\] The vector field $\hat{v}$ is hence close to the corresponding unperturbed $\hat{u}:$ \[\hat{u}= u^{K^0}\oplus (-\eta (r - \Omega)\partial_r).\] $\Omega\in\R^n$ is a vector of free parameters representing some "external frequencies" (we will see in the concrete example of the "spin-orbit problem" the physical meaning of $\Omega$). We will note $v$ and $u^0$ the part of $\hat{v}$ and $\hat{u}$ with $\Omega=0$.
\begin{thm}[Dynamical conjugacy]
\label{teorema implicita}
Let $v^{\h}$ be sufficiently close to $u^{K^0}$.  There exists a unique $\Omega\in\R^n$ close to $0$, a unique $u\in\U^{\Ham}(\alpha,-\eta)$ and a unique $g\in\G^{\omega}$ such that $\hat{v} = v + \eta\Omega\partial_r$ (close to $\hat{u} = u^0 + \eta\Omega\partial_r$) is conjugated to $u$ by $g$: $\hat{v} = \push{g}u$.
\end{thm}
\begin{proof}
Let us write the non perturbed $\hat{u}:$ 
\begin{equation}
\hat{u} = \eqsys{\dot\teta = \alpha + O(r)\\ \dot r = -\eta r + \eta\Omega + O(r^2).}
\label{sistema Locatelli Stefanelli}
\end{equation}
We remark that $\eta\Omega$ is the first term in the Taylor expansion of the counter term $b$ appearing in the normal form of theorem \ref{theorem Russmann}, applied to $\hat{v}$ close to $\hat{u}$. In particular $\hat{u}= \push{\id} u^0 + \eta\Omega\partial_r$ by uniqueness of the normal form and, if $\Omega = 0$, $\normal{T}^n_0$ is invariant for \eqref{sistema Locatelli Stefanelli}. \medskip \\Hence consider the family of maps \[\amat{llcl}{\psi\,: &(\Vh\oplus\pa{-\eta(r - \Omega)\partial_{r},\hat{u}}&\to  &\pa{\G^{\omega}\times\U(\alpha,-\eta)\times\Lambda(0,b), (\id,u^0,\eta\Omega)} \\ & \hat{v} & \mapsto & \psi(\hat{v}):=\phi^{-1}(\hat{v})=(g,u, b)}\] associating to $\hat{v}$ the unique triplet provided by the translated torus theorem \ref{theorem Russmann}.  \\ In order to prove that the equation $b=0$ implicitly defines $\Omega$, it suffices to show that $\Omega\mapsto b(\Omega)$ is a local diffeomorphism; since this is an open property with respect to the $C^1$-topology, and $\hat{v}$ is close to $\hat{u}$, it suffices to show it for $\hat{u}$, which is immediate. \\
Note in particular that $b = \sum_k \delta b_k$ where $\delta b_k$, uniquely determined at each step of the Newton scheme, is of the form $\delta b_k = \eta(\widehat{\delta\Omega_k} - \dot{\xi}_k)$.\\ Hence $b = \eta\Omega + (\text{perturbations}<< \eta\Omega)$. So there exists a unique value of $\Omega$, close to $0$, such that $b(\Omega)=0$.
\end{proof}
Note that the size of perturbation $\abs{\hat{v} - \hat{u}}_{s+\sigma}$, is independent of $\Omega$ and that constants $C'$ and $C''$ (appearing in \eqref{bound phi'^-1} and \eqref{bound phi''} in the proof of theorem \ref{teorema inversa}) are eventually uniform with respect to $\Omega$ over some closed subset of $\R^n$.
\begin{rmk}
\label{remark nu}
$\Omega$ is the value that compensates the "total translation" of the torus, given by the successive translations provided by the $\xi'$s at each step of the Newton algorithm; this can be directly seen by looking at the iterates of the Newton operator of theorem \ref{teorema inversa} applied to this problem. Using the same notations, we have $x_0 = (\id, u^0, \eta\Omega)$, $\phi(x_0)= u^0 + \eta\Omega\partial_r$ hence \[x_1 = x_0 + \phi'^{-1}(x_0)\cdot(v - \phi(x_0)),\] where $(v-\phi(x_0))$ has no more $\eta\Omega\partial_r$. Thus the term $\delta b_1$ determined by $\phi'^{-1}(x_0)\cdot(v-\phi(x_0))$ results in $\delta b_1 = - \eta\delta\xi_1$ (remember system \eqref{varphi}-\eqref{Ds}-\eqref{terza}). At the second iterate, $\delta b_2 = - \eta\delta\xi_2$,  since the term we called $\eta\widehat{\delta\Omega}$ (given by the pull-back of $\delta v_2$ by $g_1$ determined at the previous step) is $\eta\widehat{\delta\Omega}= \eta(\delta\xi_1 - \delta\xi_1)=0$\footnote{Because of the form of $g$ and the fact that $\xi\in\R^n$, the terms $\delta\xi$ and $\dot\xi$ appearing in $\delta g$ and $\dot g = g'^{-1}\cdot\delta g$ are the same.}. And so on.
\end{rmk}

\subsection{Spin-Orbit problem of Celestial Mechanics}
\label{The spin-orbit problem}
Applying theorems \ref{theorem Russmann} and \ref{teorema implicita}, the elimination of the obstructing translation parameter $b$ provides here a picture of the space of parameters proper to this physical system (see theorem \ref{teorema superfici}).\\
A satellite (or a planet) is said to be in $n:k$ spin-orbit resonance when it accomplishes $n$ complete rotations about its spin axis, while revolving exactly $k$ times around its planet (or star). There are various examples of such a motion in Astronomy, among which the Moon $(1:1)$ or Mercury $(3:2)$.

The "dissipative spin-orbit problem" of Celestial Mechanics can be modeled by the following equation of motion in $\R$:
\begin{equation}
\ddot\teta + \eta(\dot\teta - \nu) + \eps \partial_\teta f(\teta,t) = 0,
\label{spin-orbit equation}
\end{equation}

where $(\teta,t)\in\T^2$, the angular variable $\teta$ determines the position of an oblate satellite (modeled as an ellipsoid) whose center of mass revolves on a given elliptic Keplerian orbit around a fixed massive major body, $\eta>0$ is a dissipation constant depending on the internal non rigid structure of the body that responds in a non-elastic way to the gravitational forces, $\eps>0$ measures the oblateness of the satellite while $\nu\in\R$ an external free parameter proper to the physical problem. We suppose that the potential function $f$ is real analytic in all its variables.\\
See \cite{Laskar-Correia:2010} and references therein for a complete physical discussion of the model and deduction of the equation.
\comment{
We want to study the dynamics of the rotation about its spin axis (represented by the angular variable $\teta$) of a triaxial body whose center of mass revolves along a given elliptic Keplerian orbit around a fixed massive point. The rotation axis is supposed to be perpendicular to the orbital plane. The internal structure of the body responds in a non-elastic way to gravitational forces. In the case of an ellipsoid with different equatorial axis, the calculation of the potential gives out a supplementary term proportional to the difference of the two smallest axes of inertia which perturbs the rotation. Under these hypothesis, the dynamics is described by the following equation of motion in $\R$:

\begin{equation}
\ddot\teta + \eta(\dot\teta - \nu) + \eps \partial_\teta f(\teta,t) = 0.
\label{spin-orbit equation}
\end{equation}
}

\comment{ 
 In this model $(\teta,t)\in\T^2$, $\eta\in\R^{+}$ is the constant of dissipation, $\nu\in\R$ a free-parameter proper to the physical problem. We suppose that the potential function $f$ is real analytic in all its variables.\\ 
 
 }
 
Let now $\alpha$ be a fixed Diophantine frequency. In the coordinates $(\teta, r = \dot\teta - \alpha)$ the system associated to \eqref{spin-orbit equation} is
\begin{equation}
\eqsys{\dot\teta = \alpha + r\\
\dot r= -\eta r + \eta(\nu - \alpha) - \eps\partial_\teta f(\teta,t).}
\label{Celletti-Chierchia}
\end{equation}
We immediately see that when $\eps = 0$ and $\eta\neq 0$, $r = 0$ is an invariant torus provided that $\nu=\alpha$. Furthermore, the general solution of $\ddot\teta + \eta(\dot\teta - \nu)=0$ is given by \[\teta(t)=\nu t + \teta^0 + \frac{r_0 - (\nu - \alpha)}{\eta}(1-e^{-\eta t}),\] showing that the rotation tends asymptotically to a $\nu$-quasi-periodic behavior. Here the meaning of $\nu$ is revealed: $\nu$ is the frequency of rotation to which the satellite tends because of the dissipation, if no "oblate-shape effects" are present.\\ On the other hand when $\eps\neq 0$ and $\eta = 0$ we are in the conservative regime, and the classical KAM theory applies.\\
The main question then is: fixing $\alpha$ Diophantine does there exist a value of the proper rotation frequency $\nu$ such that the perturbed system possesses an $\alpha$-quasi-periodic invariant $\eta$-attractive torus?\footnote{In \cite{Celletti-Chierchia:2009} they look for a function $u:\T^2\to\R$ of the form $\teta(t)= \alpha t + u(\alpha t, t)$ satisfying \eqref{spin-orbit equation} for a particular value $\nu$. This function is found as the solution of an opportune PDE.}

\subsubsection{Extending the phase space}
In order to apply our general scheme to the non autonomous system \eqref{Celletti-Chierchia}, as usual we extend the phase space by introducing the time (or its translates) as a variable. The phase space becomes $\T^2\times\R^2$ with variable $\teta_2$ corresponding to time and $r_2$ its conjugated.\\
Hence consider the family of vector fields (parametrized by $\Omega\in\R$) $$v=v^{\h} \oplus (-\eta r + \eta\Omega)\partial_{r},$$ where $\Omega = (\nu-\alpha,0)$ and $v^{\h}$ corresponds to $$H(\teta,r)= \alpha\cdot r_1 + r_2 + \frac{1}{2}r_1^2 + \eps f(\teta_1,\teta_2).$$
\comment{\begin{equation}
v = \eqsys{\dot{\teta}_1 = \alpha + r_1\\
\dot{\teta}_2 = 1\\
\dot{r}_1 = -\eta (r_1 - (\nu - \alpha)) - \eps\frac{\partial f}{\partial\teta_1}\\
\dot{r}_2 = -\eta r_2 - \eps\frac{\partial f}{\partial\teta_2}. }
\end{equation}}
The following objects are essentially the ones introduced in section \ref{Russmann section vector fields}, taking into account the introduction of the time-variable $\teta_2 = t$ and its conjugated $r_2.$ 
Let $\bar\alpha = (\alpha,1)$ satisfy 
\begin{equation}
\abs{k_1 \alpha + k_2}\geq \frac{\gamma}{\abs{k}^\tau},\quad \forall k\in\Z^2\setminus\set{0}.
\label{diophantine time}
\end{equation}
Let $\bar{\mathcal{H}}$ be space of real analytic Hamiltonians defined in a neighborhood of $\normal{T}_0 =\T^2\times\set{0}$ such that for $H\in\bar{\mathcal{H}}$,  $\partial_{r_2} H \equiv 1$. For these Hamiltonians the frequency $\dot\teta_2 = 1$ (corresponding to time) is fixed. Let $\bar\alpha = (\alpha,1)$ and $\bar{\mathcal{K}} = \bar{\mathcal{H}}\cap \mathcal{K}^{\bar\alpha}$. Let also $\bar{\G}^{\omega}$ be the subset of $\G^{\omega}$ such that $\bar\xi=(\xi,0), \varphi(\teta)=(\varphi_1(\teta),\teta_2).$ The corresponding $\dot g\in T_{\id} \bar\G$ are $\dot g = (\dot\varphi, \,^t\dot\varphi'\cdot r + d\dot S + \dot\xi)$ with $\dot\varphi = (\dot\varphi_1,0)$ and $\dot\xi = (\dot\xi_1,0)$. Eventually $\bar\Lambda=\set{\lambda:\, \lambda(\teta,r)= b\base{r_1}} \equiv \R$
\\

By restriction, the normal form operator $$\bar\phi:\bar\G^{\omega,\sigma^2/2n}_{s+\sigma}\times\Uh_{s+\sigma}(\bar\alpha,-\eta)\times\bar\Lambda\to\pa{\Vh\oplus (-\eta r + \eta\R )\partial_r}_s,$$ $$(g,u,\lambda)\mapsto \push{g}u + b\partial_{r_1},$$ and the corresponding \[\bar\phi'(g,u,\lambda): T_g\mathcal{G}^{\omega,\sigma^2/2n}_{s+\sigma}\times\overrightarrow{\Uh_{s+\sigma}}(\bar\alpha,-\eta)\times\bar\Lambda\to\pa{\Vh\oplus(-\eta r + \eta\R)\partial_r}_{g,s},\] are now defined. 

\begin{cor}[Normal form for time-dependent perturbations] 
\label{Russmann spin-orbit}
The operator $$\bar\phi : \bar\G^{\omega,\sigma^2/2n}_{s+\sigma}\times\Uh_{s+\sigma}(\bar\alpha,-\eta)\times\bar\Lambda\to\pa{\Vh\oplus (-\eta r + \eta\R\partial_{r_1})}_s $$ is a local diffeomorphism.
\end{cor}

The proof is recovered from the one of theorem \ref{theorem Russmann}, taking into account that the perturbation belongs to the particular class $\bar{\mathcal{H}}$.

\begin{lemma}[Inversion of $\bar\phi'$]
If $(g,u,\lambda)\in \bar\G^{\omega,\sigma^2/2n}_{s+\sigma}\times\Uh_{s+\sigma}(\bar\alpha,-\eta)\times\bar\Lambda $, for every $\delta v\in \pa{\Vh\oplus(-\eta r + \eta\R\partial_{r_1})}_{g,s+\sigma}$ there exists a unique triplet $(\delta g, \delta u, \delta \lambda)\in T_g\mathcal{G}^{\omega}_{s}\times\overrightarrow{\U}^{\Ham}_{s}(\bar\alpha,-\eta)\times\Lambda$ such that \[\bar\phi'(g,u,\lambda)(\delta g,\delta u,\delta\lambda)=\delta v;\] moreover \[\max\set{\abs{\delta g}_s,\abs{\delta u}_s,\abs{b}}\leq \frac{C'}{\sigma^{\tau'}}\abs{\delta v}_{g,s+\sigma},\] the constant $C'$ depending only on $\abs{g - \id}_{s+\sigma}$ and $\abs{u - (\alpha,-\eta r)}_{s+\sigma}$.
\end{lemma}

\begin{proof}
Following the calculations made in the proof of lemma \ref{lemma Russmann} we need to solve the following homological equations: 
 \begin{align*}
 %\label{varphi spin}
 \dot{\varphi}_1'\cdot\bar\alpha - Q_{11}(\teta)(d\dot S_1 + \dot{\xi}_1) &= \dot{v}^{H}_{1,0}\\
 %\label{dS1 spin}
d\dot S_1' \cdot \bar\alpha + \eta (d\dot S_1 + \dot{\xi}_1) &= \dot{V}^{H}_{1,0} + \eta\widehat{\delta\Omega} - (\delta b + \partial_{\teta_1} v^1\delta b)\\
%\label{dS2 spin}
d\dot{S}_2'\cdot\bar\alpha + \eta d\dot{S}_2 &= \dot{V}_{2,0}^{H} -\partial_{\teta_2}v^1\delta b,
\end{align*}

 The lower indices indicate the component and the order of the corresponding term in $r$ whose they are the coefficient.\footnote{We noted with $v^1 = \varphi_1 - \id$, coming from the first component of $\varphi = (\varphi_1,\id)$.} Hence, the first one corresponds to the direction of $\teta$ and the second twos to the zero order term in $r$ in the normal direction.\\ The tangential equation relative to the time component (that we omitted above) is easily determined: computation gives $\dot v_{2,0} = 0$,  because of $\delta v\partial_{\teta_2} = \delta 1 = 0$ and the form of $g'^{-1}$, and $\dot\varphi_2 = 0$, as well as $Q(\teta)\cdot d\dot S\,\partial_{\teta_2}=0$.\\ Equations relative to the linear term, follow from the Hamiltonian character.
Solutions follows from lemmata \ref{tangent lemma} and \ref{normal lemma}, the same kind of estimates as in lemma \ref{lemma Russmann} hold, hence the required bound. 
\end{proof}
\begin{lemma} There exists a constant $C''$, depending on $\abs{x}_{s+\sigma}$ such that in a neighborhood of $(\id,u^0,0) \in \bar\G^{\omega}_{s+\sigma}\times\Uh_{s+\sigma}(\bar\alpha,-\eta)\times\bar\Lambda $ the bilinear map $\phi''(x)$ satisfies the bound
\begin{equation*}
\abs{\phi''(x)\cdot\delta x^{\tensor 2}}_{g,s}\leq \frac{C''}{\sigma^{\tau''}}\abs{\delta x}^2_{s+\sigma}.
\end{equation*}
\end{lemma}
Proof of corollary \ref{Russmann spin-orbit} follows.

\subsection{Surfaces of invariant tori}
\label{A curve of invariant tori}
The results below will follow from corollary \ref{Russmann spin-orbit} and theorem \ref{teorema implicita}.
 
\begin{thm}[Cantor set of surfaces] Let $\eps_0$ be the maximal value that the perturbation can attain. In the space $(\eps,\eta,\nu)$, to every $\alpha$ Diophantine corresponds a surface $\nu=\nu(\eta,\eps)$ ($\eps\in [0,\eps_0]$) analytic in $\eps$, on which the counter term $b$ vanishes, guaranteeing the existence of invariant attractive (resp. repulsive) tori carrying an $\alpha$-quasi-periodic dynamics. 
\label{teorema superfici}
\end{thm}

\begin{cor}[A curve of normally hyperbolic tori]
Fixing $\alpha$ Diophantine and $\eps$ sufficiently small, there exists a unique analytic curve $C_\alpha$, in the plane $(\eta,\nu)$ of the form $\nu = \alpha + O(\eps^2)$, along which the counter term $b (\nu,\alpha,\eta,\eps) $ "à la R\"ussmann" vanishes, so that the perturbed system possesses an invariant torus carrying quasi-periodic motion of frequency $\alpha$. This torus is attractive (resp. repulsive) if $\eta>0$ (resp. $\eta<0$).
\label{corollario curve}
\end{cor}

\begin{figure}[h!]
\begin{picture}(0,0)%
\includegraphics{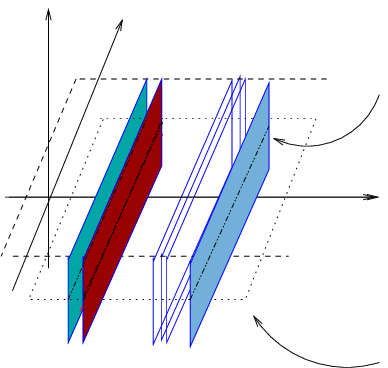}%
\end{picture}%
\setlength{\unitlength}{1657sp}%
\begingroup\makeatletter\ifx\SetFigFont\undefined%
\gdef\SetFigFont#1#2#3#4#5{%
  \reset@font\fontsize{#1}{#2pt}%
  \fontfamily{#3}\fontseries{#4}\fontshape{#5}%
  \selectfont}%
\fi\endgroup%
\begin{picture}(4347,4486)(5344,-7190)
\put(5761,-2851){\makebox(0,0)[lb]{\smash{{\SetFigFont{5}{6.0}{\familydefault}{\mddefault}{\updefault}{\color[rgb]{0,0,0}$\varepsilon$}%
}}}}
\put(9631,-5146){\makebox(0,0)[lb]{\smash{{\SetFigFont{5}{6.0}{\familydefault}{\mddefault}{\updefault}{\color[rgb]{0,0,0}$\nu$}%
}}}}
\put(6796,-3211){\makebox(0,0)[lb]{\smash{{\SetFigFont{5}{6.0}{\familydefault}{\mddefault}{\updefault}{\color[rgb]{0,0,0}$\eta$}%
}}}}
\put(5401,-4336){\makebox(0,0)[lb]{\smash{{\SetFigFont{5}{6.0}{\familydefault}{\mddefault}{\updefault}{\color[rgb]{0,0,0}$\varepsilon_0$}%
}}}}
\put(9676,-6811){\makebox(0,0)[lb]{\smash{{\SetFigFont{5}{6.0}{\familydefault}{\mddefault}{\updefault}{\color[rgb]{0,0,0}plan containing the $C_\alpha$'s}%
}}}}
\put(8236,-5461){\makebox(0,0)[lb]{\smash{{\SetFigFont{5}{6.0}{\familydefault}{\mddefault}{\updefault}{\color[rgb]{0,0,0}$\alpha$}%
}}}}
\put(9631,-3841){\makebox(0,0)[lb]{\smash{{\SetFigFont{5}{6.0}{\familydefault}{\mddefault}{\updefault}{\color[rgb]{0,0,0}$C_\alpha$}%
}}}}
\end{picture}%

\caption{The Cantor set of surfaces: transversely cutting with a plane $\eps=const$ we obtain a Cantor set of curves like the one described in corollary \ref{corollario curve}}
\end{figure}

\begin{figure}[h!]
\begin{picture}(0,0)%
\includegraphics{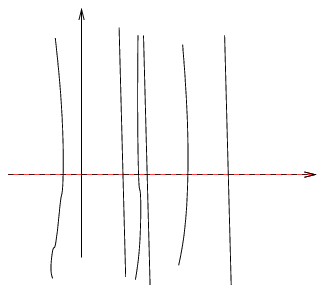}%
\end{picture}%
\setlength{\unitlength}{1657sp}%
\begingroup\makeatletter\ifx\SetFigFont\undefined%
\gdef\SetFigFont#1#2#3#4#5{%
  \reset@font\fontsize{#1}{#2pt}%
  \fontfamily{#3}\fontseries{#4}\fontshape{#5}%
  \selectfont}%
\fi\endgroup%
\begin{picture}(4680,3541)(2146,-5399)
\put(4801,-4321){\makebox(0,0)[lb]{\smash{{\SetFigFont{5}{6.0}{\familydefault}{\mddefault}{\updefault}{\color[rgb]{0,0,0}$\alpha$}%
}}}}
\put(6811,-4351){\makebox(0,0)[lb]{\smash{{\SetFigFont{5}{6.0}{\familydefault}{\mddefault}{\updefault}{\color[rgb]{0,0,0}$\nu$}%
}}}}
\put(3976,-1981){\makebox(0,0)[lb]{\smash{{\SetFigFont{5}{6.0}{\familydefault}{\mddefault}{\updefault}{\color[rgb]{0,0,0}$\eta$}%
}}}}
\put(2161,-4036){\makebox(0,0)[lb]{\smash{{\SetFigFont{5}{6.0}{\familydefault}{\mddefault}{\updefault}{\color[rgb]{0,0,0}Hamiltonian}%
}}}}
\put(2161,-4261){\makebox(0,0)[lb]{\smash{{\SetFigFont{5}{6.0}{\familydefault}{\mddefault}{\updefault}{\color[rgb]{0,0,0}axes $\eta=0$}%
}}}}
\put(4741,-2296){\makebox(0,0)[lb]{\smash{{\SetFigFont{5}{6.0}{\familydefault}{\mddefault}{\updefault}{\color[rgb]{0,0,0}$C_\alpha:\,b(\nu,\eta,\alpha,\eps)=0$}%
}}}}
\end{picture}%

\caption{The corresponding Cantor set of curves on the plane $\eps = const$, whose points correspond to an attractive/repulsive invariant torus}
\end{figure}

 \begin{proof} [Proof of theorem \ref{corollario curve}] We just need to observe the following facts.\\
The existence of the unique local inverse for $\bar\phi'$ and the bound on it and $\bar\phi''$  allow to apply theorem \ref{theorem Russmann} and prove the result once we guarantee that $$\abs{v-u^0}_{s+\sigma}=\max\pa{\eps\abs{\frac{\partial f}{\partial\teta_1}}_{s+\sigma},\eps\abs{\frac{\partial f}{\partial\teta_2}}_{s+\sigma}}\leq \delta\frac{\sigma^{2\tau}}{2^{8\tau}C^2},$$(here we have replaced the constant $\eta$ appearing in the abstract function theorem with $\delta$, in order not to generate confusion with the dissipation term). This ensures that the inverse mapping theorem can be applied, as well as the regularity propositions (\ref{lipschitz} and \ref{smoothness}).  Note that the constant $C$ appearing in the bound contains a factor $1/\gamma^2$ coming from the diophantine condition \eqref{diophantine time}, independent of $\eta$, since the remark \ref{remark uniform} still holds here.

For every $\eta\in [-\eta_0,\eta_0]$, apply theorem \ref{teorema implicita} and find the unique $\nu$, such that $$b(\nu,\eta,\alpha,\eps)=0,$$ (as in the previous case $b$ is of the form $b = \nu - \alpha + \sum_k \delta\xi_k$, smooth with respect to $\nu$ and $\eta$ and analytic in $\eps$). \\In particular the value of $\nu$ that satisfies the equation is of the form $$\nu(\eps,\eta) = \alpha + O(\eps^2).$$ This follows directly from the very first step of Newton' scheme \[x_1 = x_0 + \phi'^{-1}(x_0)\cdot (v - \phi(x_0)),\] where $x_0 = (\id, u^0,\eta (\nu - \alpha))$. Developing the expression one sees that $\delta\xi_1$ (the term of order $\eps$) is necessarily $0$, due to the particular perturbation and the constant torsion.
\end{proof}

\subsection{An important dichotomy}
\label{dichotomy}
The results obtained for the spin-orbit problem, theorem \ref{teorema implicita}, theorem \ref{teorema superfici} and corollary \ref{corollario curve}, are intimately related to the very particular nature of the equations of motions and point out an existing dichotomy between generic dissipative vector fields and the Hamiltonian-dissipative to which the spin-orbit system belongs. For a general perturbation, even if the system satisfied some torsion property, one cannot avoid both the counter term $b$ and $B$ ($B\neq 0$ a priori, since we want to keep the $\eta$-normal coefficient still). Disposing of just $n$ free parameters $\Omega_1,\ldots,\Omega_n$, the best possible result is to eliminate $b$, but it is hopeless to get rid of the obstruction represented by $B$ and have a complete control on the normal dynamics of the invariant torus.\\ In particular, for the spin-orbit problem in one and a half degree of freedom, using transformations as $g(\teta,r)= (\varphi(\teta), \teta_2, R_0(\teta) + R_1(\teta)\cdot r)$ in $\T^2\times \R$, the cohomological equations will read 
\begin{align*}
\dot\varphi' \cdot \bar\alpha - Q\cdot\dot{R_0} &= \dot {v}_0,\\
\dot{R}_0'\cdot \bar\alpha + \eta\dot{R}_0 &= \dot{V}_0 + \eta\dot{\delta\Omega}  - \dot{b},\\
\dot{R}_1'\cdot\bar\alpha + (Q\cdot\dot R_0)' &= \dot{V}_1 - \dot{B},\qquad \bar\alpha=(\alpha,1),\, \delta b, \delta B\in \R
\end{align*} 
 and, disposing of $\nu\in\R$ only, we could try at best to solve $b=0$.
 A \gr{worst situation} could even pop out: if no torsion property is satisfied, we would still have two counter-terms ($\beta$ to solve the equation tangentially and $B$ to straight the linear dynamics) but the second equation would carry a small divisor $\eta$ (divisor of the constant term in the Fourier series) which we cannot allow to get arbitrarily small. A Diophantine condition like $\abs{i\,k\cdot\alpha + \eta}\geq \gamma/(1 + \abs{k})^{-\tau}$ , for some fixed $\gamma,\tau>0$, would imply that the bound on $\eps$ of theorem \ref{teorema inversa} depends on $\eta$ through $\gamma$:  $$\eps < \gamma^4 C' \leq \eta^4 C', $$ meaning that, once $\eps$ is fixed, the curves $C_\alpha$ (obtained by eliminating $\beta$ for example) do not reach the axis $\eta=0$ in the plane $\eps=\text{ const.}$ (we noted $C'$ all the other terms appearing in the bound).
 \begin{figure}[h!]
\begin{picture}(0,0)%
\includegraphics{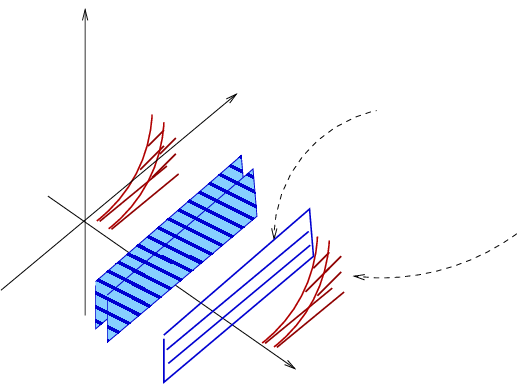}%
\end{picture}%
\setlength{\unitlength}{1657sp}%
\begingroup\makeatletter\ifx\SetFigFont\undefined%
\gdef\SetFigFont#1#2#3#4#5{%
  \reset@font\fontsize{#1}{#2pt}%
  \fontfamily{#3}\fontseries{#4}\fontshape{#5}%
  \selectfont}%
\fi\endgroup%
\begin{picture}(5930,4849)(6589,-6338)
\put(7561,-1636){\makebox(0,0)[lb]{\smash{{\SetFigFont{5}{6.0}{\familydefault}{\mddefault}{\updefault}{\color[rgb]{0,0,0}$\varepsilon$}%
}}}}
\put(10126,-6271){\makebox(0,0)[lb]{\smash{{\SetFigFont{5}{6.0}{\familydefault}{\mddefault}{\updefault}{\color[rgb]{0,0,0}$\nu$}%
}}}}
\put(10801,-4201){\makebox(0,0)[lb]{\smash{{\SetFigFont{5}{6.0}{\familydefault}{\mddefault}{\updefault}{\color[rgb]{.56,0,0}$C_\alpha$ for non Hamiltonian perturbation}%
}}}}
\put(10261,-2986){\makebox(0,0)[lb]{\smash{{\SetFigFont{5}{6.0}{\familydefault}{\mddefault}{\updefault}{\color[rgb]{0,0,.82}$C_\alpha$ for Hamiltonian perturbation}%
}}}}
\put(9136,-2941){\makebox(0,0)[lb]{\smash{{\SetFigFont{5}{6.0}{\familydefault}{\mddefault}{\updefault}{\color[rgb]{0,0,0}$\eta$}%
}}}}
\end{picture}%

\caption{The two situations: 1)blue surfaces $\nu=\nu(\eta,\eps)$ corresponding to the case "Hamiltonian + dissipation" of theorem \ref{corollario curve} 2) Red surfaces corresponding to the more generic case (no torsion and no Hamiltonian structure): they corresponds to invariant tori of co-dimension $1$ ($B\neq 0$).}
\end{figure}

\appendix
\section{Inverse function theorem \& regularity of $\phi$ }
\label{section teorema inversa}
We present here the inverse function theorem we use to prove theorem \ref{abstract Moser}. This results follow Féjoz \cite{Fejoz:2010, Fejoz:2015}. Remark that we endowed functional spaces with weighted norms and bounds appearing in Proposition \ref{proposition infinitesimal} and Lemma \ref{lemma derivata seconda} may depend on $\abs{x}_s$ (as opposed to statements given in \cite{Fejoz:2010, Fejoz:2015});  however we take here account of these (slight) differences.

Let $E=(E_s)_{0<s<1}$ and $F=(F_s)_{0<s<1}$ be two decreasing families of Banach spaces with increasing norms $\abs{\cdot}_s$ and let $B^E_s(\sigma)=\set{x\in E : \abs{x}_s<\sigma}$ be the ball of radius $\sigma$ centered at $0$ in $E_s$. \\ On account of composition operators, we additionally endow $F$ with some deformed norms which depend on $x\in B^E_s(s)$ such that \[ \abs{y}_{0,s} = \abs{y}_s \qquad \text{and}\qquad \abs{y}_{\hat{x},s}\leq \abs{y}_{x, s + \abs{x - \hat{x}}_s}.\]
Consider then operators commuting with inclusions $\phi: B^E_{s+\sigma}(\sigma) \to F_s$, with $0<s<s+\sigma<1$, such that $\phi(0) = 0$. \\We then suppose that if $x\in B^E_{s+\sigma}(\sigma)$ then $\phi'(x):E_{s+\sigma}\to F_s$ has a right inverse $\phi'^{-1}(x): F_{s+\sigma}\to E_s$ (for the particular operators $\phi$ of this work, $\phi'$ is both left and right invertible).\\ $\phi$ is supposed to be at least twice differentiable.\\
Let $\tau:=\tau'+\tau''$ and $C:=C'C''$.
\begin{thm}
Under the previous assumptions, assume
\begin{align}
\label{bound phi'^-1}
\abs{\inderiv{\phi}(x)\cdot\delta y}_s\, &\leq \frac{C'}{\sigma^{\tau'}}\abs{\delta y}_{x,s+\sigma}\\
\label{bound phi''}
\abs{\phi''(x)\cdot \delta x^{\tensor 2}}_{x,s}\, &\leq \frac{C''}{\sigma^{\tau''}}\abs{\delta x}^2_{s+\sigma},\quad \forall s,\sigma: 0<s<s+\sigma<1
\end{align}
$C'$ and $C''$ depending on $\abs{x}_{s+\sigma}$, $\tau',\tau''\geq 1$. \\
For any $s, \sigma, \eta$ with $\eta<s$ and $\eps\leq \eta \, \frac{\sigma^{2\tau}}{2^{8\tau}C^2}$ ($ C\geq 1,\sigma< 3 C$), $\phi$ has a right inverse $\psi: B^F_{s+\sigma}(\eps)\to B^E_{s}(\eta)$. In other words, $\phi$ is locally surjective: $$ B^{F}_{s+\sigma}(\eps)\subset \phi(B^{E}_{s}(\eta)).$$ 
\label{teorema inversa}
\end{thm}
Define
\begin{equation}
Q : B^E_{s+2\sigma}(\sigma) \times B^E_{s+2\sigma}\to F_s,\quad (x,\hat{x})\mapsto \phi(\hat{x}) - \phi(x) - \phi'(x)(\hat{x} - x),
\label{resto Taylor}
\end{equation}  
 the reminder of the Taylor formula.\\

\begin{lemma} 
\label{lemma Q}
For every $x,\hat{x}$ such that $\abs{x-\hat{x}}_s < \sigma$, 
\begin{equation}
\abs{Q(x,\hat{x})}_{x,s}\leq \frac{C''}{2\sigma^2}\,\abs{\hat{x} - x}^2_{{s+\sigma+\abs{\hat{x} - x}}_s}.
\label{stima Q}
\end{equation}
\end{lemma}

\begin{proof}
Let $x_t = (1-t)x + t\hat{x}$, $0\leq t\leq 1$, be the segment joining $x$ to $\hat{x}$. Using Taylor's formula,  \[Q(x,\hat{x})= \int_0^1 (1-t)\phi''(x_t)(\hat{x}-x)^2\, dt,\] hence
\begin{align*}
\abs{Q(x,\hat{x})}_{x,s}&\leq \int_0^1 (1-t)\abs{\phi''(x_t)(\hat{x}-x)^2}_{x,s}\, dt \\
&\leq \int_0^1 (1-t)\abs{\phi''(x_t)(\hat{x}-x)^2}_{x_t,s + \abs{x_t - x}_{s}}\,dt\\
&\leq \int_0^1 (1-t)\,\frac{C''}{\sigma^2} \abs{(\hat{x}-x)}^2_{s +\sigma + \abs{x_t - x}_{s}}\,dt\\
&\leq \frac{C''}{2\sigma^2}\,\abs{\hat{x} - x}^2_{{s+\sigma+\abs{\hat{x} - x}}_s}.
\end{align*}

\end{proof}

\begin{proof}[Proof of theorem \ref{teorema inversa}]
Let $s,\sigma,\eta$, with $\eta<s<1$ be fixed positive real numbers. Let also $y\in B^F_{s+\sigma}(\eps)$, for some $\eps>0$. We define the following map: 
\begin{equation*}
f: B^E_{s+\sigma}(\sigma)\to E_s,\quad x\mapsto x + \inderiv{\phi}(x)(y - \phi(x)).
\end{equation*}
We want to prove that, if $\eps$ is sufficiently small, there exists a sequence defined by induction by
\begin{equation*}
\eqsys{x_0 = 0\\
x_{n+1} = f(x_n) ,}
\end{equation*}
converging towards some point $x\in B^E_s(\eta)$, a preimage of $y$ by $\phi$.\\ Let us introduce two sequences
\begin{itemize}[leftmargin=*]
\item a sequence of positive real numbers $(\sigma_n)_{n\geq 0}$ such that $3\sum_n \sigma_n = \sigma$ be the total width of analyticity we will have lost at the end of the algorithm,\\
\item the decreasing sequence $(s_n)_{n\geq 0}$ defined inductively by $s_0 = s + \sigma$ (the starting width of analyticity), $s_{n+1}= s_n - 3\sigma_n$. Of course, $s_n\to s$ when $n\to + \infty$.
\end{itemize}  
Suppose now the existence of $x_0,...,x_{n+1}$.\\From $x_{k} - x_{k-1} =  \inderiv{\phi}(x_{k-1})(y - \phi(x_{k-1}))$ we see that $y - \phi(x_k) = - Q(x_{k-1},x_k)$, which permits to write $x_{k+1} - x_k = - \inderiv{\phi}(x_k)Q(x_{k-1},x_k)$, for $k= 1,...,n$.\\ Assuming that $\abs{x_{k} - x_{k-1}}_{s_k}\leq \sigma_k$, for $k=1,...n$, from the estimate of the right inverse and the previous lemma we get 

\begin{equation*}
\abs{x_{n+1} - x_n}_{s_{n+1}}\leq \frac{C}{2\sigma_n^{\tau}}\abs{x_n - x_{n-1}}^2_{s_{n}} \leq \ldots \leq C_n C_{n-1}^2\ldots C_1^{2^{n-1}}\abs{x_1 - x_0}^{2^n}_{s_1},
\end{equation*}
with $C_n=\frac{C}{2\sigma_n^{\tau}}$.\\ First, remark that $$\abs{x_1 - x_0}_{s_1}\leq \frac{C'}{(3\sigma_0)^{\tau'}}\,\abs{y-\phi(x_0)}_{s_0}\leq \frac{C}{2\sigma^{\tau}_0}\,\abs{y}_{s+\sigma}\leq \frac{C}{2\sigma^{\tau}_0}\,\eps.$$
Second, observe that if $C_k\geq 1$ (see remark below), $$\abs{x_{n+1}-x_n}_{s_{n+1}}\leq \pa{\eps\prod_{k\geq 0}C_k^{2^{-k}}}^{2^n}.$$
Third, note that $$ \sum_{n\geq 0} z^{2^n} = z + z^2 + z^4 + \ldots \leq z\sum_{n\geq 0} z^n \leq 2z, $$ if $z\leq\frac{1}{2}$.

The key point is to choose $\eps$ such that $\eps \prod_{k\geq 0}C_k^{2^{-k}}\leq \frac{1}{2}$ (or any positive number $< 1$)  and $\sum_{n\geq 0}\abs{x_{n+1} - x_n}_{s_{n+1}}<\eta$, in order for the whole sequence $(x_k)$ to exist and converge in $B_s(\eta)\subset E_s$. Hence, using the definition of the $C_n$'s and the fact that $$\pa{\frac{C}{2}}^{-2^{-k}}=\pa{\frac{2}{C}}^{\pa{\frac{1}{2}}^k}\Longrightarrow\prod  \pa{\frac{2}{C}}^{\pa{\frac{1}{2}}^k} = \pa{\frac{2}{C}}^{\sum\frac{1}{2^k}}=\pa{\frac{2}{C}}^2,$$ within $\sum_k \frac{1}{2^k} = \sum_k k\frac{1}{2^{k}}=2$, we obtain as a sufficient value
\begin{equation}
\eps = \eta \frac{2}{C^2}\prod_{k\geq 0}\sigma_k^{\tau\,(\frac{1}{2})^k}.
\label{bound epsilon}
\end{equation}
Eventually, the constraint $3\sum_{n\geq 0}\sigma_n = \sigma$ gives $\sigma_k= \frac{\sigma}{6}\,\pa{\frac{1}{2}}^k$, which, plugged into \eqref{bound epsilon}, gives: 
\[\eps = \eta \, \frac{2}{C^2}\pa{\frac{\sigma}{12}}^{2\tau} > \frac{\sigma^{2\tau}\eta}{2^{8^{\tau}}C^2},\]
hence the theorem. 

A posteriori, the exponential decay we proved makes straightforward the further assumption $\abs{x_k - x_{k-1}}_{s_k}<\sigma_k$ to apply lemma \ref{lemma Q}.\\ Concerning the bounds over the constant $C$, as $\sum_k \abs{x_{k+1} - x_k}_{s_{k+1}}\leq\eta$, we see that all the $\abs{x_n}_{s_n}$ are bounded, hence the constants $C'$ and $C''$ depending on them. \\Moreover, to have all the $C_n\geq 1$, as we previously supposed, it suffices to assume $C\geq \sigma/3$.
\end{proof}

\begin{rmk}
In case the operator $\phi$ is defined only on polynomially small balls $$\phi: B^E_{s+\sigma}(c_0 \sigma^{\ell})\to F_s,\, c_0>0, \forall s,\sigma$$ the statement and the proof of theorem \ref{teorema inversa} still hold, provided that $\eta$ is chosen small enough ($\eta< 2c_0(\sigma/12)^{\ell}$ suffices). \\ This is the case of the operators defined in sections \ref{dissipative Herman} and \ref{Russmann section vector fields}, where $\ell=2$. 
\label{remark polynomial}
\end{rmk}
\subsection{Local uniqueness and regularity of the normal form}
We want to show the uniqueness and some regularity properties of the right inverse $\psi$ of $\phi$, assuming the additional left invertibility of $\phi'$ (which is the case, for the particular operator $\phi'$ of interest to us).
\begin{defn}
We will say that a family of norms $(\abs{\cdot}_s)_{s>0}$ on a grading $(E_s)_{s>0}$ is \emph{log-convex} if for every $x\in E_s$ the map $ s\mapsto \log\abs{x}_s $ is convex.
\end{defn}

\begin{lemma} If $(\abs{\,\cdot\,}_s)$ is log-convex, the following inequality holds \[\abs{x}^2_{s+\sigma}\leq \abs{x}_s \abs{x}_{s+\tilde{\sigma}},\quad \forall s,\sigma,\tilde{\sigma}=\sigma(1+\frac{1}{s}).\] 
\end{lemma}

\begin{proof}
If $f: s\mapsto \log\abs{x}_s$ is convex, this inequality holds \[f\pa{\frac{s_1 + s_2}{2}}\leq \frac{f(s_1) + f(s_2)}{2}.\] Let now $x\in E_s$, then
\[\log\abs{x}_{s+\sigma}\leq\log\abs{x}_{\frac{2s + \tilde{\sigma}}{2}}\leq \frac{1}{2}\pa{\log\abs{x}_s + \log\abs{x}_{s+\tilde{\sigma}}}=\frac{1}{2}\log(\abs{x}_s\abs{x}_{s+\tilde{\sigma}}),\]hence the lemma. 
\end{proof}
Let us assume that the family of norms $(\abs{\cdot}_s)_{s>0}$ of the grading $(E_s)_{s>0}$ are log-convex. To prove the uniqueness of $\psi$ we are going to assume that $\phi'$ is also left-invertible. 

\begin{prop}[Lipschitz continuity of $\psi$] 
\label{lipschitz}
Let $\sigma<s$. If $y,\hat{y}\in B^{F}_{s+\sigma}(\eps)$ with $\eps = {3^{-4\tau}2^{-16\tau}}\frac{\sigma^{6\tau}}{4C^3}$, the following inequality holds
\[\abs{\psi(y) - \psi(\hat{y})}_s\leq L \abs{y-\hat{y}}_{x,s+\sigma},\] with $L=2C'/\sigma^{\tau'}$. In particular, $\psi$ being the unique local right inverse of $\phi$, it is also its unique left inverse.
\end{prop}

\begin{proof}
In order to get the wanted estimate we introduce an intermediate  parameter $\xi$, that will be chosen later, such tat $\eta<\xi<\sigma<s<s+\sigma$.\\ 
To lighten notations let us call $\psi(y)=: x$ and $\psi(\hat{y})=: \hat{x}$. Let also $\eps=\frac{\xi^{2\tau}\eta}{2^{8\tau}C^2}$ so that if $y,\hat{y}\in B^{F}_{s+\sigma}(\eps)$, $x,\hat{x}\in B^E_{s+\sigma -\xi}(\eta)$, by theorem \ref{teorema inversa}, provided that $\eta< s + \sigma -\xi$ - to check later. In particular, we assume that any $x,\hat{x}\in B^{E}_{s + \sigma -\xi}$ satisfy $\abs{x - \hat{x}}_{s+\sigma - \xi}\leq 2\eta.$
Writing \[(x-\hat{x})=\phi'^{-1}(x)\cdot\phi(x)(x-\hat{x}),\]
and using $$\phi'(x)(x-\hat{x})= \phi(\hat{x})-\phi(\hat{x})-Q(x,\hat{x}),$$ we get \[ x-\hat{x}= \inderiv{\phi}(x)\pa{\phi(\hat{x}) - \phi(x) - Q(x,\hat{x})}.\]
Taking norms we have
 \begin{align*}
\abs{x-\hat{x}}_s  &\leq \frac{C'}{\sigma^{\tau'}}\abs{y - \hat{y}}_{x,s+\sigma} + \frac{C}{2\xi^{\tau}}\abs{x-\hat{x}}^2_{s+2\xi + \abs{x-\hat{x}}_{s + \xi}},\\
&\leq \frac{C'}{\sigma^{\tau'}}\abs{y - \hat{y}}_{x,s+\sigma} + \frac{C}{2\xi^{\tau}}\abs{x-\hat{x}}^2_{s+2\xi + 2\eta},
\end{align*} 
 by lemma \ref{lemma Q} and the fact that $\abs{x-\hat{x}}_{s+\xi}\leq\abs{x-\hat{x}}_{s+\sigma-\xi}$ (choosing $\xi$ so that $2\xi< \sigma $ too). \\ Let us define $\tilde{\sigma}=(2\xi + 2\eta)(1 + 1/s)$ and use the interpolation inequality \[\abs{x-\hat{x}}^2_{s+2\eta +2\xi}\leq\abs{x-\hat{x}}_s\abs{x-\hat{x}}_{s+\tilde{\sigma}}\] to obtain \[(1 - \frac{C}{2\xi^\tau}\abs{x-\hat{x}}_{s+\tilde{\sigma}})\abs{x-\hat{x}}_s\leq \frac{C'}{\sigma^{\tau'}}\abs{y-\hat{y}}_{x,s+\sigma}.\]
We now choose $\eta$ so small to have
\begin{itemize}[leftmargin=*]
\item $\tilde{\sigma}\leq\sigma - \xi$, which implies $\abs{x-\hat{x}}_{s+\tilde{\sigma}}\leq 2\eta$. It suffices to have $\eta\leq\frac{\sigma}{2(1+\frac{1}{s})} - \frac{3}{2}\xi$.
\item $\eta\leq\frac{\xi^{\tau}}{2C}$ in order to have $\frac{C}{2\xi^\tau}\abs{x-\hat{x}}_{s+\sigma}\leq\frac{1}{2}.$
\end{itemize}
A possible choice is $\xi = \frac{\sigma^2}{12}$ and $\eta = \pa{\frac{\sigma}{12}}^{2\tau}\frac{1}{4C},$ hence our choice of $\eps$.
\end{proof}

\begin{prop}[Smooth differentiation of $\psi$]
\label{smoothness}
 Let $\sigma<s<s+\sigma$ and $\eps$ be as in proposition \ref{lipschitz}. There exists a constant $K$ such that for every $y,\hat{y}\in B_{s+\sigma}^F(\eps)$ we have \[\abs{\psi(\hat{y}) - \psi(y) - \inderiv{\phi}(\psi(y))(\hat{y} - y)}_s\leq K(\sigma) \abs{\hat{y} - y}^2_{x,s+\sigma},\] and the map $\psi': B_{s+\sigma}^F(\eps)\to L(F_{s+\sigma},E_s)$ defined locally by $\psi'(y)=\phi'^{-1}(\psi(y))$  is continuous. In particular $\psi$ has the same degree of smoothness as $\phi$.
\end{prop}
\begin{proof}
 Let's baptize some terms
\begin{itemize}[leftmargin=*]
\item $\Delta := \psi(\hat{y}) - \psi(y) - \inderiv{\phi}(x)(\hat{y} - y) $
\item $\delta:= \hat{y} - y$, the increment
\item $\xi:= \psi(y + \delta) - \psi(y)$
\item $\Xi:= \phi(x + \xi) - \phi(x)$.
\end{itemize}
With these new notations we can see $\Delta$ as 
\begin{align*}
\Delta &= \xi - \inderiv{\phi}(x)\cdot\Xi\\
&=\inderiv{\phi}(x)(\phi'(x)\cdot\xi - \Xi)\\
& =\inderiv{\phi}(x)(\phi'(x)\xi - \phi(x+\xi) + \phi(x))\\
& = -\inderiv{\phi}(x)Q(x,x+\xi)
\end{align*}
Taking norms we have
\[\abs{\Delta}_s\leq K \abs{\hat{y} - y}^2_{x,s+\bar{\sigma}}\] by proposition \ref{lipschitz} and lemma \ref{lemma Q}, for some $\bar\sigma$ which goes to zero when $\sigma$ does, and some constant $K>0$ depending on $\sigma$ . Up to substituting $\sigma$ for $\bar{\sigma}$, we have proved the statement.\\
In addition \[\psi'(y)= \phi^{-1}(y)'= \phi'^{-1}\circ \phi^{-1}(y)= \phi'^{-1}(\psi(y)),\] the inversion of linear operators between Banach spaces being analytic, the map $y\mapsto \phi'^{-1}(\psi(y))$ has the same degree of smoothness as $\phi'$.
\end{proof}

It is sometimes convenient to extend $\psi$ to non-Diophantine characteristic frequencies $(\alpha,a)$. Whitney smoothness guarantees that such an extension exists.
Let suppose that $\phi(x)=\phi_{\nu}(x)$ depends on some parameter $\nu\in B^k$ (the unit ball of $\R^k$) and that it is $C^1$ with respect to $\nu$ and that estimates on $\phi'^{-1}_{\nu}$ and $\phi_{\nu}''$ are uniform with respect to $\nu$ over some closed subset $D$ of $\R^{k}$. 

\begin{prop}[Whitney differentiability]Let us fix $\eps, \sigma, s$ as in proposition \ref{lipschitz}. The map $\psi : D\times B^{F}_{s+\sigma}(\eps)\to B^{E}_s(\eta)$ is $C^1$-Whitney differentiable and extends to a map $\psi : \R^{2n}\times B^{F}_{s+\sigma}(\eps)\to B^{E}_s(\eta) $ of class $C^1$. If $\phi$ is $C^k$, $1\leq k\leq \infty$, with respect to $\nu$, this extension is $C^k$.
\label{Whitney}
\end{prop}

\begin{proof}
Let $y\in B^{F}_{s+\sigma}(\eps)$. For $\nu,\nu + \mu\in D$, let $x_\nu = \psi_{\nu}(y)$ and $x_{\nu + \mu}=\psi_{\nu+\mu}(y)$, implying \[\phi_{\nu+\mu}(x_{\nu+\mu}) - \phi_{\nu+\mu}(x_{\nu}) = \phi_{\nu}(x_\nu) - \phi_{\nu+\mu}(x_{\nu}).\] It then follows, since $y\mapsto\psi_{\nu+\mu}(y)$ is Lipschitz, that \[\abs{x_{\nu+\mu} - x_{\nu}}_s\leq L \abs{\phi_{\nu}(x_{\nu}) - \phi_{\nu+\mu}(x_{\nu})}_{x_{\nu},s+\sigma},\] taking ${y}=\phi_{\nu +\mu}(x_{\nu}),\hat{y}=\phi_{\nu+\mu}(x_{\nu+\mu}).$ In particular since $\nu\mapsto\phi_{\nu}(x_{\nu})$ is Lipschitz, the same is for $\nu\mapsto x_{\nu}.$
Let us now expand $\phi_{\nu+\mu}(x_{\nu+\mu})=\phi(\nu+\mu, x_{\nu+\mu})$ in Taylor at $(\nu,x_{\nu}).$ We have 
\[\phi(\nu+\mu,x_{\nu+\mu}) = \phi(\nu,x_{\nu}) + D\phi(\nu,x_{\nu})\cdot (\mu, x_{\nu+\mu} - x_{\nu}) + O(\mu^2, \abs{x_{\nu+\mu} - x_{\nu}}_s^2),\] hence formally defining the derivative $\partial_{\nu} x_{\nu} := - \phi'^{-1}_{\nu}(x_{\nu})\cdot\partial_{\nu} \phi_{\nu}(x_{\nu}),$ we obtain \[x_{\nu+\mu} - x_{\nu} - \partial_{\nu} x_{\nu}\cdot\mu = \phi'^{-1}_{\nu}(x_{\nu})\cdot O(\mu^2),\] hence \[\abs{x_{\nu+\mu} - x_{\nu} - \partial_{\nu} x_{\nu}\cdot\mu}_s = O(\mu^2)\] by Lipschitz property of $\nu\mapsto x_{\nu},$ when $\mu\mapsto 0$, locally uniformly with respect to $\nu$. Hence $\nu\mapsto x_{\nu}$ is $C^1$-Whitney-smooth and by Whitney extension theorem, the claimed extension exists. Similarily if $\phi$ is $C^k$ with respect to $\nu$, $\nu\mapsto x_{\nu}$ is $C^k$-Whitney-smooth. Ssee \cite{Abraham-Robbin:1967} for the straightforward generalization of Whitney's theorem to the case of interest to us: $\psi$ takes values in a Banach space instead of a finite dimension vector space; but note that the extension direction is of finite dimension though.
\end{proof}

\section{Inversion of a holomorphism of $ \T^n_s $ }
We present here a classical result on the inversion of holorphisms on the complex torus $\T^n_s$ that intervened to guarantee the well definition of normal form operators $\phi$.

All complex extensions of manifolds are defined at the help of the $\ell^\infty$-norm, \[\T^n_s = \set{\teta\in\T^n_\C : \abs{\teta}:= \max_{1\leq j\leq n}\abs{\Im{\teta_j}}\leq s}.\]
Let also define $\R^n_s := \R^n \times (-s,s)$ and consider the universal covering of $\T^n_s$, $p: \R^n_s\to \T^n_s$.
\begin{thm}
\label{theorem well def} Let $v:\T^n_s\to\C^n$ be a vector field such that $\abs{v}_s < \sigma/n$. The map $\id + v:\T^n_{s-\sigma}\to\R^n_{s}$ induces a map $\varphi = \id + v : \T^n_{s-\sigma}\to\T^n_s$ which is a biholomorphism and there is a unique biholomorphism $\psi : \T^n_{s-2\sigma}\to\T^n_{s-\sigma}$ such that $\varphi\circ\psi = \id_{\T^n_{s - 2\sigma}}.$ \\
In particular the following hold: \[\abs{\psi - \id}_{s-2\sigma}\leq \abs{v}_{s-\sigma}\] and, if $\abs{v}_s<\sigma/2n$ \[\abs{\psi' - \id}_{s-2\sigma}\leq \frac{2}{\sigma}\abs{v}_{s}.\]
\end{thm}

\begin{proof}
Let $\hat\varphi:= \id + v\circ p : \R^n_s\to\R^n_{s+\sigma}$ be the lift of $\varphi$ to $\R^n_s$.\\Let's start proving the injectivity and surjectivity of $\hat\varphi$; the same properties for $\varphi$ descend from these.
\begin{itemize}[leftmargin=*]
\item $\hat\varphi$ is injective as a map from $\R^n_{s-\sigma}\to\R^n_s$.\\Let $\hat\varphi(x)=\hat\varphi(x')$, from the definition of $\hat\varphi$ we have 
\begin{align*}
\abs{x -x'}=\abs{v\circ p(x') - v\circ p(x)} &\leq \int_0^1 \sum_{k=1}^n\abs{\partial_{x_k} \hat{v}}_{s-\sigma}\abs{x_k'-x_k}\,dt\leq \frac{n}{\sigma}\abs{v}_s\abs{x-x'}\\ & < \abs{x - x'},
\end{align*}
hence $x'=x$.
\item $\hat\varphi: \R^n_{s-\sigma}\to \R^n_{s-2\sigma}\subset \hat\varphi (\R^n_{s-\sigma})$ is surjective.\\ Define, for every $y\in\R^n_{s-2\sigma}$ the map \[f: \R^n_{s-\sigma}\to\R^n_{s-\sigma},\, x\mapsto y - v\circ p(x),\] which is a contraction (see the last but one inequality of the previous step). Hence there exists a unique fixed point such that $\hat{\varphi}(x) = x + v\circ p(x) = y.$ 
\end{itemize}
For every $k\in 2\pi\Z^n$, the function $\R^n_s\to\R^n_s,\, x\mapsto \hat\varphi(x+k) - \hat\varphi(x)$ is continuous and $2\pi\Z^n$-valued. In particular there exists $A\in\GL_n(\Z)$ such that $\hat\varphi(x + k)= \hat\varphi(x) + Ak$.
\begin{itemize}[leftmargin=*]
\item $\varphi: \T^n_{s-\sigma}\to\T^n_s$ is injective.\\ 
Let $\varphi(p(x))=\varphi(p(x')),$ with $p(x),p(x')\in\T^n_{s-\sigma}$, hence $\hat\varphi(x')=\hat\varphi(x) + k' ,$ for some $k'\in 2\pi\Z^n$. Hence $\hat\varphi(x' - A^{-1}k')=\hat\varphi(x)$, and for the injectivity of $\hat\varphi$, $p(x)=p(x')$. In particular $\varphi$ is biholomorphic:
\begin{lemma}[\cite{Fritzsche-Grauert}] If $G\subset\C^n$ is a domain and $f: G\to\C^n$ injective and holomorphic, then $f(G)$ is a domain and $f: G\to f(G)$ is biholomorphic.
\end{lemma}
\item That $\varphi: \T^n_{s-\sigma}\to \T^n_{s-2\sigma}\subset \varphi(\T^n_{s-\sigma})$ is surjective follows from the one of $\hat\varphi$.
\item Estimate for $\psi: \T^n_{s-2\sigma}\to \T^n_{s-\sigma}$ the inverse of $\varphi$.\\ Let $\hat\psi : \R^n_{s-2\sigma}\to\R^n_{s-\sigma}$ be the inverse of $\hat\varphi$, and $y\in\R^n_{s-2\sigma}.$ From the definition of $\hat\varphi$, $v\circ p(\hat\psi(y))= y - p(\hat\psi(y)) = y - \hat\psi(y)$. Hence \[\abs{\hat\psi(y) - y}_{s-2\sigma} = \abs{v\circ p (\hat \psi (y))}_{s-2\sigma}\leq \abs{v}_{s-2\sigma}\leq\abs{v}_{s-\sigma}.\]
\item Estimate for $\psi' = \varphi'^{-1}\circ\varphi^{-1}$. We have \[\abs{\psi' - \id}_{s-2\sigma} \leq \abs{\varphi'^{-1} - \id}_{s-\sigma}\leq \frac{\abs{\varphi' - \id}_{s-\sigma}}{1 - \abs{\varphi' - \id}_{s-\sigma}}\leq \frac{2n}{2n -1}\frac{\abs{v}_s}{\sigma} \leq 2\frac{\abs{v}_s}{\sigma},\] by triangular and Cauchy inequalities.
\end{itemize}
\end{proof}

\begin{cor}[Well definition of the operators $\phi$]
\label{cor well def}For all $s,\sigma$
\begin{itemize}[leftmargin=*]
\item  if $g\in\G^{\sigma/n}_{s + \sigma}$, then $g^{-1}\in\A(\normal{T}^n_s,\normal{T}^n_{s+\sigma})$
\item if $g\in\G^{\omega,\sigma^2/2n}_{s + \sigma}$, then $g^{-1}\in\A(\normal{T}^n_s,\normal{T}^n_{s+\sigma})$.
\end{itemize}
As a consequence, the operators $\phi$ in \eqref{normal operator}, \eqref{operator Herman} and \eqref{normal Russmann} are well defined.
\end{cor}
\begin{proof}
We recall the form of $g\in\G^{\sigma/n}_{s+\sigma}$: \[g(\teta,r)= (\varphi(\teta), R_0(\teta) + R_1(\teta)\cdot r).\]
$g^{-1}$ reads \[g^{-1}(\teta,r)=(\phi^{-1}(\teta), R_1^{-1}\circ\varphi^{-1}(\teta)\cdot(r - R_0\circ\varphi(\teta))).\]
Up to rescaling norms by a factor $1/2$ like $\norm{x}_s := \frac{1}{2}\abs{x}$, the first statement is straightforward from theorem \ref{theorem well def}. By abuse of notations, we keep on indicating $\norm{x}_s$ with $\abs{x}_s$.\medskip\\
Concerning those $g\in\G^{\omega,\sigma^2/2n}_{s+\sigma}$ we recall that $g^{-1}$ is given by
\[g^{-1}(\teta,r)=(\varphi'^{-1}(\teta), \,^t\varphi'\circ\varphi^{-1}(\teta)\cdot r - \rho\circ\varphi^{-1}(\teta));\]
if $\abs{\varphi^{-1} - \id}_{s}< \sigma$ and $\abs{\rho}_{s+\sigma}<\sigma/2$ with $\abs{r\cdot\varphi'\circ\varphi^{-1}(\teta)}_s < \sigma/2$ we get the wanted thesis. Just note that 
 \[\abs{\,^t(\varphi' - \id)\cdot r}_s \leq \frac{n\abs{r}_s}{\sigma}\abs{\varphi - \id}_{s+\sigma}\leq \sigma/2,\]
 the factor $n$ coming from the transposition. 
\end{proof}
\section{Estimates on the Lie brackets of vector fields}
This is just an adaptation to vector fields on $\normal{T}^n_{s+\sigma}$ of the analogous lemma for vector fields on the torus $\T^n_s$ in \cite{Poschel:2011}.
\begin{lemma}
Let $f$ and $g$ be two real analytic vector fields on $\normal{T}^n_{s+\sigma}$. The following inequality holds
\[\abs{\,\sq{f,g}\,}_s\leq \frac{2}{\sigma}\pa{1+\frac{1}{e}}\abs{f}_{s+\sigma}\abs{ g}_{s+\sigma}.\]
\label{lemma lie brackets}
\end{lemma}
\begin{proof}
Consider $f=(f^\teta,f^r) = \sum_{j=1}^n f^{\teta_j}\base{\teta_j} + f^{r_j}\base{r_j}$ and $g=(g^\teta,g^r) = \sum_{j=1}^n g^{\teta_j}\base{\teta_j} + g^{r_j}\base{r_j}$.
From the definition of the Lie Brackets we have $[{f},{g}] = \sum_k f(g^k) - {g}(f^k),$ where every component $k$ reads
\begin{align*}
[f,g]^k &= \sum_{j=1}^n (f^{\teta_j}\frac{\partial g^k}{\partial\teta_j} + f^{r_j}\frac{\partial g^k}{\partial r_j}) - (g^{\teta_j}\frac{\partial f^k}{\partial \teta_j} + g^{r_j}\frac{\partial f^k}{\partial r_j})\\
        &= (Dg\cdot f - Df\cdot g)^k. 
\end{align*}
We observe that for an holomorphic function $h: \normal{T}^n_{s+\sigma}\to \C$, one has 
 \[\abs{\frac{\partial h}{\partial r_j}}_s = \sum_k \abs{\frac{\partial h _k(r)}{\partial r_j}}_s e^{\abs{k} s}\leq \sum_k \frac{1}{\sigma}{\abs{h_k(r)}_{s+\sigma}} e^{\abs{k}s}\leq  \frac{1}{\sigma}\abs{h}_{s+\sigma},\]
and 
\begin{align*}
\abs{\frac{\partial h}{\partial \teta_j}}_s &= \sum_k \abs{k_j}\abs{h_k(r)}_s e^{\abs{k}s}\leq \sum_k \abs{k} \abs{h_k(r)}_s e^{\abs{k} (s+\sigma)}\, e^{-\abs{k}\sigma}\\
& \leq \frac{1}{e\sigma}\sum_k \abs{h_k(r)}_{s+\sigma}e^{\abs{k}(s+\sigma)} = \frac{1}{e\sigma}\abs{h}_{s+\sigma},
\end{align*}
where we bound $\abs{k}e^{-\abs{k}\sigma}$ with the maximum attained by $x e^{-x\sigma},\, x > 0$, in $1/\sigma$, that is $1/e\sigma$. \\ Therefore, consider $f$ and $g$ in their Fourier's expansion, $Dg\cdot f$ read
\[Dg \cdot f = \sum_{k,\ell}\,i {k\cdot f^\theta_\ell}g_k e^{i (k+\ell)\teta} + D_r g_k\cdot f^{r}_\ell\, e^{i (k + \ell)\cdot\teta} = \sum_{k,\ell} i\,k\cdot f^{\teta}_{\ell - k} g_k \,e^{i\ell\cdot\teta}+ D_r g_k\cdot f^{r}_{\ell-k}e^{i\ell\cdot\teta}. \]
Passing to norms we have the following inequality
\[\abs{Dg\cdot f}_s\leq \sum_{k,\ell} \abs{k}\abs{f^{\teta}_{\ell - k}}\abs{g_k} e^{\abs{k}s}e^{\abs{\ell-k}s} + \abs{D_r g_k}\abs{f^{r}_{\ell - k}} e^{\abs{k}s}e^{\abs{\ell -k }s}\leq\]
\begin{align*}
&\leq \sum_{k,\ell}\abs{k}\abs{g_k} e^{-\abs{k}\sigma} e^{\abs{k}(s+\sigma) }\abs{f^{\teta}_{\ell - k}}  e^{\abs{\ell-k}s} + \abs{D_r g_k}e^{\abs{k}s}\,\abs{f^r_{\ell - k}} e^{\abs{\ell - k}s}\\
&\leq \frac{1}{e\sigma}\abs{g}_{s+\sigma}\abs{f}_{s+\sigma} + \frac{1}{\sigma}\abs{g}_{s+\sigma}\abs{f}_{s+\sigma},
\end{align*}
which follows from the previous remark. Hence the lemma.
\end{proof}

\comment{\subsection*{Acknowledgments} \co{This work would have never seen the light without the mathematical (and moral) support and advises of A. Chenciner and J. Féjoz all the way through my Ph.D. I'm grateful and in debt with both of them. Thank you to T. Castan, B. Fayad, J.-P. Marco, P. Mastrolia and L. Niederman for the enlightening discussions we had and their constant interest in this (and future) work.} }

%%%%%%%% BIBLIOGRAPHY %%%%%%%%%%%%%%%%%%%%

\comment{\bibliographystyle{plain}
\bibliography{finalbiblio}
\nocite{Herman:1983}
\nocite{EFK:2013}
\nocite{EFK:2015}
\nocite{Fayad-Krikorian:2009}
\nocite{Herman:notes}}

\end{document}